\pdfoutput=1
\RequirePackage{ifpdf}
\ifpdf 
\documentclass[pdftex]{sigma}
\else
\documentclass{sigma}
\fi

\numberwithin{equation}{section}

\newtheorem{teo}{Theorem}[section]
\newtheorem*{Theorem*}{Theorem}
\newtheorem{cor}[teo]{Corollary}
\newtheorem{lem}[teo]{Lemma}
\newtheorem{prop}[teo]{Proposition}
\newtheorem{prob}[teo]{Problem}

{ \theoremstyle{definition}
	\newtheorem{defin}[teo]{Definition}
	
	\newtheorem{ese}[teo]{Example}
	\newtheorem{oss}[teo]{Remark} }

\usepackage{amscd}
\usepackage{braket}
\usepackage{tikz-cd}
\usepackage{wasysym}
\usepackage{enumitem}
\usepackage{multirow}
\usetikzlibrary{arrows}

\DeclareFontFamily{U}{wncy}{}
\DeclareFontShape{U}{wncy}{m}{n}{<->wncyr10}{}
\DeclareSymbolFont{mcy}{U}{wncy}{m}{n}
\DeclareMathSymbol{\sha}{\mathord}{mcy}{"58}

\newcommand{\N}{\mathbb{N}}
\newcommand{\Z}{\mathbb{Z}}
\newcommand{\C}{\mathbb{C}}
\renewcommand{\S}{\mathbb{S}}
\renewcommand{\i}{\mathrm{i}}
\renewcommand{\d}{\mathrm{d}}

\newcommand{\Y}{\mathrm{Y}}
\newcommand{\g}{\mathfrak{g}}
\newcommand{\zem}{\mathfrak{Z}}

\newcommand{\cB}{\mathcal{B}}
\newcommand{\Sel}{\mathrm{Sel}}
\newcommand{\Fp}{\mathrm{F}}
\newcommand{\rFp}{\tilde{\mathrm{F}}}

\newcommand{\f}{\mathsf{f}}

\begin{document}
\allowdisplaybreaks

\newcommand{\arXivNumber}{1911.11040}

\renewcommand{\PaperNumber}{018}

\FirstPageHeading

\ShortArticleName{Algebras of Non-Local Screenings and Diagonal Nichols Algebras}

\ArticleName{Algebras of Non-Local Screenings\\ and Diagonal Nichols Algebras}

\Author{Ilaria FLANDOLI and Simon D.~LENTNER}
\AuthorNameForHeading{I.~Flandoli and S.D.~Lentner}
\Address{Department of Mathematics, University of Hamburg,\\
Bundesstra{\ss}e 55, 20146 Hamburg, Germany}
\Email{\href{mailto:ilaria.flandoli@gmail.com}{ilaria.flandoli@gmail.com}, \href{mailto:simon.lentner@uni-hamburg.de}{simon.lentner@uni-hamburg.de}}
\URLaddress{\url{http://simon.lentner.net}}

\ArticleDates{Received November 02, 2020, in final form February 14, 2022; Published online March 11, 2022}

\Abstract{In a vertex algebra setting, we consider non-local screening operators associated to the basis of any non-integral lattice. We have previously shown that, under certain restrictions, these screening operators satisfy the relations of a quantum shuffle algebra or Nichols algebra associated to a diagonal braiding, which encodes the non-locality and non-integrality. In the present article, we take all finite-dimensional diagonal Nichols algebras, as classified by Heckenberger, and find all lattice realizations of the braiding that are compatible with reflections. Usually, the realizations are unique or come as one- or two-parameter families. Examples include realizations of Lie superalgebras.	We then study the associated algebra of screenings with improved methods. Typically, for positive definite lattices we obtain the Nichols algebra, such as the positive part of the quantum group, and for negative definite lattices we obtain a certain extension of the Nichols algebra generalizing the infinite quantum group with a large center.}


\Classification{16T05; 17B69}
	
{\small \tableofcontents}

\section{Introduction}
Let $\Lambda$ be a lattice with basis $a_1,\dots,a_r$ and inner product $(-,-)$, not necessarily integral. We consider the Gram matrix $m_{ij}=(a_i,a_j)$ and a braiding matrix $q_{ij}={\rm e}^{\pi {\rm i} (a_i,a_j)}$. In the theory of Nichols algebras, we can associate to the data $(q_{ij})$ the Nichols algebra $\mathcal{B}(q_{ij})$ of diagonal type. In the theory of vertex algebras, we can associate to $\Lambda,(m_{ij})$ an abelian intertwining algebra and to each $a_i$ a non-local screening operator $\zem_{i}$ on the Heisenberg vertex algebra $\mathcal{H}^r$.

 It was proven by the second author in \cite{Len17} that the screening operators $\zem_{i}$ obey the relations of the Nichols algebra $\mathcal{B}(q_{ij})$, provided that $(m_{ij})$ is \emph{subpolar} (Definition~\ref{def_smallnessF}). This condition imposes a certain lower bound for the sums of~$m_{ij}$ over any subset of indices and it ensures convergence of integrals such as
 \[\int_0^1\cdots\int_0^1 \prod_{i<j} (z_i-z_j)^{m_{ij}}\mathrm{d} z_1\cdots\mathrm{d}z_n.\]

The goal of this article is to find all lattices $\Lambda,(m_{ij})$, such that the associated diagonal braiding $(q_{ij})$ gives a \emph{finite-dimensional} Nichols algebra in the classification of \cite{Heck06}, and such that the Weyl reflections on $(q_{ij})$ in the theory of Nichols algebras lift to reflections on $(m_{ij})$ in a suitable sense (Definition \ref{Mm}). In this case we say that $\Lambda$, $(m_{ij})$ \emph{realizes} the braiding matrix~$(q_{ij})$. This provides an interesting zoo of examples to extend the screening method and all its related questions, such as the \emph{logarithmic Kazhdan Lusztig conjecture} to cases beyond quantum groups. This main motivation for our article will be discussed further below.

As a second goal, for each realizing lattice $\Lambda$ we continue to study the algebra of screening operators. First we determine which data $(m_{ij})$ are subpolar, in which case the screening algebra is a surjective image of the Nichols algebra. We find that all Nichols algebras that do not come as $q$-parametrized families have at most one realizing lattice, and this realization is subpolar. For the realizing lattices that come in $q$-parametrized families,\footnote{Cartan type, super Lie type and the color Lie algebra \cite[Table~1, row~6]{Hecklist}. Some come with more than one parameter~$q$, $m$. The families \cite{Heck06} rows~5 and~6 have no realizations.} we encounter $m$-parametrized families of realizations with $q={\rm e}^{\pi\i m}$ and outside a certain range of parameters~$m$ the subpolarity of~$(m_{ij})$ fails. In these situations we study the products of screening operators and their relations using analytic continuation. Since we have not yet found a general approach for this, we explicitly treat the most common relations: the commutation relations $[x_i,x_j]_q$, the truncation relations of simple root vectors $x_i^n=0$, and the Serre relations $[x_i,[x_i,x_j]_q]_q$ for a~Cartan matrix entry $c_{ij}=-1$, where $[x_i,x_j]_q$ is the braided commutator defined in formula~\eqref{formula_qcommutator}. As it turns out, each of these relations can be analytically continued and holds outside an explicit set of poles.

The method of analytic continuation we employ to this end are on the one hand explicit expressions in terms of Gamma functions due to the Selberg integral formula and to the $A_n$-Selberg integral formula~\cite{TV03,War09}, and on the other hand a new less explicit analytic continuation for general Selberg integrals using recursion in Section~\ref{sec_recursion}. We believe that the latter method can be greatly improved and should settle a large class of relations. Moreover we believe that our results suggest an explicit $\g$-Selberg integral formula for all generalized root systems (e.g., for Lie superalgebras), but both questions are beyond the scope of the present article. From a~Nichols algebra perspective it is interesting, that we require the knowledge which \emph{order} the zero in the quantum symmetrizer of a relation has. For example, the quantum Serre relation for a~simply-laced quantum group at $q^2\neq -1$ exhibits a~simple zero. At $q^2=-1$ it exhibits a~double zero and at the same time the Selberg integral exhibits for $m\in -\frac{1}{2}-\N_0$ a~simple pole, thus the quantum Serre relation holds even in these cases.

As a main example, let $\mathfrak{g}$ be a complex finite-dimensional semisimple Lie algebra of rank $r$ with simple roots $\alpha_1, \dots, \alpha_r$ and Killing form normalized to $(\alpha_i,\alpha_i)\in\{2,2d\}$ with lacity $d\in\{1,2,3\}$. Let $q\neq \pm1$ be a root of unity and consider the braiding $q_{ij} = q^{(\alpha_i, \alpha_j)}$, whose associated Nichols algebra $\mathcal{B}(q_{ij})$ is the positive part of the small quantum group $u_q(\mathfrak{g})^+$ \cite{AS10, Lusz93,Rosso98}. For every rational number $m\not\in \Z$ with $q= {\rm e}^{{\rm i} \pi m}$ we obtain a realizing lattice by taking the root lattice of~$\mathfrak{g}$ rescaled by~$m$.

We find that the rescaled root lattices are the only realizations $(m_{ij})$ of $(q_{ij})$ if $\operatorname{ord}\big(q^2\big)>d+1$ and $\operatorname{ord}\big(q^{2d}\big)>2$. Otherwise, there exist additional families of realizations $(m_{ij})$, mostly associated to Lie superalgebras, which incidentally have for these small orders of $q$ the same braiding matrices, but different realizing lattices. For example, for $\g$ of type $A_{n}$ at $q^2=-1$ we have additional realizations associated to $A(n'\vert n'')$, $n'+n''=n-1$ that contain root lattices of $A(n')$, $A(n'')$ rescaled by $m'$, $m''$ with $m'+m''=1$ and a fermionic root of length~$1$. Since the screening operators and their relations depend on $(m_{ij})$, not just on $(q_{ij})$, these realizations behave very different, more similar to Lie superalgebras.

Assume now that $(m_{ij})$ is the realization obtained by rescaling the root lattice of $\g$ by $m$, assume further that $\g$ is simply-laced and $q^2\neq -1$, then we can summarize our analytic continuation results on the relations of the screening operators as follows:
\begin{itemize}\itemsep=0pt
		\item For $0<m<1$ the parameters $(m_{ij})$ are subpolar and thus all Nichols algebra relations hold. Differently spoken, the algebra of screenings is a surjective image of the Borel part of the small quantum group $u_q(\g)$.
		\item For $m<0$ the Serre relations hold. The truncation relations of simple root vectors fail. Differently spoken, for $q^2\neq 1$ the algebra of screenings is a surjective image of the Borel part of the Kac--DeConcini--Procesi quantum group $U_q^\mathcal{K}(\g)$.
		\item for $m>1$ the Serre relations and the truncation relations of simple root vectors hold.
	\end{itemize}
We conjecture that for $m>1$ also the truncation relations of non-simple root vectors hold, so that also in this case we get the Borel part of $u_q(\g)$. 	
	For $q^2=-1$ there is an additional relation $[x_j,[x_i,[x_j,x_k]_q]_q]_q=0$ in \cite{AA17}, which we conjecture to hold for $m<0$, but we have no guess whether it holds for $m>1$.	We conjecture that all surjections above are in fact isomorphisms by the universal property of the Nichols algebra, since the screening operators have a Leibniz-type product rule.
	
The similar statement for simply-laced quantum supergroups is more complicated and depends on $m'$, $m''$ and their relation in the specific case, see Corollary \ref{cor_superNA}. We encounter again the small quantum super group and the Kac--DeConcini--Procesi quantum super group. We also encounter versions of the latter where only the truncation relations for simple root vectors rescaled by $m'$ or $m''$ fail, for example in $A(n',n'')$ those in $A(n')$ or $A(n'')$, which are interesting intermediate Hopf algebras. The truncation relations for the fermionic simple root $x_\f^2=0$ always holds as expected. The additional relation has ranges for $m$ in which we can prove it, but in particular for large $m'$, $m''$ we have no assertion on the additional relation or the Serre relation.

We now discuss the background and application on the vertex algebra side in more detail: A vertex algebra $\mathcal{V}$ \cite{FBZ04, Kac98} is, very roughly speaking, an algebra whose multiplication depends analytically on a complex variable $z$.
Screening operators $\zem_a$ are linear endomorphisms on $\mathcal{V}$ constructed by taking the left-multiplication with a fixed element $a\in\mathcal{V}$ and integrating~$z$ over a circle around the singularity at $z=0$. Using the operator product expansion one can compute the commutator of such operators, and in this way screening operators are important sources of actions of Lie algebras on vertex algebras.

Vertex algebras have a natural notion of representations, and under certain fi\-ni\-te\-ness-as\-sump\-tions the category of representations becomes a braided tensor category \cite{HLZ06}. The tensor product is defined by a universal property involving so-called intertwining operators and the braiding comes from analytically continuing these from $z$ to $-z$ around their singularity at $z=0$. In particular, nontrivial double braiding means the analytical continuation is multivalued.

We can now attempt to define screening operators $\zem_a$ for elements $a$ which are not in the vertex algebra $\mathcal{V}$ but in some vertex algebra module $\mathcal{M}$, using intertwining operators instead of the vertex algebra's multiplication. We call these \emph{non-local screening operators}, because we are now dealing with integrals over multivalued functions and we encounter nontrivial double braidings. Such screening operators are less well-behaved, e.g., with respect to the Virasoro action, and they do not form a Lie algebra. The main result of \cite{Len17} is that instead, under the assumption of subpolarity mentioned above, any fixed set of non-local screenings on the Heisenberg vertex algebra and lattice vertex algebras (and conjecturally on every suitable vertex algebra) generates the Nichols algebra associated to the braiding on $\mathcal{M}$. Instances of non-local screening operators have appeared in the literature for a~long time (see, e.g., \cite{FF84,Fel89}). In the setting discussed next, it was conjectured that they generate the Borel parts of quantum groups, which is now settled by the results in~\cite{Len17}. Semikhatov and Tipunin proposed in \cite{Sem14, ST13} to extend this program to Nichols algebras of diagonal type. Our article builds on these ideas.

One major intention behind studying these non-local screening operators was and is the \emph{logarithmic Kazhdan Lusztig conjecture} \cite{AM08,FT10,FGST06a}, which roughly states the following: Fix a~semisimple complex finite-dimensional Lie algebra~$\g$ and consider the vertex algebra $\mathcal{V}_\Lambda$ associated to the rescaled root lattice $\Lambda=\sqrt{p}\Lambda_\g$, whose braided tensor category of representations is the category of vector spaces graded by $\Lambda^*/\Lambda$. Consider the non-local screening operators~$\zem_{a_i}$ associated to inversely rescaled coroots $a_i=\alpha_i^\vee/\sqrt{p}$. Then the kernel of all $\zem_{a_i}$ defines a~vertex subalgebra $\mathcal{W}\subset\mathcal{V}$, whose category of representations is conjecturally a non-semisimple modular tensor category equivalent to the category of representations of the small quantum group $u_q(\g)$, $q={\rm e}^{\frac{2\pi {\rm i}}{2p}}$, more precisely to some quasi-Hopf algebra variant. In the smallest case $\g=\mathfrak{sl}_2$ the conjecture was solved affirmatively, after about 20 years of research by several groups \cite{AM08,CLR21, CGR20,FGST06a, GN21, NT11, TW13}. For quantum groups associated to arbitrary~$\g$ see \cite{AM14,FT10,GLO18,Len17}.

In a more general setting and in view of the mentioned results on Nichols algebras, the second author has proposed the following problem, which is probably very hard:
\begin{prob} Let $\mathcal{V}$ be a vertex algebra and $\zem_{a_1},\dots,\zem_{a_r}$ a set of non-local screenings generating a Nichols algebra $\mathfrak{B}$ $($or some extension, due to poles$)$. What is the category of representations of the kernel of screenings $\mathcal{W}\subset\mathcal{V}$?
\end{prob}

We expect that the finite-dimensionality of $\mathfrak{B}$ implies finiteness conditions on $\mathcal{W}$ and its category of representations. From a physics intuition, the problem can be interpreted as asking for the representation theory of an orbifold of $\mathcal{V}$ by an action of the Nichols algebra $\mathfrak{B}$ or some extension. A generic guess for a braided tensor category would be the relative Drinfeld center of the representations of the algebra $\mathfrak{B}$ inside $\operatorname{Rep}(\mathcal{V})$, relative to $\operatorname{Rep}(\mathcal{V})$. In the original setting of the logarithmic Kazhdan Lusztig conjecture, this correctly reproduces the expected quasi-Hopf algebra variant of the small quantum group \cite{CGR20,GLO18, Ne21}. However, already for the so-called $p,p'$-models, which would correspond to $\mathfrak{sl}_2$, $q={\rm e}^{\pi {\rm i}\frac{p'}{p}}$, the result is slightly larger, see, e.g.,~\cite{GRW09}. A~second question is how general are the modular tensor categories obtained by this procedure:

\begin{prob}
Classifying semisimple modular tensor categories is hard, let alone non-semi\-simple modular tensor categories. Can we classify modular tensor categories $\mathcal{U}$ whose semisimple part $($roughly speaking$)$ is a fixed semisimple modular tensor category~$\mathcal{C}$? Again, the generic choice for such $\mathcal{U}$ is a relative Drinfeld center of representations of some Nichols algebra $\mathfrak{B}\in\mathcal{C}$.
\end{prob}

A very small step towards such a classification for $u_q(\mathfrak{sl}_2)$ was recently obtained in work of the second author~\cite{CLR21}, which finally settled the smallest case of the logarithmic Kazhdan Lusztig conjecture. The overall idea behind this second proposed problem is a categorical version of the Andruskiewitsch--Schneider program~\cite{AS10}, which aims to classify finite-dimensional Hopf algebras for a given semisimple coradical, and additionally to ask for the existence of a nondegenerate braiding. For example, the main result in cit.~loc.~is a classification of all finite-dimensional Hopf algebras with coradical a finite abelian group algebra $\C[A]$ (of order excluding certain small prime divisors) in terms of one ore more Nichols algebras $u_q(\mathfrak{g})^+$ in the modular tensor category of $A$-Yetter--Drinfeld modules, plus so-called lifting data. In the present context, we would equivalently speak about the modular tensor category $\mathcal{C}=\operatorname{Vect}_{A}$ and different choices of quadratic forms on $A$, and we want to ask further which of the lifting data admits a nondegenerate braiding. A probable answer would be, that only the quantum doubles $U(\chi)$ of Nichols algebras of diagonal type defined in~\cite{Heck07} (respectively again quasi-Hopf algebra versions, depending on the quadratic form) give rise to such modular tensor categories with semisimple part $\mathcal{C}$. Of course there might be further examples not coming from quasi-Hopf algebras.

The intend behind the present article is to conduct a sweep for all possible examples, in which $\mathcal{V}$ is still the lattice vertex algebra and the Nichols algebra involved is not necessarily associated to a quantum group, but finite-dimensional. Thereby we encounter most (but not all) finite-dimensional Nichols algebras of diagonal type. Our guess would be that the kernel of screenings has the same modular tensor category of representations as the associated~$U(\chi)$. Where we have parametrized families of solutions (in particular in Lie and super-Lie type) we have ranges of parameters in which some or all truncation relations in the Nichols algebras do not hold. One would expect the associated $\mathcal{W}$ to have a representation category related to the relative Drinfeld center of this algebra, or, more algebraically spoken, to the category of representations of a mixed quantum group, see, e.g., \cite{G21, GG17} and \cite{Ang16, L19}. A~further step would be to consider the folded Nichols algebras in~\cite{Len14} over central extensions of abelian groups, for which $\mathcal{V}$ should be taken to be the respective orbifold of a lattice vertex algebra.

The paper is organized as follows:
In Section~\ref{section2} we present preliminaries on Nichols algebras and in particular on the Weyl groupoid of reflections for a braiding of diagonal type.

In Section~\ref{section3} we briefly present the notion of a vertex algebra $\mathcal{V}_\Lambda$, its representation theory
and the action of screening operators on it. Then we review the main result and its proof in~\cite{Len17} that non-local screenings $\zem_i=\zem_{a_i}$ for the top elements ${\rm e}^{a_i}$ in modules of the Heisenberg vertex algebra fulfill the Nichols algebra relations, if the symmetric matrix $m_{ij}=(a_i,a_j)$ is \emph{subpolar} (Definition \ref{def_smallnessF}), which puts lower bounds on the sums of $m_{ij}$ over all subsets $J$ of indices and ensures the convergence of the relevant integrals and related hypergeometric series. As we show, subpolarity of $(m_{ij})$ and of all matrices $(m_{ij}^\iota)$ obtained from $(m_{ij})$ by repeating some rows and columns follows, if the matrix $(m_{ij})$ is positive definite and all $|m_{ii}|\leq 1$. Under the assumption of subpolarity, the Nichols algebra relations for screening operators follow by expanding monomials of screening operators in terms of certain integrals $\mathrm{F}((m_i),(m_{ij}))$, which in turn can be rewritten as a quantum symmetrizer of another integral $\tilde{\mathrm{F}}((m_i),(m_{ij}))$, which exists for subpolar parameters.

In Section~\ref{section4} we want to improve the results reviewed in the last section by analytically continuing $\tilde{\mathrm{F}}((m_i),(m_{ij}))$ to parameters $(m_{ij})$ beyond subpolarity and study which of the most common Nichols algebra relations still hold for screening operators: In Section~\ref{sec_commutativity} we take an explicit analytic continuation in terms of Beta functions and show that the commutativity relations always continue to hold (see however Example \ref{ex_notcommuative} for a counterexample). In Section~\ref{sec_truncation} we take an explicit analytic continuation in terms of Gamma functions using the Selberg integral formula \cite{Sel44} and thereby we show that the truncation relations of simple root vectors continue to hold iff $m_{ii}>0$. In Section \ref{sec_recursion} we derive (somewhat similar to the Gamma function) a recursive formula for generalized Selberg integrals for $n=3$ and thereby we obtain a non-explicit analytic continuation with explicit simple poles in the region $m_{12}+m_{23}+m_{13}>-2$ and $m_{12}>-1$. We would expect this ansatz to work in general if the subpolarity condition holds merely for $J=I$. This result implies our result on the Serre relations for realizations $(m_{ij})$ of Cartan type for $m>0$, while for $m<0$ one would have to consider additional boundary terms, and for the realization $(m_{ij})$ of super Lie type for $m<0$, while for $m>0$ the subpolarity condition fails even for $J=I$.
On the other hand, we use the $A_n$-Selberg integral formula in \cite{TV03,War09} which produces an explicit expression in terms of Gamma functions for realizations of type $A_n$, but only for a certain linear combination of generalized Selberg integrals. For this reason we can only put this to use for Serre relations of Cartan type for $m<0$. On the other hand this second approach has the potential to give explicit answers for, say, truncation relations of non-simple root vectors in Cartan type $A_n$.

In Section~\ref{section5} we formulate the precise conditions under which we call $(m_{ij})$ (resp.~the associated lattices $\Lambda$) a \emph{realization} of a given braiding~$(q_{ij})$. We also introduce the notions of a pair of vertices $(i,j)$ resp.~a root $\alpha$ being $m$-Cartan or $m$-truncation, according to the respective notions of being $q$-Cartan and $q$-truncation in Section~\ref{section2}, and we discuss the relations between these notions. We visualize a realization by adding the rational numbers $m_{ii}$, $m_{ij}+m_{ji}=2{m_{ij}}$ below a $q$-diagram with decorations $q_{ii}$, $q_{ij}q_{ji}$. For example, the Cartan type realization of $u_q(\g)^+$ in type $C_3$ in the next section with $q={\rm e}^{\pi\i m}$ is depicted as
\[
	\begin{tikzpicture}
	\draw (0,0)--(1.6,0) (1.8,0)--(3.4,0);
	\draw (-0.1,0) circle[radius=0.1cm] node[anchor=south]{$ q^2$} node[anchor=north]{$ 2m$}
	(1.7,0) circle[radius=0.1cm] node[anchor=south]{$ q^2 $}node[anchor=north]{$2m$}
	(3.5,0) circle[radius=0.1cm] node[anchor=south]{$ q^4 $}node[anchor=north]{$ 4m$};
	\draw (0.9,0) node[anchor=south]{$ q^{-2}$}node[anchor=north]{$ -2m\;\;$};
	\draw (2.7,0) node[anchor=south]{$ q^{-4}$} node[anchor=north]{$ -4m\;\;$};
	\end{tikzpicture}
\]

In Section~\ref{Cartan} we address braidings and Nichols algebras of Cartan type, which give the Borel parts $u_q(\g)^+$ of the small quantum groups. We rescale the root lattice of $\g$ by a parameter $m$ and prove that this lattice $\Lambda$ with $m_{ij}=(\alpha_i,\alpha_j)_\g m$ always realizes this braiding. Then we aim to calculate the algebra of screening operators depending on $m$. We establish that subpolarity holds for $\frac{1}{2d}\geq m>0$. Beyond these values, the truncation relations of simple root vectors hold for $m>0$ and the Serre relations (for $\g$ simply laced) hold regardless of $m$. Regrettably, we are neiter able to check the truncation relations for non-simple root vectors (although we would assume them to hold accordingly) nor the additional relations that appear for certain small orders of $q$ \cite{AA17}. For both we would either require analytic continuation of more complicated integrals (extending the techniques discussed in Section~\ref{section4}) or a reflection theory of screening operators, which links the screening operator expression associated to a non-simple root to the corresponding screening operator of a simple root after a change of base in $\Lambda$ by reflection. On the other hand, we regard the possibly failing additional relations to be interesting candidates for highly nontrivial extensions or liftings of Nichols algebra appearing as algebras of screenings.

In Section~\ref{SuperLie} we proceed as in the previous section, but for Lie superalgebras. In Definition \ref{mSuperLie} we consider the standard chamber, in which there is a single fermionic simple root $\alpha_{\f}$ and make an ansatz for a matrix $(m_{ij})$ depending in general on $2$ parameters $m'$, $m''$ corresponding to a~rescaling of the root lattices of the subsystems generated by $\alpha_1,\dots, \alpha_{\f-1}$ and $\alpha_{\f+1},\dots,\alpha_r$ and ${m_{\f\f}=1}$. To prove that this is indeed a realization, we formulate in Corollary \ref{Mcondotherchamb} a~condition for all roots, which we prove in Corollary~\ref{3 conditions II} to hold automatically except in four explicit situation. When we go through all root systems in Sections~\ref{subsec_rank2super} and~\ref{subsec_rankNsuper}, we determine these open situations and compute additional conditions relating $m'$, $m''$ such that also these situations hold. We then verify that the subsequent conditions relating $q'$, $q''$ are indeed the conditions given in Heckenberger's list; one could say that we derive the logarithmic versions of these conditions. For example for $A(n',n'')$ we find $m'+m''=1$ and correspondingly the condition $q'q''=-1$.

Proving that $(m_{ij})$ is the unique realization for large order of $q$ in Lemma~\ref{superlieclassification} is on the other hand a simple application of our result in Cartan type. Again, for small order of $q$ we can have roots that are both $q$-Cartan and $q$-truncation, and correspondingly multiple solutions $(m_{ij})$, for which these roots are possibly either $m$-Cartan or $m$-truncation.

In Section~\ref{section8} we construct realizing lattices for all other finite-dimensional diagonal Nichols~algebras in rank~2. In Section~\ref{Strange} we construct realizations $(m_{ij})$ and check that they are compatible under reflection. The first case (row~6 in \cite[Table~1]{Hecklist}) corresponds to a $\Z_3$-(co\-lor-)Lie algebra \cite{AAB14, Yam07}, it has a free parameter~$q$ and is accordingly again realized by a $1$-parameter family of lattices. As for the Lie and Lie super type we get a Nichols algebra for $m>0$, while for $m<0$ some truncation relations fail. For $q=-1$ there are again additional relations in~\cite{AA17} that we cannot account for. For all other Nichols algebras in rank 2 we find a unique realization, subpolarity holds and the screening algebra is the Nichols algebra.
In Section~\ref{Classification} we show that the examples presented in the previous sections
exhaust all realizing lattices for rank~2 finite-dimensional diagonal
Nichols algebras braiding and that the classification is thus complete.

Finally in Sections~\ref{rank3} and~\ref{section10} we generalize the construction and classification to rank~3 and explain how the answer in rank $\geq 4$ can be obtained in each instance. We conclude by a table of all realizations in rank~$2$ and~$3$.

\section{Preliminaries on Nichols algebras}\label{section2}
We start by giving the basic definitions and examples regarding Nichols algebras of diagonal type. For a thorough account, we refer the reader to the text book~\cite{HS20} and the survey~\cite{AA17}, which includes generators and relations for each Nichols algebra of diagonal type.

\subsection{Definition and properties}
In this article we work over the field of complex numbers $\C$. A \emph{braided vector space} $(M,c)$ is a~vector space together with a \emph{braiding} $c\colon M\otimes M\to M\otimes M$, which is a linear map that fulfills the braid relation or Yang--Baxter equation
\begin{gather*}
(\mathrm{id}\otimes c)(c\otimes \mathrm{id})(\mathrm{id}\otimes c)
=(c\otimes \mathrm{id})(\mathrm{id}\otimes c)(c\otimes \mathrm{id}).\end{gather*}
Hence, a braided vector space comes with an action $\rho_n$ of the \emph{braid group} $\mathbb{B}_n$ on $M^{\otimes n}$ by acting on the $i$-th and $(i+1)$-th tensor factor:
\[c_{i,i+1}:=\mathrm{id}\otimes \cdots \otimes c \otimes \cdots \otimes \mathrm{id}.\]
In this article we restrict ourselves to braidings of the following type:
\begin{defin}
	Let $M$ be a vector space of dimension~$r$, later called \emph{rank}, and with a fixed basis $x_1,\dots, x_r$. Let $(q_{ij})$ for ${i,j=1,\dots, r}$ be an arbitrary matrix with $q_{ij}\in \C^\times$. With this data we define a \emph{braiding of diagonal type} on~$M$ via
	\[c\colon \ c(x_i\otimes x_j)=q_{ij} x_j\otimes x_i.\]
\end{defin}
\begin{defin}\label{def_QuantumSymmetrizer}
 Let $(M, c)$ be a braided vector space. There is a canonical projection $\mathbb{B}_n\twoheadrightarrow \mathbb{S}_n$, which sends the braiding $c_{i,i+1}$ to the transposition~{$(i,i+1)$}. The set-theoretic \emph{Matsumoto section} $s\colon \mathbb{S}_n \rightarrow \mathbb{B}_n$ is defined by $(i,i+1)\mapsto c_{i,i + 1}$ and by the property $s(xy) = s(x)s(y)$ whenever the length of~$xy$ (as a shortest word in a finitely presented group) is the sum of the lengths of~$x$ and~$y$. Then we define the \emph{quantum symmetrizer} as a linear endomorphism of $M^{\otimes n}$ by\looseness=-1
\begin{equation*}
 \sha_{q,n}: =\sum_{\tau\in \mathbb{S}_n} \rho_n(s(\tau)),
\end{equation*}
where $\rho_n$ is the representation of $\mathbb{B}_n$ on $M^{\otimes n}$ induced by the braiding $c$.
Then the \emph{Nichols algebra} or \emph{quantum shuffle algebra} generated by $(M,c)$ is defined~by
\[\cB(M) :=\bigoplus _{n} M^{\otimes n}/ \operatorname{ker} (\sha_{q,n}).\]
\end{defin}
This particular definition of Nichols algebras is due to Woronowicz \cite{Wor89} and Rosso \cite{Rosso98}, and in this way Nichols algebras will appear in the present article. It enables one in principle to compute $\mathcal{B}(M)$ in each degree, but it is very difficult to find generators and relations for $\mathcal{B}(M)$, since in general the kernel of the map $\sha_{q,n}$ is hard to calculate in explicit terms.
In fact $\cB(M)$ is a Hopf algebra in a braided sense and as such it enjoys several equivalent universal properties, see, e.g.,~\cite[Chapters 1 and~7]{HS20}.

\subsection{Examples}
We now present some examples.
\begin{lem}\label{lm_braidingfactor}
	For a diagonal braiding $(q_{ij})$ we have explicitly
	\[\sha_{q,n}: =\sum _{\sigma\in \mathbb{S}_n} q(\sigma)\sigma,\qquad
	q(\sigma)=\prod_{i<j,\;\sigma(i)>\sigma(j)} q_{ij}.\]
\end{lem}	

 As an immediate consequence:
\begin{lem}[Rank $1$]\label{lm_truncation}
 Let $M=x\C$ be a $1$-dimensional vector space with braiding given by $q_{11}=q\in\C^\times$, then
 \[\C\ni \sha_{q,n}=\sum_{\tau\in\mathbb{S}_n}q_{11}^{|\tau|}=\prod_{k=1}^n\frac{1-q^k}{1-q}=:[n]_q!.\]
 This polynomial has zeros at all $q\neq 1$ of order $\leq n$, so the Nichols algebra is
 \[\cB(M)=\begin{cases}
		\C[x]/\big(x^\ell\big), & \text{for } q\in\mathbb{G}_\ell, \ \ell>1,\\
		\C[x], & \text{else},
 \end{cases}\]
where we denote by $\mathbb{G}_\ell$ the set of primitive $\ell$th root of unity.
\end{lem}
\begin{ese}[quadratic relations]\label{exm_quadraticRelation}
	Let $x_i,x_j\in M$. Then we have a quadratic relation
	\[x_ix_j-q_{ij}x_jx_i=0\]
in the Nichols algebra iff the double braiding is trivial $q_{ij}q_{ji}=1$.
For example, the braiding matrices with all entries $+1$ or $-1$ have as Nichols algebra the polynomial algebra or the exterior algebra, respectively.
\end{ese}

We introduce the notation of the $q$-commutator
\begin{gather}
[x_{i_1}\cdots x_{i_m},x_{j_1}\cdots x_{j_n}]_q\nonumber\\
\qquad{} :=(x_{i_1}\cdots x_{i_m})(x_{j_1}\cdots x_{j_n})
-\bigg(\prod_{\substack{1\leq a\leq m \\ 1 \leq b\leq n}} q_{i_aj_b}\bigg)
(x_{j_1}\cdots x_{j_n})(x_{i_1}\cdots x_{i_m}).\label{formula_qcommutator}
\end{gather}
\begin{defin}Starting with \cite{Heck06}, a braiding of diagonal type $(q_{ij})$ can be encoded into a~graph decorated with complex numbers:	
	Each node corresponds to an element~$x_i$, $1\leq i\leq r$ in the basis of $M$ and is decorated with the complex number $q_{ii}$ (self braiding). The edge between any~$x_i$ and~$x_j$ is decorated with the complex number $q_{ij}q_{ji}$ (double braiding):
 	\begin{gather*}
		\begin{tikzpicture}
		\draw (0,0)--(1.6,0);
		\draw (-0.1,0) circle[radius=0.1cm] node[anchor=south]{$ q_{11}$}
		(1.7,0) circle[radius=0.1cm] node[anchor=south]{$ q_{22}$};
		\draw (0.7,0) node[anchor=south]{$ q_{12}q_{21}$};
		\end{tikzpicture}
	\end{gather*}
For $q_{ij}q_{ji}=1$ we do not draw the edge; as discussed above, in this case $x_i$, $x_j$ commute up to the factor $q_{ij}$. For $q_{ij}q_{ji}\neq 1$ we call the vertices \emph{connected} $i\sim j$. The authors call this a \emph{$q$-diagram} to distinguish it from the notion of a Dynkin diagram in the next subsection.
\end{defin}

A deeper reason behind this definition is that the $q$-diagram captures the information that essentially determines the structure of the Nichols algebra. Nichols algebras with the same $q$-diagram are not isomorphic in general, but they have, e.g., the same dimension. In a suitable sense, they are equivalent up to a $2$-cocycle twist.

\begin{ese}[quantum group, \cite{AS10, Lusz93,Rosso98}]\label{exm_quantumgroup}
 Let $\g$ be a finite-dimensional complex semisimple Lie algebra of rank $r$ with root system $\Phi$ and simple roots $\alpha_1,\dots,\alpha_r$ and Killing form $(\alpha_i,\alpha_j)$. Let $q$ be a primitive $\ell$-th root of unity. Consider the $r$-dimensional vector space $M$ with diagonal braiding
 $q_{ij}:=q^{(\alpha_i,\alpha_j)}$.
 Then the Nichols algebra $\cB(M)$ is isomorphic to the positive part $u_q(\g)^+$ of the small quantum group $u_q(\g)$, which is a quotient of the deformation of the universal enveloping of a Lie algebra $U(\g)$.
\end{ese}

\noindent
As an example, the $q$-diagrams for $\g$ of type $A_1\times A_1$, $A_2$ and $B_2$ are
\[
	\begin{tikzpicture}
	\draw (-0.1,0) circle[radius=0.1cm] node[anchor=south]{$ q^2$}
	(1.7,0) circle[radius=0.1cm] node[anchor=south]{$ q^{2}$};
	\end{tikzpicture}
	\hspace{2cm}
	\begin{tikzpicture}
	\draw (0,0)--(1.6,0);
	\draw (-0.1,0) circle[radius=0.1cm] node[anchor=south]{$ q^2$}
	(1.7,0) circle[radius=0.1cm] node[anchor=south]{$ q^{2}$};
	\draw (0.7,0) node[anchor=south]{$ q^{-2}$};
	\end{tikzpicture}
	\hspace{2cm}
	\begin{tikzpicture}
	\draw (0,0)--(1.6,0);
	\draw (-0.1,0) circle[radius=0.1cm] node[anchor=south]{$ q^2$}
	(1.7,0) circle[radius=0.1cm] node[anchor=south]{$ q^{4}$};
	\draw (0.7,0) node[anchor=south]{$ q^{-4}$};
	\end{tikzpicture}
\]

\subsection{Generalized root system and Weyl groupoid}\label{sec_rootsys}
Every Nichols algebra, which is finite-dimensional and of diagonal type, comes with a generalized \emph{root system}, a \emph{Cartan graph}, a \emph{Weyl groupoid} and a \emph{PBW-type basis} \cite{Heck06,HS08, HY08}. For the same statement beyond diagonal type see \cite{AHS10}. These structures play in many respects a similar role as root system, Cartan matrix and Weyl group play for Lie algebras.

Before giving the definitions relevant to this article, we summarize how Weyl groupoids and generalized root systems are explained geometrically in \cite{Cun10}: A root system in the usual sense is an arrangement of hyperplanes in some Euclidean vector space $\mathbb{R}^r$ and a choice of normal vectors thereof, called \emph{roots}, which is stable under the reflections on these hyperplanes and fulfills some integrality condition. The connected components of the complement of all hyperplanes are called \emph{Weyl chambers}, and the normal vectors of the walls of any fixed Weyl chamber give a basis of the vector space, a \emph{set of simple roots}. The reflections act transitively on the Weyl chambers and every roots has integer coefficients with respect to every fixed set of simple roots.

In a generalized setting, called a \emph{crystallographic arrangement}, we drop the euclidean product of the ambient vector space (so the roots are a choice of linear functions), but we keep demanding that every root has integer coefficients with respect to every set of simple roots. The reflection~$\sigma_i$ on the $i$-th wall is the associated linear map that sends the simple root~$\alpha_i$ to $-\alpha_i$ and $\alpha_j$, $i\neq j$ to $\alpha_j-c_{ij}\alpha_i$, where $-c_{ij}$ is the maximal value for which this is a root. In contrast to usual root systems, the set of all roots, written in different bases of simple roots, may not always give the same set of coordinate tuples. Consequently the reflections generate a Weyl groupoid, whose objects are different types of Weyl chambers, and each object has attached its own Cartan matrix~$c_{ij}$ depicted by a Dynkin diagram. Such data has been axiomatized and classified under the name \emph{Cartan graph} \cite{AHS10, Heck07, HS08, HY08}, which we now discuss.
Actually, this behaviour is already familiar from Lie superalgebras, where different sets of simple roots contain different parities, and in extreme cases such as $D(2,1;\alpha)$ even different Dynkin diagrams, see Example~\ref{D21}.

In the following account we follow \cite[Section~3.2]{Heck07} or \cite[Chapters~9,~10 and~15]{HS20}, where more details can be found:

\newcommand{\ZZ}{\Z}
\newcommand{\Cm}{C}
\newcommand{\cm}{c}
\newcommand{\rfl}{\rho}
\newcommand{\s}{\sigma}
\renewcommand{\r}{\mathrm{r}}
\newcommand{\Cc}{\mathcal{C}}
\newcommand{\Wg}{\mathcal{W}}
\newcommand{\re }{^\mathrm{re}}
\newcommand{\rer }[1]{(R\re)^{#1}}
\newcommand{\rsC }{\mathcal{R}}
\begin{defin}[generalized Cartan matrix]
	Let $I:=\{1,\dots,r\}$, where $r$ is called {\it rank}, and
	$\{\alpha_i\,|\,i\in I\}$ the standard basis of $\ZZ ^I$.
	A {\it generalized Cartan matrix}
	$\Cm =(\cm _{ij})_{i,j\in I}$
	is a matrix in $\ZZ ^{I\times I}$ such that
	\begin{enumerate}\itemsep=0pt
		\item[(M1)] $\cm _{ii}=2$ and $\cm _{jk}\le 0$ for all $i,j,k\in I$ with
		$j\not=k$,
		\item[(M2)] if $i,j\in I$ and $\cm _{ij}=0$, then $\cm _{ji}=0$.
	\end{enumerate}
\end{defin}
In the following definition we think on $A$ as a set of Weyl chambers.
\begin{defin}[Cartan graph]
	Let $A$ be a non-empty set, $\rfl _i\colon A \to A$ a map for all $i\in I$,
	and $\Cm ^a=(\cm ^a_{jk})_{j,k \in I}$ a generalized Cartan matrix
	in $\ZZ ^{I \times I}$ for all $a\in A$. The quadruple
	\[ \Cc = \Cc \big(I,A,(\rfl _i)_{i \in I}, (\Cm ^a)_{a \in A}\big)\]
	is called a \textit{Cartan graph} if
	\begin{enumerate}\itemsep=0pt
		\item[(C1)] $\rfl _i^2 = \mathrm{id}$ for all $i \in I$,
		\item[(C2)] $\cm ^a_{ij} = \cm ^{\rfl _i(a)}_{ij}$ for all $a\in A$ and
		$i,j\in I$.
	\end{enumerate}
\end{defin}

\begin{defin}[Weyl groupoid]
	Let $\Cc = \Cc \big(I,A,(\rfl _i)_{i \in I}, (\Cm ^a)_{a \in A}\big)$ be a
	Cartan graph. For all $i \in I$ and $a \in A$ define $\s _i^a \in
	\operatorname{Aut}(\Z ^I)$ by
	\begin{align*}
	\s _i^a (\alpha_j) = \alpha_j - \cm _{ij}^a \alpha_i \qquad
	\text{for all $j \in I$.}
	\end{align*}
	The \textit{Weyl groupoid of} $\Cc $
	is the category $\Wg (\Cc )$ such that $\operatorname{Ob} (\Wg (\Cc ))=A$ and
	the morphisms are compositions of maps
	$\s _i^a$ with $i\in I$ and $a\in A$,
	where $\s _i^a$ is considered as an element in $\operatorname{Hom} (a,\rfl _i(a))$.
	The cardinality of $I$ is called the \textit{rank of} $\Wg (\Cc )$.
\end{defin}
A Cartan graph axiomatizes a set of Cartan matrices, one for every Weyl chamber (or every type of Weyl chamber) $a\in A$, and reflections $\s_i^a$ on simple roots~$\alpha_i$ in the Weyl chamber~$a$, which are linear maps between a space $\Z^I$ attached to $a$ and to the Weyl chamber after reflection~$\rho_i(a)$.

Let $\Cc $ be a Cartan graph. For all $a\in A$ define the set of \emph{real roots} at $a$ by
\[ \rer a=\big\{ \s _{i_1}\cdots \s_{i_k}(\alpha_j)\,|\,
k\in \N _0,\,i_1,\dots,i_k,j\in I\big\}\subseteq \ZZ ^I.\]
A real root $\alpha\in \rer a$ is called positive if $\alpha\in \N _0^I$.

\begin{defin}[root system]
	Let $\Cc =\Cc \big(I,A,(\rfl _i)_{i\in I},(\Cm ^a)_{a\in A}\big)$ be a Cartan
	graph. For all $a\in A$ let $R^a\subseteq \ZZ ^I$, and define
	$m_{i,j}^a= |R^a \cap (\N_0 \alpha_i + \N_0 \alpha_j)|$ for all $i,j\in
	I$ and $a\in A$. We say that
	\[ \rsC = \rsC (\Cc, (R^a)_{a\in A}) \]
	is a \textit{root system of type} $\Cc $, if it satisfies the following
	axioms.
	\begin{enumerate}\itemsep=0pt
		\item[(R1)]
		$R^a=R^a_+\cup - R^a_+$, where $R^a_+=R^a\cap \N_0^I$, for all
		$a\in A$.
		\item[(R2)]
		$R^a\cap \ZZ\alpha_i=\{\alpha_i,-\alpha_i\}$ for all $i\in I$, $a\in A$.
		\item[(R3)]
		$\s _i^a(R^a) = R^{\rfl _i(a)}$ for all $i\in I$, $a\in A$.
		\item[(R4)]
		If $i,j\in I$ and $a\in A$ such that $i\not=j$ and $m_{i,j}^a$
		finite, then
		$(\rfl _i\rfl _j)^{m_{i,j}^a}(a)=a$.
	\end{enumerate}
\end{defin}
\begin{lem}
 Let $\Cc $ be a Cartan graph and $\rsC$ a root system of type $\Cc$. Let $a\in A$. Then for all $i\neq j$
 \[-c^a_{ij}=\max \big\{m\in\N_0\,|\,\alpha_j+m\alpha_i \in R^a_+\big\}.\]
\end{lem}

As a convention, we name a positive root $\alpha$ by indicating with which multiplicity each simple root $\alpha_i$ appears in $\alpha$ (in some fixed Weyl chamber), e.g.,
\[\alpha_{12}=\alpha_1+\alpha_2,\qquad
\alpha_{123}=\alpha_1+\alpha_2+\alpha_3,\qquad
\alpha_{12 2}=\alpha_{112}=2\alpha_1+\alpha_2, \qquad \dots.\]
For roots $\alpha,\beta\in R^a$ we can define a Cartan matrix entry independent of $a$
 \[-c_{\alpha,\beta}=\max \{m\in\N_0\,|\,\alpha+m\beta \in R^+\}.\]

The root system $\rsC $ is called \textit{finite} iff for all $a\in A$ the
set $R^a$ is finite. By~\cite{CH09} if $\rsC $ is a finite root system
of type $\Cc $, then $\rsC =\rsC \re $, and hence $\rsC \re $ is a root
system of type $\Cc$ in that case.

\begin{ese}[Cartan type]Let $\g$ be a semisimple finite-dimensional complex Lie algebra. It is well-known that is uniquely determined (up to isomorphisms) by its root system, which is the root system of a finite Weyl group $W$. The corresponding Cartan graph has exactly one object~$a$ and~$C^a$ is the Cartan matrix of $W$. The set $R^a$ is the root system of~$W$. Alternatively, we can consider the Cartan graph with one object $a$ for each Weyl chamber and the Cartan matrices attached to all objects are equal.
\end{ese}

 The finite Weyl groupoids are classified in \cite{CH09, CH10}; apart from the finite Weyl groups there are an infinite family in rank~$2$, parametrized by triangulations of $n$-gons, an additional series $D_{n,m}$ and 74 sporadic examples.
 It is proven in \cite[Theorem~1.1]{Cun10} that the crystallographic arrangement mentioned in this section's introduction are in bijection to connected simply connected Cartan graphs for which the set of real roots is finite.

We now discuss how Cartan graphs are attached to Nichols algebras of diagonal type, see \cite[Section~3.3]{Heck07} respectively \cite[Chapter~15]{HS20}, and we introduce some additional notation for our purpose:

\begin{lem}\label{lm_Serre} Suppose for some $m\in\N_0$ holds $q_{ii}^{-m} = q_{ij}q_{ji}$ or $q_{ii}^{(1+m)} = 1$, then in the Nichols algebra $\mathfrak{B}(q_{ij})$ the \emph{quantum Serre relation} holds
	\[(\mathrm{ad}_c\,x_i)^{1+m} x_j=0\]
	with $(\mathrm{ad}_c\,x_i):=[x_i,y]_c$ the braided commutator defined above.
\end{lem}
\begin{defin}
To every braiding matrix $(q_{ij})$ we define the associated Cartan matrix $(c_{ij})$ for all $i\neq j$ by
\begin{equation*}
c_{ii}=2 \qquad \text{and} \qquad c_{ij}:=-\min  \big\lbrace m \in \N_0 \mid q_{ii}^{-m} = q_{ij}q_{ji} \  \text{or} \  q_{ii}^{(1+m)} = 1,
\ q_{ii}\neq 1\big\rbrace.
\end{equation*} 	
We assume from now on that $c_{ij}$ is finite. We call two vertices $i\neq j$ \emph{connected} $i\sim j$ iff $c_{ij}\neq 0$.
\end{defin}
We need technical terms to refer to these two conditions. More global terms will be discussed in Definition~\ref{def_qCartanRoot}. We add a prefix $q$- to stress the dependence on the braiding~$(q_{ij})$ and to distinguish them from the respective notions with prefix $m$-~depending on the inner product~$(m_{ij})$ of the lattice.
\begin{defin}
	For a braiding matrix $(q_{ij})$, we call an ordered pair of indices $(i,j)$ with $i\neq j$ to be \emph{$q$-Cartan} if the first condition holds in the minimum~$m$, and we call it \emph{$q$-truncation} if the second condition holds in the minimal $m$, i.e., if
	\begin{align*}
	q_{ii}^{c_{ij}}=q_{ij}q_{ji},  \qquad \text{or}  \qquad q_{ii}^{1-c_{ij}}=1.
	\end{align*}
	A pair $(i,j)$ can be \emph{both} $q$-Cartan and $q$-truncation, namely iff
	$q_{ij}q_{ji}=q_{ii}$.
 	Otherwise we will call a pair \emph{only} $q$-Cartan, respectively \emph{only} $q$-truncation.
\end{defin}

We remark that if $q_{ij}q_{ji}$ is a power of $q_{ii}$ at all, then $(i,j)$ is already $q$-Cartan.

Pairs $(i,j)$ with $i\not\sim j$ are always $q$-Cartan and never $q$-truncation for $q_{ii}\neq 1$.

The matrix $(c_{ij})$ associated to $(q_{ij})$ is a generalized Cartan matrix. The braiding matrix $(q_{ij})$ can be extended uniquely to a bicharacter $q\colon \Z^r\times \Z^r\to \C^\times$ with $q(\alpha_i,\alpha_j)=q_{ij}$. A base change by precomposing with a reflection~$\s_k$ gives a new bicharacter and braiding matrix $(q_{ij}')=\r_k(q_{ij})$ defined by
\[q_{ij}':=q(\s_k(\alpha_i),\s_k(\alpha_j))\]
A short but important calculation \cite[equation~(3)]{Heck05} gives
\[
q_{kk}'=q_{kk},\qquad
q_{ii}'=q_{ii}\cdot p_{ki}^{c_{ki}},\qquad
q_{ki}'q_{ik}'=q_{ki}q_{ik}\cdot p_{ki}^{-2},\qquad
q_{ij}'q_{ji}'=q_{ij}q_{ji}\cdot p_{ki}^{c_{kj}}p_{kj}^{c_{ki}},
\]
where, in our previous wording,
\[p_{ki}:=\begin{cases}
1,&\text{if $(k,i)$ is $q$-Cartan},\\
q_{ii}^{-1}q_{ij}q_{ji}&\text{if $(k,i)$ is $q$-truncation}.
\end{cases}
\]
\begin{cor}
If $(k,i)$ is $q$-Cartan for all $i\neq k$, then $\r_k(q_{ij})=(q_{ij})$.
\end{cor}
The set $A$ of all braiding matrices arising in this way together with the maps $\r_k$ defines a Cartan graph. The Nichols algebras associated to these different braiding matrices are not isomorphic, but they have the same dimension (if finite), an isomorphic Drinfeld double, and are closely related by \cite{BLS15, Heck07,HS11}. By \cite[Remark~2.8]{AHS10} and \cite[Theorem~3.13]{Heck07} we can write as $\N_0^r$-graded vector spaces
\[\mathfrak{B}(q_{ij})= \bigotimes_ {\alpha\in R^+} \C[x_\alpha]/\big(x_\alpha^{\operatorname{ord}(q(\alpha,\alpha))}\big)\]
and then the set of \emph{degrees} $R^+$ is a root system for this Cartan graph. In particular, $\mathfrak{B}(q)$ has finite dimension if the root system $R^+$ is finite and all $\operatorname{ord}(q(\alpha,\alpha))$ are finite.

\begin{ese}[$D(2,1;\alpha)$, \cite{HY08}]\label{D21}
	We consider, as an example, the finite-dimensional diagonal Nichols algebra of rank $3$ with a braiding $(q_{ij})$ in an initial Weyl chamber, which has the following properties
	\[q_{ii}=-1,\qquad q_{ij}q_{ji}=\zeta,\]
	with $i\neq j$ and $\zeta \in \mathbb{G}_3$ a primitive third root of unity. As $q$-diagram
\[
		\begin{tikzpicture}
		\draw (0,0)--(-0.8,-1.4)--(0.8,-1.4)--(0,0);
		\draw (0,0.1) circle[radius=0.1cm] node[anchor=south]{$ -1$}
		(-0.9,-1.4) circle[radius=0.1cm] node[anchor=east]{$ -1$}
		(0.9,-1.4) circle[radius=0.1cm] node[anchor=west]{$ -1$};
		\draw (-0.6,-0.7) node[anchor=east]{$ \zeta$};
		\draw (0.6,-0.7) node[anchor=west]{$ \zeta$};
		\draw (0,-1.4) node[anchor=north]{$ \zeta$};
		\end{tikzpicture}
\]

As it turns out, the overall root system has seven positive roots. If $\lbrace\alpha_1, \alpha_2, \alpha_3\rbrace$ are the simple roots in the Weyl chamber shown above, then the positive roots in this basis are
\[\{ \alpha_1, \alpha_2,  \alpha_3,  \alpha_{12}, \alpha_{23},  \alpha_{13},  \alpha_{123} \}\]
and the Cartan matrix of this Weyl chamber is
\[\big(c_{ij}^{\mathrm{I}}\big)= \begin{bmatrix}
\hphantom{-}2 & -1 & -1 \\
-1 & \hphantom{-}2 & -1 \\
-1 & -1 & 2
\end{bmatrix}.
\]
We now reflect around $\alpha_2$. Then the new simple roots are $\lbrace\alpha_{12}, -\alpha_2, \alpha_{23}\rbrace$. We compute the new $q$-diagram $\r_2(q_{ij})=(q_{ij}')$ using this basis transformation and the extension of $(q_{ij})$ to a~bicharacter~$q$:
\begin{gather*}
q_{22}'=q(-\alpha_{2},-\alpha_{2})= q_{22}=-1,\\
q_{11}'=q(\alpha_{12},\alpha_{12})= (q_{12}q_{21})q_{11}q_{22} =\zeta,\\
q_{33}'=q(\alpha_{23},\alpha_{23})= (q_{23}q_{32})q_{33}q_{22}=\zeta,\\
q_{12}'q_{21}'=q(\alpha_{12},-\alpha_{2})q(\alpha_{12},-\alpha_{2})
=(q_{12}q_{21})q_{22}^{-2}=\zeta^{-1}, \\
q_{32}'q_{23}' =q(\alpha_{23},-\alpha_{2})q(\alpha_{23},-\alpha_{2})
=(q_{23}q_{32})q_{22}^{-2}=\zeta^{-1}, \\
q_{13}'q_{31}' =q(\alpha_{12},\alpha_{23})q(\alpha_{23},\alpha_{12})
=(q_{12}q_{21})(q_{23}q_{32})(q_{13}q_{31})q_{22}^{2}=1.
\\
	\begin{tikzpicture}
	\draw (0,0)--(1.6,0) (1.8,0)--(3.4,0);
	\draw (-0.1,0) circle[radius=0.1cm] node[anchor=south]{$ \zeta$}
	(1.7,0) circle[radius=0.1cm] node[anchor=south]{$ -1$}
	(3.5,0) circle[radius=0.1cm] node[anchor=south]{$ \zeta$};
	\draw (0.8,0) node[anchor=south]{$ \zeta^{-1}$};
	\draw (2.6,0) node[anchor=south]{$ \zeta^{-1}$};
	\end{tikzpicture}
\end{gather*}
In this new basis the positive roots are
\[\{ \alpha_{12},  -\alpha_2,  \alpha_{23},  \alpha_{1},  \alpha_{3},  \alpha_{123},  \alpha_{13} \}\]
and the new Cartan matrix is hence
\[\big(c_{ij}^{\mathrm{II}}\big)= \begin{bmatrix}
\hphantom{-}2 & -1 & \hphantom{-}0 \\
-1 & \hphantom{-}2 & -1 \\
\hphantom{-}0 & -1 & \hphantom{-}2
\end{bmatrix}.
\]
Even though this particular Cartan matrix is of type $A_3$, the root system has one root more than the root system $A_3$. In the following figure we show the hyperplane arrangement of the root system in an affine resp.~projective picture:
\begin{center}
	\includegraphics[width=0.47\textwidth]{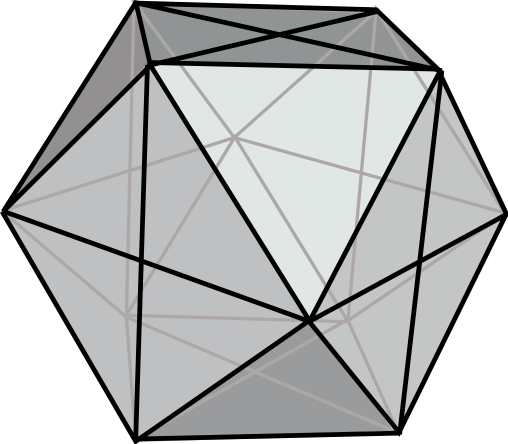}
	\hspace{1cm}
	\includegraphics[width=0.43\textwidth]{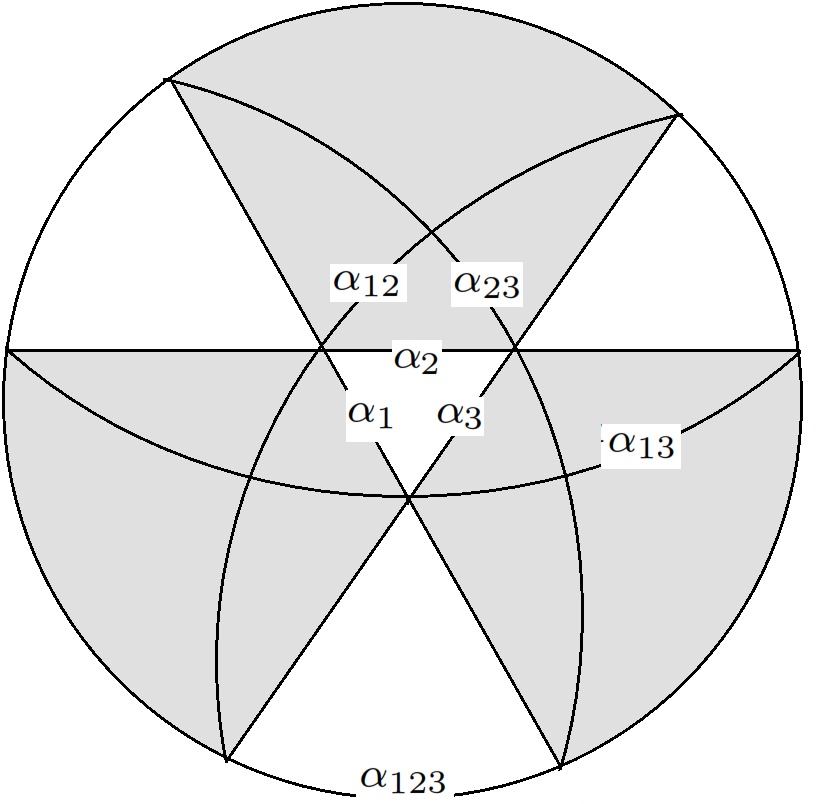}
\end{center}
Each of the seven projective lines corresponds to the hyperplane through the origin orthogonal to one root. Each triangle is a Weyl chamber with the three adjacent hyperplanes corresponding to the three simple roots. Equilateral triangles (white) correspond to the Cartan matrix I and right triangles (grey) to the Cartan matrix~II.
\end{ese}

We now introduce some more properties of roots in a root system associated to a braiding matrix $(q_{ij})$, which are relevant to this article.

\begin{defin} \label{def_qCartanRoot}\quad
\begin{enumerate}\itemsep=0pt
	\item A root $\alpha$ is called \emph{$q$-Cartan root}, if in every Weyl chamber containing $\alpha=\alpha_i$ as a~simple root, all pairs of vertices $(i,j)$ with $i\sim j$ are $q$-Cartan.
	\item A root $\alpha$ is called \emph{$q$-truncation}, if in every Weyl chamber containing $\alpha=\alpha_i$ as a~simple root, all pairs of vertices $(i,j)$ with $i\sim j$ are $q$-truncation.
	\item A root $\alpha$ is called \emph{only $q$-Cartan root}, if in every Weyl chamber containing $\alpha=\alpha_i$ as a~simple root, all pairs of vertices $(i,j)$ with $i\sim j$ are only $q$-Cartan.
\end{enumerate}
\end{defin}
The term \emph{Cartan-} was coined in this context by \cite{Ang13} as \emph{Cartan vertex}~$\alpha_i$ (meaning that all pairs of vertices $(i,j)$ are $q$-Cartan). The term \emph{Cartan root} $\alpha$ appears first in \cite{Ang16}. In \cite{AAR19} initially the independence of this notion from the basis of simple roots was proven implicitly. The authors thank I. Angiono for explaining the following much simpler proof that uses a base-independent characterization of the term: 	
\begin{prop}\label{prop_IvanCartanRoot} A root
	$\alpha$ is a $q$-Cartan root iff the values $q(\alpha,\beta)q(\beta,\alpha)$ for all $\beta\in\Z^r$ lie in the multiplicative group generated by $q(\alpha,\alpha)$.
	
	As a consequence, suppose in some Weyl chamber, which containing $\alpha=\alpha_i$ as a simple root, that all pairs of vertices $(i,j)$ with $i\sim j$ are $q$-Cartan. Then the same is already true in all such Weyl chambers and $\alpha$ is a Cartan root.
\end{prop}
The other two notions we defined above do not have such nice characterizations (and are probably less fundamental) and require the knowledge of the entire root system. The following special case is frequent and easy to recognize:
\begin{ese}[fermionic root]
	A root $\alpha$ with $q(\alpha,\alpha)=-1$ is $q$-truncation.
\end{ese}
\begin{oss}Only very few root systems admit roots $\alpha$ that are neither $q$-Cartan nor $q$-truncation. In such a case, there has to exist a Weyl chamber containing $\alpha=\alpha_i$ as a simple root and the following type of rank $3$ subdiagram
	\begin{gather*}
		\begin{tikzpicture}
		\draw (0,0)--(1.6,0) (1.8,0)--(3.4,0);
		\draw (-0.1,0) circle[radius=0.1cm];
		\draw (1.7,0) circle[radius=0.1cm]
		node[anchor=south]{$q$}
		node[anchor=north]{$\vphantom{X^X}\alpha_i$};
		\draw (3.5,0) circle[radius=0.1cm];
		\draw (0.8,0) node[anchor=south]{$q^{-a}$};
		\draw (2.6,0) node[anchor=south]{$q'$};
		\end{tikzpicture}
	\end{gather*}
	where $q^{-a}\neq q$ (in particular $q\neq -1$), so the pair $(2,1)$ is $q$-Cartan and not $q$-truncation and where $q'$ is no power of $q$, so the pair $(2,3)$ is $q$-truncation and not $q$-Cartan. Such a chamber must exist, otherwise $\alpha$ is $q$-truncation, and then such a neighbor must exist in this chamber, otherwise $\alpha$ is $q$-Cartan in all chambers.
	
A quick inspection of all $q$-diagrams of rank $3$ thus shows that roots being neither $q$-Cartan nor $q$-truncation appear for~\cite[rows~16 and~17]{Hecklist}, and consequently for $q$-diagrams of rank $\geq 4$ containing either of these.
\end{oss}
Roots $\alpha$ that are only $q$-Cartan are $q$-Cartan roots such that $\operatorname{ord}(q(\alpha,\alpha))>1-c_{ij}$ in all Weyl chambers containing $\alpha=\alpha_i$ as a simple root, and all $j\sim i$. For example for Nichols algebras of Cartan type $u_q(\mathfrak{g})^+$ we have (see Proposition~\ref{prop_CartanTrunc})
\begin{center}\renewcommand{\arraystretch}{1.2}
	\begin{tabular}{lll}
		$\mathfrak{g}$ & $q$ & both $q$-Cartan and $q$-truncation\\
		\hline
		$A_n$ & $q^2=-1$ & all roots \\
		$B_n$, $C_n$, $F_4$ & $q^4=-1$ & long roots \\
		$B_n$, $C_n$, $F_4$ & $q^2\in\mathbb{G}_3$ & short roots \\
		$G_2$ & $q^6=-1$ & long roots \\
		$G_2$ & $q^2\in \mathbb{G}_4$ & short roots
	\end{tabular}
\end{center}
It is interesting, that these are precisely the cases, in which the Nichols algebras need additional relations, see~\cite{AA17}.

\newcommand{\V}{\mathcal{V}}
\section{Preliminaries on screening operators}\label{section3}
\subsection{Vertex algebras and their representation theory}\label{VOAREP}
A \emph{vertex operator algebra} (VOA) is, roughly speaking, a commutative algebra that depends analytically on a complex variable~$z$. More precisely, a vertex operator algebra~$\V$ is an infinite-dimensional graded vector space with a linear map
\[\Y\colon \ \V\otimes_\C \V\to \V\big[\big[z,z^{-1}\big]\big],\]
where $\V\big[\big[z,z^{-1}\big]\big]$ denotes Laurent series in a formal variable $z$ with coefficients in~$\V$. The axioms of a vertex operator algebra include a version of commutativity or locality, which relates $\Y(a,z)\Y(b,w)$ and $\Y(b,w)\Y(a,z)$ for $z,w,z-w\neq 0$. As an implication, one also has a version of associativity, which relates these two expressions to $\Y(\Y(a,z-w)b,w)$.
An additional axiom requires that conformal transformations of the variable $z$ in $\Y(a,z)$ are compatible with an action of the Virasoro algebra on $\V$, which is part of the data. Standard mathematical textbooks on vertex operator algebras include \cite{FBZ04, Kac98}. Vertex operator algebras are motivated by physics, where they describe the holomorphic (chiral) part of a 2-dimensional quantum field theory with conformal symmetry.

There is a straightforward notion of a module $W$ over a vertex algebra $\V$. Under certain finiteness-assumptions on a vertex operator algebra $\mathcal{V}$, the category of $\mathcal{V}$-modules has a tensor product $\boxtimes$ and a braiding~\cite{HLZ10}.

\begin{ese}
 The easiest example of a vertex operator algebra is the Heisenberg algebra~$\mathcal{H}^r$ based on an $r$-dimensional Euclidean vector space $\mathbb{R}^r$. The irreducible modules $\mathcal{H}_a$ are pa\-ra\-met\-ri\-zed by vectors $a\in\mathbb{R}^r$, with $\mathcal{H}_0=\mathcal{H}^r$, tensor product $\mathcal{H}_a\boxtimes \mathcal{H}_b=\mathcal{H}_{a+b}$, and braiding given by the scalar ${\rm e}^{\pi {\rm i}(a,b)}$.
\end{ese}

From the perspective of our article, this is already an interesting vertex operator algebra: In the next section we will define screening operators $\zem_{a_i}$, and the idea of this article is to analyze the algebra generated by these screening operators, which will be largely determined by the braidings $q_{ij}={\rm e}^{\pi {\rm i}(a_i,a_j)}$.

We also introduce some more vertex algebras that eventually come into play, and which motivate our work:

\begin{ese}
For every even integral lattice $\Lambda$ with inner product $(\,, \,)$, it is possible to associate a \emph{lattice vertex algebra} $\mathcal{V}_\Lambda$. Its category of representations is equivalent to the category of $\Lambda^*/\Lambda$-graded vector spaces, with associator $\omega$ and braiding~$\sigma$ associated to a quadratic form~${\rm e}^{\pi {\rm i} (a, a)}$, $a \in \Lambda^*$.
If we restrict these modules to the Heisenberg algebra $\mathcal{H}^r\subset \V_\Lambda$, the modules decompose and we have $\mathcal{V}_{a+\Lambda}=\bigoplus_{a'\in a+\Lambda} \mathcal{H}_{a'}$.
\end{ese}

Both examples are treated in \cite{FBZ04,Kac98} at the level of vertex algebras and modules. The standard reference for the tensor product and braiding in both examples is~\cite{DL93}. See also the overview in \cite[Section~2.3]{CGR20}.

\subsection{Screening operators and Nichols algebra relations} \label{smallTheorems}
We are now going to define the main protagonists of this paper, the \emph{screening operators}. They go back to~\cite{DF84} and appear throughout vertex operator literature. Our main focus are screening operators for elements in a module different from the vacuum module, and we call those \emph{non-local screening operators}. We introduce them from a slightly novel perspective:

Given $\mathcal{V}$ a VOA, $W$ module of $\mathcal{V}$ and $w \in W$. The tensor product $W\boxtimes U$ with some other module $U$ is defined in~\cite{HLZ06} by the universal property that there exists an intertwining operator
\[\Y(w, z)\colon \ W\otimes U \rightarrow (W \boxtimes U)[\mathrm{log} z]\{\{z\}\},\]
where $\{\{z\}\}$ denotes power series with arbitrary complex exponents, and their matrix elements are multivalued analytic functions on $\C\backslash\{0\}$. If we evaluate $\Y$ on our fixed element $w\in W$, we get a map
\[\Y(w, z)\colon \ U \rightarrow (W \boxtimes U)[\mathrm{log} z]\{\{z\}\}.\]
	Integrating the variable $z$ over the unit circle around $z=0$, lifted to a path in the multivalued covering, we get linear maps into the algebraic closure
	\[\mathfrak{Z}_w \colon \  U \rightarrow \overline{W \boxtimes U}.\]
	These maps are called \emph{screening operators}. If $W=\mathcal{V}$ is the vertex algebra itself, then the intertwining operator is simply the vertex operator, which is a power series with integer exponents and without logarithms; in this case the integration returns simply the $z^{-1}$-coefficient or residue of $\Y$. If we ask in addition that the conformal weight is $h(w)=1$, then the commutator formula shows that the screening operator commutes with the action of the Virasoro algebra. On the other hand, for non-local screening operators, as defined above, all non-integer $z$-powers contribute to the integral in the multivalued covering, which is why we get a result in the algebraic closure. Also the consequence of $h(w)=1$ is more subtle, one would expect that at least certain suitable powers of screening operators commute with the Virasoro algebra. Defining and studying non-local screening operators in this way is new and might appear strange~-- however, non-local screening operators do appear prominently in literature for a long time. For example, the Felder complex~\cite{Fel89} consists of suitable powers of non-local screening operators.
	
If the singularity of products $\Y(w_1,z_1)\cdots \Y(w_n,z_n)$ at points $z_i=z_j$ is mild enough so that certain integrals converge (\emph{subpolar}, see below), then we expect these screening operators to fulfill the relations of the Nichols algebra~$\cB(W)$ associated with the module~$W$ and its braiding in the representation category $\operatorname{Rep}(\V)$, which expresses the multivaluedness (or non-locality) of the intertwining operator.
If the singularity at $z=0$ is more severe, then the screening operators should generate some algebra extension of the Nichols algebra.

Establishing this result for general VOAs is work in progress, but for Heisenberg and lattice VOA and their explicit intertwiners, it is proven in \cite{Len17}:
\begin{defin}[{\cite[Definition 5.4]{Len17}}]\label{def_smallnessF}
	Let $I=\{1,\dots,n\}$. A set of complex parameters $(m_{ij})$ with $1\leq i\leq j\leq n$ is called \emph{subpolar}, if for every subset $J\subset I$ with $|J|\geq 2$ the following inequality holds:
	\[\sum_{i<j,\;i,j\in J} \mathfrak{Re}(m_{ij}) > -|J|+1.\]
	As a weakening, $(m_{ij})$ is called \emph{subpolar on intervals} with respect to a total order of the index set, if the inequality holds for all $J\subset I$ with $|J|\geq 2$ which have the form of an interval $J=[a,b]=\{x\,|\, a\leq x \leq b\}$.
\end{defin}
We later consider monomials $x_{\iota(1)}\cdots x_{\iota(n)}$ of total degree $n$ in $r$ variables $x_1,\dots,x_r$. Correspondingly, we now consider the following sets of parameters:
\begin{defin}
Suppose $(m_{ij})$ with $1\leq i\leq j\leq r$ is given. Then for any map
\[\iota\colon \ \{1,\dots,n \}\to \{1,\dots, r\}\]
we define a new set of complex parameters $(m^\iota_{i'j'})$
\[m^\iota_{i'j'}:=m_{\iota(i'),\iota(j')},\qquad 1\leq i'\leq j'\leq n.\]
Up to permuting indices, $(m^\iota_{i'j'})$ is determined by $(m_{ij})$ and the \emph{degrees} of~$\iota$
\[d_i^\iota:=\big|\iota^{-1}(i)\big|,\qquad \sum_{i=1}^r d_i^\iota=n.\]
\end{defin}
\begin{oss}
		 $\Lambda=\Z^r$ with basis $\alpha_1,\dots,\alpha_r$ and inner product $(\alpha_1,\alpha_j)=m_{ij}$. Then, as a~geometric reformulation, subpolarity of $m^\iota_{i'j'}$ is equivalent to the cube $\prod_{i=1}^r[0,d_i]$ intersecting the ball defined by
\[\frac{1}{2}(\alpha,\alpha)-\big(\alpha,\rho-\rho^\vee\big)\leq 1\]
only in $\alpha=0$ and $\alpha=\alpha_i$, which is on the boundary, if $\rho$, $\rho^\vee$ are vectors in the complexification of $\Lambda$ satisfying $(\alpha_i,\rho)=m_{ii}$ and $\big(\alpha_i,\rho^\vee\big)=1$.
\end{oss}

\begin{lem}[{\cite[Lemma 5.5]{Len17}}]\label{pos+small} If $\Lambda$ is positive definite and $(a_i, a_i)\leq 1$ for a fixed basis $a_i$, then $(m_{ij})$ and $(m^\iota_{i'j'})$ for all $\iota$ are subpolar.
\end{lem}

\begin{teo}[{\cite[Theorem 7.1]{Len17}}]\label{NArelations}
For a non-integral lattice $\Lambda$ of rank $r$ and elements $a_1, \dots, \allowbreak a_n \in \Lambda$, we consider the elements ${\rm e}^{a_i}$ in modules of the associated Heisenberg VOA $\mathcal{H}$. The braiding between two elements is
\[{\rm e}^{a_i} \otimes {\rm e}^{a_j} \mapsto q_{ij} \; {\rm e}^{a_j} \otimes {\rm e}^{a_i},\qquad
 \text{where} \quad q_{ij} := {\rm e}^{{\rm i} \pi m_{ij}}, \quad m_{ij}:=(a_i, a_j).\]
Consider the diagonal Nichols algebra $\mathcal{B}(q_{ij})$ with braiding matrix $(q_{ij})$ and generators $x_{i}$ with $1\leq i\leq r$. Then for any relation in the Nichols algebra, homogeneous in degree $(d_1, \dots, d_r) \in \N^r$, the same relation holds for the screening operators $\mathfrak{Z}_{a_i}$, assuming that $\big(m^\iota_{ij}\big)$ is subpolar for $d_i=\iota^{-1}(i)$.
\end{teo}

\begin{ese} In the case $\Lambda= \frac{1}{\sqrt{p}} \Lambda_\mathfrak{g}$, with $\Lambda_\mathfrak{g}$ the root-lattice of a complex finite-dimensional simple Lie algebra $\mathfrak{g}$ of rank $r$, and $\ell=2p$ an even integer, we obtain as $\mathcal{B}(q_{ij})$, the small quantum group $u_q(\mathfrak{g})^+$, where $q$ is a primitive $\ell$-th root of unity and the braiding is
\[ q_{ij} = {\rm e}^{{\rm i} \pi \big(\frac{1}{\sqrt{p}}\alpha_i, \frac{1}{\sqrt{p}}\alpha_j\big)} = {\rm e}^{\frac{2{\rm i} \pi}{\ell}(\alpha_i,\alpha_j)} = q^{(\alpha_i,\alpha_j)},\]
for $\alpha_1,\dots,\alpha_r$ the basis of simple roots in $\Lambda_\mathfrak{g}$.
\end{ese}

Since we want to improve Theorem \ref{NArelations}, we first recall the steps of the proof in~\cite{Len17}: Any iteration of screening operators (in a rather general type of vertex algebra) can by Theorem~4.3 cit.~loc.~be expanded as follows
\begin{teo}\label{associativity}
\begin{gather*}
\left(\prod_{i=1}^n \zem_{a_i}\right)  v
=\sum_{(m_i),(m_{ij})} \sum_{(k_i)\in \N^n} [(k_i),(m_i),(m_{ij}),(a_i)]\cdot \Fp((m_i+k_i),(m_{ij})),
\end{gather*}
where we introduce
\begin{gather*}
[(k_i),(m_i),(m_{ij}),(a_i)]\\
\qquad{} :=v^{(n+1)}\prod_{i=n}^1 \frac{\partial^{k_i}}{k_i!}a_i^{(n+1)}
  \prod_{1\leq i \leq n}\big\langle a_i^{(1)},v^{(i)}\big\rangle_{m_i}
\prod_{1\leq i<j \leq n} \big\langle a_i^{(n-j+2)},a_j^{(n-i+1)}\big\rangle_{m_{ij}},\\
\Fp((m_i),(m_{ij}))
:=\sum_{(k_{ij})\in \N_0^{n\choose 2}}\prod_{i=1}^n \operatorname{res}\big(z_i^{(m_i+k_i)+\sum_{i<j}(m_{ij}-k_{ij})+\sum_{j<i}k_{ji}}\big)\prod_{i<j} (- 1)^{k_{ij}}{m_{ij} \choose k_{ij}}.
\end{gather*}
The notation in these formulae is as follows:
\begin{itemize}\itemsep=0pt
	\item $[(k_i),(m_i),(m_{ij}),(a_i)]$ as well as $(a_i)$, $v$ are elements in modules of the vertex algebra, which is modeled by some $($in our case commutative, cocommutative$)$ Hopf algebra, see Section~\ref{section3} cit.~loc. Expressions $a^{(1)}\otimes a^{(2)}$ denote in Sweedler notation the iterated coproduct and $\langle-,-\rangle_m$ are certain pairings, nonzero only for a~finite number of values $m_i,m_{ij}\in \C$ over which we sum. The only properties of $[(k_i),(m_i),(m_{ij}),(a_i)]$ we require in the following are
	\begin{itemize}\itemsep=0pt
		\item It is invariant under permutation of the index set, i.e., simultaneously permuting $(k_i)$, $(m_i)$, $(m_{ij})$, $(a_i)$, if commutativity and cocommutativity is in place. Thus commutativity of screening operators depends entirely on $\Fp((m_i),(m_{ij}))$.
		\item It is invariant under permuting only the $(k_i)$, if the permutation preserves the partition of the index into subsets of indices with equal~$a_i$. In particular, the product of screenings only depends on the symmetrization $\Fp((m_i),(m_{ij}))^{{\rm sym}}$ over these subsets of indices.
	\end{itemize}
	\item $\Fp((m_i),(m_{ij}))$ is an infinite ${n\choose 2}$-fold series of complex numbers $($a~generalized hypergeometric series on the boundary of its convergence disc$)$ depending on a set of complex parameters $(m_{ij})$ and $(m_i)$ for $1\leq i<j\leq n$. The symbol $\operatorname{Res}(z^m)$ denotes a~formal residue of the possibly multivalued function
	\[\operatorname{Res}(z^m)
	 :=  \begin{cases}
	0, & m\in\Z \backslash\{-1\},\\
	1, & m= -1,\\
	\dfrac{({\rm e}^{2\pi{\rm i} (m+1)}-1)}{2\pi{\rm i}(m+1)},
	& m\not\in\Z.
	\end{cases}\]
	Convergence of this series is rather subtle and holds for example if $(m_{ij})$ is subpolar, see Lemmas~{\rm 5.2} and~{\rm 5.22} cit.~loc.
	\end{itemize}
Note that the quantity $(m_{ij})$ is coupled to the pairing of $a_i$, $a_j$, and the quantity $(m_i)$ is coupled to the pairing of $a_i$ with the element $v$ acted upon.
\end{teo}

We call $\Fp((m_i),(m_{ij}))$ the \emph{quantum monodromy numbers}, and they play in some analytical sense the role of structure constants for screening operators. They can be expressed as an integral of a multivalued function in $n$ complex variables over a suitable lift of $(\mathbb{S}_1)^n$ to the multivalued covering. The integral converges for subpolar~$(m_{ij})$, see Section~5.2 cit.~loc.
\begin{gather*}
\Fp((m_i),(m_{ij})) =\int\cdots\int_{[e^0,{\rm e}^{2\pi}]^n} {\rm d} z_1\cdots {\rm d} z_n \prod_i z_i^{m_i}\prod_{i<j}(z_i - z_j)^{m_{ij}}.
\end{gather*}
	With the Main Theorem~5.20 cit.~loc.~we express this function as quantum symmetrizer of another integral, which converges if $(m_{ij})$ is subpolar on intervals
\begin{gather*}
\Fp((m_i),(m_{ij})) =\sum_{\sigma\in \S_n} q(\sigma) \rFp\big(\big(m_{\sigma^{-1}(i)}\big),\big(m_{\sigma^{-1}(i)\sigma^{-1}(j)}\big)\big),\\
\rFp((m_i),(m_{ij}))
 =\int\cdots\int_{\triangle} {\rm d} z_1\cdots {\rm d} z_n \prod_i z_i^{m_i}\prod_{i<j}(z_i - z_j)^{m_{ij}},
\end{gather*}
where $\triangle:=\big\{\big({\rm e}^{2\pi\i t_1},\dots,{\rm e}^{2\pi\i t_n}\big) \,|\, 0<t_1<\cdots <t_n<1\big\}$ and where $q(\sigma)$ is the braiding factor in the quantum symmetrizer with respect to the braiding $q_{ij}={\rm e}^{\pi\i m_{ij}}$. The contour integral $\rFp$ can by Lemma~5.15 cit.~loc.~be deformed to a sum of real integrals, which have additional poles depending on~$(m_i)$:
\begin{gather*}
\rFp((m_i),(m_{ij}))
 :=\frac{1}{(2\pi\i)^n}\sum_{k=0}^n (-1)^{n-k}\left(\prod_{i=n-k+1}^n {\rm e}^{2\pi\i\;m_i}\right)\sum_{\eta\in\S_{n-k,\overline{k}}}\left(\prod_{i<j,\;\eta(i)>\eta(j)}{\rm e}^{\pi\i\;m_{ij}}\right)\\
 \hphantom{\rFp((m_i),(m_{ij})):= }{}
 \times
\Sel\big(\big(m_{\eta^{-1}(i)}\big),(0),\big(m_{\eta^{-1}(i)\eta^{-1}(j)}\big)\big),
\end{gather*}
where $\mathbb{S}_{n-k,\overline{k}}:=\{\eta\in \mathbb{S}_n \,|\, \forall_{i<j\leq n-k}\;\eta(i)<\eta(j)\text{ and } \forall_{n-k<i<j}\;\eta(i)>\eta(j)\}$ denotes a version of $(n-k,k)$-shuffles and $\Sel$ the generalized Selberg integral
\begin{gather*}
\Sel((m_i),(\bar{m}_i),(m_{ij})):=\int\!\cdots\!\int_{1> z_1>\dots >z_{n}> 0}\! {\rm d}z_1\cdots {\rm d}z_{n}
\prod_i z_i^{m_i}\prod_i (1-z_i)^{\bar{m}_i}\prod_{i<j} (z_i-z_j)^{m_{ij}}.
\end{gather*}
Now we restrict ourselves to the Heisenberg vertex algebra $\mathcal{H}^r$ of rank~$r$, to a not necessarily integral lattice $\Lambda=\Z^r$ with basis $\alpha_1,\dots, \alpha_r$ and scalar product $(\alpha_i,\alpha_j)$, and to screening operators $\zem_1,\dots,\zem_r$ associated to the elements $a_i={\rm e}^{\alpha_i}$.
An arbitrary monomial of total degree~$n$ in these screening operators is
\[ \zem_{\iota(1)}\cdots \zem_{\iota(n)},\qquad \iota\colon \ \{1,\dots,n\}\to \{1,\dots,r\}\]
and if we expand the action of this monomial on an element $v\in \mathcal{H}_\lambda$ by the formula above, then quantum monodromy numbers $\Fp((m_i+k_i),(m_{ij}))$ with $m_i=(\alpha_i,\lambda)$ and $m_{ij}=(\alpha_i,\alpha_j)$ and $k_i\in \Z$ appear.

Consider accordingly the braided vector space $\C^r$ with basis $x_1,\dots, x_r$ and braiding $q_{ij}={\rm e}^{(\alpha_i,\alpha_j)}$.
Consider the Nichols algebra of diagonal type (be it finite or infinite-dimensional). A~monomial in the tensor algebra is of the form
\[\hphantom{\sum_{\iota} c_\iota\cdot }
x_{\iota(1)}\cdots x_{\iota(n)},\qquad \iota\colon \ \{1,\dots,n\}\to \{1,\dots,r\}.\]
A~linear combination of such monomials
\[\sum_{\iota} c_\iota\cdot x_{\iota(1)}\cdots x_{\iota(n)},\qquad \iota\colon \ \{1,\dots,n\}\to \{1,\dots,r\}\]
is zero in the Nichols algebra iff it is in the kernel of the quantum symmetrizer $\sha_q$ with respect to the braiding matrix $q_{\iota(i),\iota(j)}$ and the associated braiding factor~$q_\iota(\sigma)$. This can be reformulated to a condition on the coefficients~$c_\iota$:
\begin{align*}
0&=\sum_{\iota}c_{\iota}\sum_{\sigma\in\S_n} q_{\iota}(\sigma) \cdot \cdot \left(x_{\iota(\sigma^{-1}(1))} \cdots x_{\iota(\sigma^{-1}(n))}\right)\\
&=\sum_{\iota}\left(\sum_{\sigma\in\S_n} q_{\iota\circ\sigma}(\sigma)c_{\iota\circ\sigma} \right) \left(x_{\iota(1)} \cdots x_{\iota(n)}\right),
\end{align*}
which means that the bracket vanishes for each $\iota$.

Take on the other hand a corresponding product of screenings $\zem_{\iota(1)}\cdots \zem_{\iota(n)}$, expand it by Theorem \ref{associativity} and then rearrange by substituting $\iota\circ\sigma$ for~$\iota$, which also changes $a_{\iota(i)}$ to $a_{\iota(\sigma^{-1}(i))}$, and then by using the simultaneous symmetry in $[(k_i),(m_{\iota(i)}),(m_{\iota(i)\iota(j)})]$ while renaming the summation indices $k_{\sigma^{-1}(i)}$ again~$k_i$:
 \begin{gather*}
\sum_{\iota} c_{\iota}\left(\prod_{i=1}^n \zem_{\iota(i)}\right)  v
=\sum_\iota c_\iota \sum_{(k_{i})} [(k_i),(m_{\iota(i)}),(m_{\iota(i)\iota(j)})]\cdot \Fp((m_{\iota(i)}+k_i),(m_{\iota(i),\iota(j)})) \\
=\sum_\iota c_\iota \sum_{(k_{i})} [(k_i),(m_{\iota(i)}),(m_{\iota(i)\iota(j)})]
\sum_{\sigma\in \S_n} q_{\iota}(\sigma)\;\rFp((m_{\iota(\sigma^{-1}(i))}+k_{\sigma^{-1}(i)}),(m_{\iota(\sigma^{-1}(i))\iota(\sigma^{-1}(j))}))\\
=\sum_\iota\sum_{\sigma\in \S_n} q_{\iota\circ\sigma}(\sigma) c_{\iota\circ\sigma} \sum_{(k_{i})} [(k_i),(m_{\iota(\sigma(i))}),(m_{\iota(\sigma(i))\iota(\sigma(j))})]
)\cdot\rFp((m_{\iota(i)}+k_{\sigma^{-1}(i)}),(m_{\iota(i)\iota(j)}))\\
=\sum_\iota\sum_{(k_{i})} [(k_{i}),(m_{\iota(i)}),(m_{\iota(i)\iota(j)})]
\cdot\left(\sum_{\sigma\in \S_n} q_{\iota\circ\sigma}(\sigma) c_{\iota\circ\sigma} \right)\rFp((m_{\iota(i)}+k_{i}),(m_{\iota(i)\iota(j)})).
\end{gather*}
Assume that the integral $\rFp((m_i),(m_{ij}))$ converges at the given set of parameters, then by the assumed Nichols algebra relation the bracket vanishes and
\[\left(\sum_{\sigma\in \S_n} q_{\iota\circ\sigma}(\sigma) c_{\iota\circ\sigma} \right)\rFp((m_{\iota(i)}+k_{i}),(m_{\iota(i)\iota(j)}))=0.\]
As already stated, this is the case for subpolar parameters $(m_{ij})$. Otherwise we are interested in the nonzero result after analytic continuation. As discussed, only the symmetrizations over all permutations of the index set enter in this formula. In particular for the Nichols algebra relations to hold, it suffices that
\[\left(\sum_{\sigma\in \S_n} q_{\iota\circ\sigma}(\sigma) c_{\iota\circ\sigma} \right)\rFp(\iota(1)\cdots\iota(n))^{{\rm sym}}=0,\]
where we abbreviate the following functions in $(k_i)$, $\big(m_{\iota(i)}\big)$, $\big(m_{\iota(i)\iota(j)}\big)$ by
\begin{gather*}
\Fp(\iota(1)\cdots\iota(n)):=\Fp\big(\big(m_{\iota(i)}+k_{i}\big),\big(m_{\iota(i)\iota(j)}\big)\big),\\
\rFp(\iota(1)\cdots\iota(n)):=\rFp\big(\big(m_{\iota(i)}+k_{i}\big),\big(m_{\iota(i)\iota(j)}\big)\big).
\end{gather*}
A different way to put the previous computation would be writing the quantum symmetrizer formula in the new notation and in a formal basis of the tensor algebra $x_{\iota(1)}\cdots x_{\iota(n)}$ as
\begin{gather*}
\sum_\iota  \Fp(\iota(1)\cdots\iota(n))\cdot (x_{\iota(1)} \cdots x_{\iota(n)})^{{\rm sym}}\\
\qquad {} = \sum_{\sigma\in\S_n} q_\iota(\sigma)
\Fp\big(\iota\big(\sigma^{-1}(1)\big)\cdots\iota\big(\sigma^{-1}(n)\big)\big)^{{\rm sym}}
\cdot x_{\iota(\sigma^{-1}(1))} \cdots \big(x_{\iota(\sigma^{-1}(n))}\big).
\end{gather*}
This concludes the overview of the proof of Theorem~\ref{NArelations}.

\subsection{Central charge}
A realization in the sense of our article provides a set of elements in a Euclidean vector space $a_1,\dots a_r\in \C^r$, the corresponding screening operators of the Heisenberg algebra $\mathcal{H}^r$, and their algebra relations. We now also wish to fix an action of the Virasoro algebra on $\mathcal{H}^r$. As discussed in \cite{FL17}, it is usually desirable to choose the Virasoro structure, where all conformal weights $h(a_i)=1$, which gives a unique choice for $Q$. This implies for local screenings that the screenings commute with the Virasoro algebra action, but for nonlocal screenings the implication is more subtle: Their action on the vacuum modules commutes with the Virasoro action, and in general we expect that suitable powers of nonlocal screenings commute with the Virasoro action, just like in the Felder complex \cite{Fel89}. Regardless, we now fix the Virasoro action such that $h(v_i)=1$ holds.

\begin{prop}\label{central_charge} For the Heisenberg algebra, there is a family of Virasoro structures pa\-ra\-met\-ri\-zed by the choice of an element $Q\in \C^r$, called background charge~{\rm \cite{FF84}}. There is a uni\-que~$Q$ solving
	\begin{equation*}
	h(a_i)=\frac{1}{2}(a_i,a_i) -(a_i,Q)=1, \qquad i=1, \dots, n.
	\end{equation*}
	The central charge of the system will be \[c = r - 12(Q, Q).\]
	In particular for rank $2$, we have as in~{\rm \cite{Sem14}} the explicit formula
	\begin{equation}\label{centralcharge}
	c = 2-3\frac{|a_1(m_{22}-2)-a_2(m_{11}-2)|^2}{m_{11}m_{22}-m_{12}^2}.
	\end{equation}
\end{prop}

\section{Analytical continuation of screening relations}\label{section4}

A product of screening operators does not necessarily converge beyond subpolar~$(m_{ij})$. However, we can attempt to analytically continue the functions $\Fp$, $\rFp$, $\Sel$ to the subpolar region. Our main interest is in which regions of parameters Theorem~\ref{NArelations} holds or which extensions of Nichols algebras we find.

We consider linear combinations of monomials $\zem_{a_1}\cdots \zem_{a_n}$ of $n$ screening operators, acting on the $\mathcal{H}^r$-module $\mathcal{H}_{v}$ generated by ${\rm e}^v$. We set $m_i := (a_i, v)$ and $m_{ij} := (a_i, a_j)$ for $1 \leq i, j \leq n$.

\begin{prob} For each $\iota\colon \{1,\dots,n\}\to\{1,\dots,r\}$ and each $(m_{ij})$ attached to a finite-dimen\-sio\-nal Nichols algebra in this article, find the full analytic continuation and poles of the functions $\Sel(\iota(1)\cdots \iota(n))= \Sel((m_{\iota(i)}+k_i),(m_{\iota(i)\iota(j)}))$.
\end{prob}
For a realization of Cartan type $m_{ij}=(\alpha_i,\alpha_j)_\g m$ with $m\in \mathbb{Q}$ as in Definition~\ref{mCartan}, a linear combination of these integrals is the so-called $\g$-Selberg integral and can be expressed as a~product of Gamma functions. This existence of such an formula in certain cases is the Mukhin--Varchenko conjecture~\cite{MV00} and a general $\g$-Selberg integral formula was proven for~$A_2$ in~\cite{TV03} and for~$A_n$ in~\cite{War09}. The last source also contains in Theorem~6.1 a version of Kadell's integral with Jack polynomials necessary for the case~$(k_i)\neq (0)$. It would be tempting to use these results to get at least in the case~$A_n$ a full analytic continuation of all monomials, and to check other relations, such as the truncation relation of non-simple roots and the additional relations for $q=-1$.
\begin{prob}
Does there exist a Selberg integral formula in the sense of {\rm \cite{MV00,TV03,War09}} attached to any Nichols algebra root system for $\Sel((m_{\iota(i)}+k_i),(m_{\iota(i)\iota(j)}))$ with the parameters~$(m_{ij})$ obtained in the present article?
\end{prob}

We start with a toy case.

\subsection{Commutativity relations}\label{sec_commutativity}
\newcommand{\Beta}{\mathsf{B}}
For $n=2$ the subpolarity condition reads $m_{12}>-1$. From the formulas in \cite[Example~5.21]{Len17} we immediately obtain the following analytic continuations, where $\Beta(x,y)=\frac{\Gamma(x)\Gamma(y)}{\Gamma(x+y)}$ is the Euler Beta function, and $\Gamma(x)$ is the Gamma function, which is meromorphic on $\C$ with simple poles at $x\in-\N_0$:
\begin{gather*}
\Sel(m_1,m_2,0,0,m_{12})
=\frac{1}{2+m_1+m_2+m_{12}} \Beta(m_2+1,m_{12}+1).
\end{gather*}
This function has poles at $m_2,m_{12}\in -\N$ and at $m_1+m_2+m_{12}=-2$:
\begin{gather*}
\rFp(m_1,m_2,m_{12})=\frac{1}{(2\pi\i)^2}\big(1-{\rm e}^{2\pi\i m_2}\big)
	 \frac{\Beta(m_2+1,m_{12}+1)}{m_1+m_2+m_{12}+2}\\
\hphantom{\rFp(m_1,m_2,m_{12})=}{}
-\frac{1}{(2\pi\i)^2}{\rm e}^{2\pi\i m_2+\pi\i m_{12}}\big(1-{\rm e}^{2\pi\i m_1}\big)
	 \frac{\Beta(m_1+1,m_{12}+1)}{m_1+m_2+m_{12}+2}.
\end{gather*}
This function has poles at most at $m_{12}\in -\N$. The poles at $m_i\in-\N$ are removed by the exponential prefactor, and the pole at $m_1+m_2+m_{12}=-2$ is removed by an equality of the two summands at these values, visible after applying the Euler reflection formula
\begin{gather*}
\Gamma(z)\Gamma(1-z)=\frac{\pi}{\sin(\pi z)}\\
\hphantom{\Gamma(z)\Gamma(1-z)}{}
=\frac{1}{(2\pi\i)^2}\big({-}2\i {\rm e}^{\pi\i m_2}\big)\sin(\pi m_2) \frac{\Gamma(m_{12}+1)\Gamma(m_2+1)\Gamma(m_2+m_{12}+2)^{-1}}{m_1+m_2+m_{12}+2}\\
\hphantom{\Gamma(z)\Gamma(1-z)=}{}
-\frac{1}{(2\pi\i)^2}{\rm e}^{2\pi\i m_2+\pi\i m_{12}}\big({-}2\i {\rm e}^{\pi\i m_1}\big)\sin(\pi (m_1+m_{12}+1))\\
\hphantom{\Gamma(z)\Gamma(1-z)=}{}
\times\frac{\Gamma(m_{12}+1)\Gamma(-m_1)^{-1}\Gamma(-m_1-m_{12}-1)}{m_1+m_2+m_{12}+2}.
\end{gather*}
We now turn finally to
\begin{gather*}
\Fp(m_1,m_2,m_{12}) =\frac{{\rm e}^{2\pi\i m_2}-1}{2\pi\i}\frac{{\rm e}^{2\pi\i m_1+2\pi\i m_{12}}-1}{2\pi\i}
	 \frac{1}{m_1+m_2+m_{12}+2} \\
\hphantom{\Fp(m_1,m_2,m_{12}) =}{}
 \times \left(\Beta(m_2+1,m_{12}+1)+\frac{\sin\pi m_1}{\sin\pi(m_1+m_{12})}\Beta(m_1+1,m_{12}+1)\right).
\end{gather*}
This function has poles at a subset of $m_{12}\in -\N$, depending on $m_1$, $m_2$. We rearrange the sum to make the quantum symmetrizer formula
\[ \Fp(m_1,m_2,m_{12}) =\rFp(m_1,m_2,m_{12})+{\rm e}^{\pi\i m_{12}}  \rFp(m_2,m_1,m_{12})
\] visible:
\begin{gather*}
	 =\frac{1}{(2\pi\i)^2}\big(1-{\rm e}^{2\pi\i m_2}\big)
	 \frac{\Beta(m_2+1,m_{12}+1)}{m_1+m_2+m_{12}+2}\\
	 \quad{} -\frac{1}{(2\pi\i)^2}{\rm e}^{2\pi\i m_2+\pi\i m_{12}}\big(1-{\rm e}^{2\pi\i m_1}\big)
	 \frac{\Beta(m_1+1,m_{12}+1)}{m_1+m_2+m_{12}+2}\\
	 \quad{} +{\rm e}^{\pi\i m_{12}}\cdot \frac{1}{(2\pi\i)^2}\big(1-{\rm e}^{2\pi\i m_1}\big)
	 \frac{\Beta(m_1+1,m_{12}+1)}{m_1+m_2+m_{12}+2}\\
	 \quad{} -{\rm e}^{\pi\i m_{12}}\cdot \frac{1}{(2\pi\i)^2}{\rm e}^{2\pi\i m_1+\pi\i m_{12}}\big(1-{\rm e}^{2\pi\i m_2}\big)
	 \frac{\Beta(m_2+1,m_{12}+1)}{m_1+m_2+m_{12}+2}.
	 \end{gather*}
\begin{cor}\label{cor_product}
	The product of two screenings $\zem_1\zem_2$ can be analytically continued to parameters with $m_{12}\not\in-\N$ $($or further, depending on $m_1$, $m_2)$.
\end{cor}

In the tensor algebra $\C\langle x_1,x_2 \rangle$ we have
\[\sha_q[x_1,x_2]_q=\sha_q(x_1x_2-q_{12}\;x_2x_1)=(1-q_{12}q_{21})x_1x_2,\]
which is zero for $q_{12}q_{21}=1$. In the Nichols algebra, this gives in this case the quadratic relation $[x_1,x_2]_q=0$ in Example~\ref{exm_quadraticRelation}, which is the quantum Serre relation for $c_{12}=0$ in Lemma~\ref{lm_Serre}.

The corresponding expression of screening operators depends by the results in the last section on the analyticity and zeroes of
\begin{align*}
\Fp(12)-q_{12}\Fp(21)
&=(1-q_{12}q_{21})\rFp(12).
\end{align*}

\begin{cor}The expression $[\zem_1,\zem_2]_q$ can be analytically continued to all values of $m_1$,~$m_2$, $m_{12}$ and it vanishes for all~$m_{12}$ except if $m_{12}\in \Z$ $($where $q_{12}q_{21}={\rm e}^{2\pi\i m_{12}}=1)$ and $m_{12}< 0$ $($where~$\rFp$ has poles$)$.
\end{cor}	
\begin{oss}
	In the realizations $(m_{ij})$ derived below condition~(\ref{cond7}A) requires $m_{ij}=0$ whenever $q_{ij}q_{ji}=1$. Hence in these cases, commutativity always continues to hold for screening operators.
\end{oss}
\begin{ese}\label{ex_notcommuative}
	A typical example where the $q$-commutator $[\zem_1,\zem_2]_q$ is nonzero in contrast to the Nichols algebra is the standard case of two local screenings, i.e., $m_1,m_2,m_{12}\in\Z$. In this case the standard (anti-)commutator formula
	\begin{gather*}
	[\zem_1,\zem_2]_\pm =\operatorname{Res}\big(\Y\big(\zem_1{\rm e}^{\alpha_2},z\big)\big),
	\end{gather*}
	where $\zem_1{\rm e}^{\alpha_2}$ is zero unless $\Y({\rm e}^{\alpha_1},z){\rm e}^{\alpha_2}$ has a pole at $z=0$, which is the case for $m_{12}=(\alpha_1,\alpha_2)\allowbreak <0$. For example, in the case $m_{12}=-1$ we have\footnote{This formula appears in the Frenkel--Kac--Segal construction and shows that the local screenings in the lattice vertex algebra of the root lattice of a Lie algebra $\g$ generates $U(\g)$. See \cite[Section~5.6]{Kac98} or in our context \cite[Section~3.6]{Len17}.}
	\begin{gather*}
	[\zem_1,\zem_2]_+ =\operatorname{Res}\big(\Y\big({\rm e}^{\alpha_1+\alpha_2},z\big)\big)=:\zem_{1+2}.
	\end{gather*}
\end{ese}	
	
\subsection{Truncation relations}\label{sec_truncation}
Let $\iota\colon \{1,\dots,n\}\to \{1\}$ be the constant function $\iota(i)=1$, then $(m_{\iota(i)\iota(j)})=m_{11}$ and $m_{\iota(i)}=m_1$ for all $i$,~$j$. Consider
\[\rFp(\underbrace{11\cdots 1}_n)=\rFp(\iota(1)\iota(2)\cdots \iota(n))=\rFp((m_{\iota(i)}+k_i),(m_{\iota(i)\iota(j)})).\]
\begin{lem}[$n$-th power]\label{lm_rFpTrunc}
	The function $\rFp(11\cdots 1)$ extends analytically to
	\[m_{11}\not\in -\N\frac{2}{k},\qquad k=2,\dots, n\]
	with at most $($depending on $m_1)$ double poles for $k=2,\dots, n-1$ and a simple pole for $k=n$.
\end{lem}
Before we prove this, we state some consequences. In the tensor algebra $\C[x]$ we have by Lemma~\ref{lm_truncation}
\[
\sha_q x^n=\left(\prod_{k=1}^n\frac{1-q^k}{1-q} \right)x^n,
\]
which is zero for $q^k=1$ for $k=2,\dots, n$ and thus gives the truncation relation $x^n=0$ for $n\geq \operatorname{ord}(q)$.

The corresponding expression of screening operators depends by the results of the previous section on the analyticity and zeroes of the function
\[
\Fp(11\cdots1)^{{\rm sym}}=\left(\prod_{k=1}^n\frac{1-q^k}{1-q} \right)\rFp(11\cdots1)^{{\rm sym}}.
\]

\begin{cor}\label{cor_trun}
The power of a screening $\zem_1^n$ can be analytically continued to parameters with
\[m_{11}\not\in -\N\frac{2}{k},\qquad k=2,\dots, n-1\]
and it vanishes for all $m_{11}$ fulfilling the condition $m_{11}\in -\Z\frac{2}{k}$ for $k=2,\dots, n$ $($where the associated $q$-polynomial is zero$)$ and $m_{11}\geq 0$ $($where $\rFp$ has no poles$)$.
\end{cor}
Note that by this result the truncation relation $\zem^n$, $n=\operatorname{ord}\big({\rm e}^{\pi\i m_{11}}\big)$ can always be analytically continued, but higher powers $\zem^n$, $n>\operatorname{ord}\big({\rm e}^{\pi\i m_{11}}\big)$ may not converge. This is related to the product not converging, see Corollary~\ref{cor_product}.

A typical example where $\zem^n$ is nonzero in contrast to the Nichols algebra is the case $m_{ij}=-\frac{2}{p}$, where one can compute that $\zem^p$ is proportional to a single screening associated to ${\rm e}^{p\alpha}$. Since $(p\alpha,\alpha)\in 2\Z$ we call this screening \emph{local screening} since ${\rm e}^{p\alpha}$ is an element in a suitably defined integral lattice vertex algebra. This computation appears in the Liouville model and higher rank analogy, where the new screening is a long screening, see \cite[Section~6.4]{Len17}.

\begin{proof}[Proof of Lemma \ref{lm_rFpTrunc}]
	 In our special situation with equal $m_{ij}=m_{11}$ and $m_i=m_1$ we use the factorization in Lemma 6.3 cit.~loc.
	\begin{gather*}
	\rFp((m_{i}+k_i),(m_{ij}))
	:=\prod_{s=0}^{n-1} \big( \big({\rm e}^{\pi {\rm i} m_{11}}\big)^s {\rm e}^{2\pi {\rm i} m_{1}} -1 \big) \cdot
	\Sel((m_{11}+k_i),(0),(m_{11}))
	\end{gather*}
	and evaluate $\Sel$ for $(k_i)=(0)$ with the Selberg integral formula \cite{Sel44}
	\begin{gather*}
	\Sel((a-1),(b-1),(2c))  = \prod_{k=0}^{n-1} \frac{\Gamma(a + kc)\Gamma(b + kc)\Gamma(1 + (k+1)c)}{\Gamma(a + b + (n+k-1)c)\Gamma(1 + c)}
	\end{gather*}
	for $a-1=m_{1}$, $b-1=0$, $2c=m_{11}$, and for arbitrary $(k_i)$ with the following refinement called Kadell's integral~\cite{Kad97} and \cite[Section~VI.10]{Mac95}:
	For any partition in at most $n$ parts
	$l=(l_1,\dots,l_n)$, $l_1\geq \cdots \geq l_n\geq 0$ and for $P_{l}^{(1/c)}(z_1,\dots, z_n)$ the associated Jack~polynomial, we have
	\begin{gather*}
	 \idotsint\limits_{1> z_1>\dots >z_{n}> 0}
	P_{l}^{(1/c)}(z_1,\dots, z_n)\prod_i z_i^{a-1}\prod_i (1-z_i)^{b-1}\prod_{i<j} (z_i-z_j)^{2c} \;\d z_1\cdots \d z_{n}\\
	\qquad{} =\prod_{1\leq i<j\leq n}\frac{\Gamma((j-i+1)c+l_i-l_j)}
	{\Gamma((j-i)c+l_i-l_j)}
	\prod_{k=0}^{n-1}
	\frac{\Gamma(a+kc+l_{n-k})\Gamma(b+kc)}
	{\Gamma(a+b+(n+k-1)c+l_{n-k})}.
	\end{gather*}
We want to study the analyticity of $\rFp$. Since the Jack polynomials form a basis of the symmetric polynomials, it is sufficient to study analyticity of Kadell's integral for all partitions $l$. All possible poles come from the Gamma-functions in the numerators.
	\begin{itemize}\itemsep=0pt
		\item Consider $\Gamma(a+kc+l_{n-k})$ with simple poles at most at
		\[a+ k c \in -\N_0, \qquad k=0,\dots,n-1.\]
		These cancel with the zeroes in the prefactor $\big(\big({\rm e}^{\pi {\rm i} m_{\alpha\alpha}}\big)^s {\rm e}^{2\pi {\rm i} m_{\alpha\lambda}} -1\big)$, so at these values $\rFp$ is analytic. We remark however, that these exceptional non-zero values of $\tilde{F}$ give rise to reflection operators~\cite[Section~6.3]{Len17}.
		
		We remark that the poles in $\Sel$ depending on $m_i$ must always disappear in $\rFp$, because $\rFp$ is a contour integral avoiding the singularity at $z_i=0$.
		\item Consider $\Gamma(b+kc)$ for $b=1$, which gives simple poles for
		\begin{gather*}
		kc  \in -\N,  \qquad k=2,\dots,n-1,
		\end{gather*}
		while the possible pole for $c\in -\N$ for $k=1$ cancels with the many zeroes coming from the denominators $\Gamma((j-i)c+l_i-l_j)$ for $j-i=1$ with $l_i-l_j\leq 0$.
		\item Consider $\Gamma((j-i+1)c+l_i-l_j)$ for all $i$, $j$ with fixed $k:=j-i+1$, which are $n-k+1$ terms that can together produce a pole up to this order for
		\begin{gather*}
		(l_i-l_j) + kc \in -\N_0, \qquad k=2,\dots,n.
		\end{gather*}
		On the other hand consider the $n-k-1$ terms in the denominator $\Gamma((j'-i')c+l_{i'}-l_{j'})$ for $j'=j$, $i'=i-1$ for $i\neq 1$. Since $l_{i'}-l_{j'}\geq l_i-l_j$ these zeroes cancel with the possible poles above, leaving only a possible single pole at $l_1-l_2+kc\in -\N_0$. A possible pole for $c=0$ cancels again with the many zeroes coming from the denominators $\Gamma((j-i)c+l_i-l_j)$ for $j-i=1$ with $l_i-l_j \leq 0$ (note that previously we have only used these denominators for $j-i>1$ and for $c\in -\N$), leaving possible poles
		\begin{gather*}
		kc \in -\N, \qquad k=2,\dots,n.
		\end{gather*}
	\end{itemize} 	
	This proves the assertion.
\end{proof}

We remark that our calculation can be compared with the special case $(k_i)=(0)$ described by the easier Selberg integral above. Conversely, in the case of general $(k_i)$ we integrate additional positive integer powers~$z_i$, so it is reasonable that the case $(k_i)=(0)$ already exhibits the maximal set of poles. However, in presence of Kadell's integral we chose to be explicit here.

\subsection{Analytical continuation by recursion}\label{sec_recursion}

We introduce a way of analytically continuing $\tilde{F}$ without explicitly computing it. We will use this in the Serre relations, but also it seems suitable also for more complicated relations.

A trivial recursion relation is obtained by splitting off a factor
\[
(z_k-z_l)^{m_{kl}}=(z_k-z_l)(z_k-z_l)^{m_{kl}-1}
\] in the integrand, multiplying $(z_k-z_l)$ out and joining additional $z_k$ resp.~$z_l$ to the powers $z_k^{m_k}$ resp $z_l^{m_l}$. Thus
\begin{gather*}
	\rFp((m_i),(m_{ij}))
	=\rFp((m_i+\delta_{i,k}),(m_{ij}-\delta_{i,k}\delta_{j,l}))
	-\rFp((m_i+\delta_{i,l}),(m_{ij}-\delta_{i,k}\delta_{j,l})).
\end{gather*}
In particular $\rFp$ can be analytically continued from some set of $(m_{ij})$ to positive translates $(m_{ij})+(N_{ij})$ for $N_{ij}\in\N_0$. Note that if $(m_{ij})$ is subpolar, so is $(m_{ij})+(N_{ij})$. It is more difficult to continue to smaller~$m_{ij}$, where poles appear:

\begin{lem}\label{lm_recursion}Assume $k<l$ are fixed indices, then we have the following
recursion relation
\begin{gather*}\rFp((m_i),(m_{ij})) =\sum_{k'\neq k,l} -\frac{m_{kk'}}{1+m_{kl}}\operatorname{sgn}(k'-k) \rFp((m_i),(m_{ij}+\delta_{i,k}\delta_{j,l}-\delta_{i,k}\delta_{j,k'}))\\
\hphantom{\rFp((m_i),(m_{ij})) =}{}
-\frac{m_{k}}{1+m_{kl}} \rFp(m_i-\delta_{i,l}),(m_{ij}+\delta_{i,k}\delta_{j,l})),
 \end{gather*}
where $\operatorname{sgn}(x)=\pm1$ denotes the sign of $x$ and we
restrict ourselves for definiteness, e.g., to the set of $(m_{ij})$ with
all $m_{ij}>0$. For $k=1$ the formula gets an additional summand
\[
\frac{1}{1+m_{kl}}\int_1^{{\rm e}^{2\pi\i }}\d z_2\cdots
\int_{z_{n-1}}^{{\rm e}^{2\pi\i }}\d z_{n-1} \prod_{1\neq i}
z_i^{m_i}(1-z_i)^{m_{1i}}\prod_{1\neq i<j}(z_i - z_j)^{m_{ij}}.
\]
\end{lem}

When we write $\int_{1}^{{\rm e}^{2\pi\i t}}$ we mean the
corresponding arc on the unit circle, lifted to the universal covering
of $\C\backslash\{0\}$.

\begin{proof}We integrate by parts with respect to the variable $z_k$ and with
respect to the factor $(z_k-z_l)^{m_{kl}}$:
\begin{gather*}
\rFp((m_i),(m_{ij}))
=\int_1^{{\rm e}^{2\pi\i }}\d z_1 \int_{z_1}^{{\rm e}^{2\pi\i }}\d z_2\cdots\int_{z_{n-1}}^{{\rm e}^{2\pi\i }}\d z_n\prod_i
z_i^{m_i}\prod_{i<j}(z_i - z_j)^{m_{ij}}\\
=\frac{1}{1+m_{kl}}\int_1^{{\rm e}^{2\pi\i }}\d z_1\cdots 
\int_{z_{k-2}}^{{\rm e}^{2\pi\i }}\d z_{k-1}\\
\quad {} \times\Bigg[\int_{z_k}^{{\rm e}^{2\pi\i }}\d z_{k+1}\cdots\int_{z_{n-1}}^{{\rm e}^{2\pi\i }}\d z_n \prod_i
z_i^{m_i}\prod_{i<j}(z_i - z_j)^{m_{ij}+\delta_{i,k}\delta_{j,l}}\Bigg]_{z_k=z_{k-1}}^{z_k={\rm e}^{2\pi\i }}\\
\quad{} -\sum_{k'\neq k,l} \frac{m_{kk'}}{1+m_{kl}}\operatorname{sgn}(k'-k)
\int_1^{{\rm e}^{2\pi\i }}\d z_1 \cdots \int_{z_{n-1}}^{{\rm e}^{2\pi\i }}\d z_n \prod_i
z_i^{m_i}\prod_{i<j}(z_i -z_j)^{m_{ij}+\delta_{i,k}\delta_{j,l}-\delta_{i,k}\delta_{j,k'}} \\
\quad{}-\frac{m_{k}}{1+m_{kl}}\int_1^{{\rm e}^{2\pi\i }}\d z_1\cdots
\int_{z_{n-1}}^{{\rm e}^{2\pi\i }}\d z_n \prod_i z_i^{m_i-\delta_{i,k}}\prod_{i<j}(z_i -z_j)^{m_{ij}+\delta_{i,k}\delta_{j,l}},
\end{gather*}
and the derivative of $\int_{z_k}^{{\rm e}^{2\pi\i}}$ vanishes due to $(z_k-z_{k+1})^{m_{k,k+1}}$. Also it should be silently implied that if $k'<k$ then the roles of $i$, $j$ are switched where appropriate. Now the boundary term in the square bracket vanishes for $z_k={\rm e}^{2\pi\i}$, because then the integration domain for ${\rm d}z_{k+1}$ is $\big\{{\rm e}^{2\pi\i }\big\}$ and for $k\neq 1$ it vanishes for $z_k=z_{k-1}$, again because of a term $(z_{k-1}-z_k)^{m_{k-1,k}}$. For $k=1$ the boundary term for $z_1=1$ produces an additional term
\[
\frac{1}{1+m_{kl}}\int_1^{{\rm e}^{2\pi\i }}\d z_2\cdots \int_{z_{n-1}}^{{\rm e}^{2\pi\i }}\d z_{n-1} \prod_{1\neq i}
z_i^{m_i}(1-z_i)^{m_{1i}}\prod_{1\neq i<j}(z_i - z_j)^{m_{ij}}.\tag*{\qed}
\]\renewcommand{\qed}{}
\end{proof}

This is still not satisfactory, because $m_{kl}$ is increased at the
expense of $m_{kk'}$ being decreased, but at least the overall sum of
the $m_{ij}$ does not decrease. We now demonstrate how this can be
applied in the case $n=3$. Inductively applying the previous lemma $n$
times to $(k,l)=(2,3)$ gives
\begin{gather*}
\rFp(m_1,m_2,m_3,m_{12},m_{13},m_{23})\\
\qquad =\sum_{i+j=n}(-1)^j\frac{{m_{13} \choose i}{m_{3} \choose j}}{{m_{23}+n
\choose n}}\rFp(m_1,m_2-j,m_3,m_{12},m_{13}-i,m_{23}+n).
\end{gather*}

Now suppose $(m_{ij})$ fulfills $m_{12}+m_{13}+m_{23}>-2$ (subpolarity for $I=\{1,2,3\}$) and $m_{12}>-1$ (subpolarity for $I=\{1,2\}$) but not necessarily $m_{23}>-1$ (subpolarity for $I=\{2,3\}$). Then for $n$ sufficient large such that $m_{23}+n>-1$ all parameters in the previous formula are subpolar on intervals.

\begin{cor}\label{cor_RecursionContinuation3} The previous recursion formula gives an analytic continuation of $\rFp((m_i){,}(m_{ij}))$ for $n=3$ to all parameters $(m_{ij})$ with $m_{12}+m_{13}+m_{23}>-2$ and $m_{12}>-1$, with at most simple poles at $m_{23}\in-\N$.
\end{cor}

\subsection[Serre relations for Cartan matrix entry $-1$]{Serre relations for Cartan matrix entry $\boldsymbol{-1}$}\label{sec_Serre}

We want to study the Serre relation $[\zem_1,[\zem_1,\zem_2]_q]_q=0$, which holds by Lemma~\ref{lm_Serre} for $q_{12}q_{21}=q_{11}^{-1}$ or $q_{11}^2=1$. We start again by setting up an analytical continuation. Let, in more generality
\[\iota_N\colon \ \{0,\dots,n\}\to \{1,2\},\qquad
\iota(N)=2,\qquad \iota(i)=1, \qquad i\neq N.\]

\begin{lem}[$N=0$] There exists an analytic continuation of
\[\Sel(211\cdots 1)=\Sel((m_{\iota_0(i)}+k_i),(0), (m_{\iota_0(i)\iota_0(j)}))\]
as a product of Gamma functions, with at most simple poles at
	\begin{alignat*}{3}
		& m_1+ k \frac{m_{11}}{2}  \in -\N, \qquad && k=0,\dots,n-1,& \\
		& m_{12}+k\frac{m_{11}}{2}  \in -\N,  \qquad&&  k=0,\dots,n-1, & \\
		& k\frac{m_{11}}{2}  \in -\N,  \qquad && k=2,\dots,n, &
		\end{alignat*}
	and at $n+ nm_1+m_2 + nm_{12} + {n\choose 2} m_{11} \in -\N$.
\end{lem}
	
\begin{proof}By the substitution $\tilde{z}_i=z_i/z_0$ we can isolate the first variable in the Selberg integral and integrate it
\begin{gather*}
 \Sel((m_{\iota(i)}+k_i),(0), (m_{\iota(i)\iota(j)}))_{0\leq i<j\leq n}\\
\qquad{} =\idotsint\limits_{1> z_0>\cdots >z_n> 0} z_0^{m_2+k_0} \prod_{i=1}^n z_i^{m_1+k_i} \prod_{j=2}^{n} (z_1 - z_j)^{m_{12}} \prod\limits_{1 \leq i < j \leq n} (z_i - z_j)^{m_{11}}  \d z_0 \d z_1 \cdots \d z_n\\
\qquad{} =\int_0^1 z_0^{n+(m_2+k_0) + \sum_j m_{12} + \sum_{i<j} m_{11} + \sum_i (m_1+k_i)} \d z_0\\
\qquad\quad{}\times  \idotsint\limits_{1> z_1>\cdots >z_n > 0} \prod_{i=1}^n \tilde{z}_i^{m_1+k_i} \prod_{j=1}^{n} (1 - \tilde{z}_j)^{m_{12}} \prod_{2 \leq i < j \leq n} (\tilde{z}_i - \tilde{z}_j)^{m_{11}}  \d\tilde{z}_1 \cdots \d\tilde{z}_n\\
\qquad{} =\frac{1}{1+\tilde{m}}
  \Sel((m_{\iota(i)}+k_i),(0), (m_{\iota(i)\iota(j)}))_{0\leq i<j\leq n},
\end{gather*} 	
where $\tilde{m}:=n+ nm_1+m_2 + nm_{12} + {n\choose 2} m_{11}+\sum_{i=0}^n k_i$. Hence analytic continuation is again possible using Kadell's integral, as in the previous section and we thus analyze the poles in
\begin{gather*}
 \frac{1}{1+\tilde{m}}\idotsint\displaylimits_{1> z_1>\dots >z_{n}> 0}
P_{l}^{(1/c)}(z_1,\dots, z_n)\prod_i z_i^{a-1}\prod_i (1-z_i)^{b-1}\prod_{i<j} (z_i-z_j)^{2c} \;\d z_1\cdots \d z_{n}\\
\qquad{} =\frac{1}{1+\tilde{m}}\prod_{1\leq i<j\leq n}\frac{\Gamma((j-i+1)c+l_i-l_j)}
{\Gamma((j-i)c+l_i-l_j)}
\prod_{k=0}^{n-1}
\frac{\Gamma(a+kc+l_{n-k})\Gamma(b+kc)}
{\Gamma(a+b+(n+k-1)c+l_{n-k})}
\end{gather*}
for $a-1=m_1$, $b-1=m_{12}$, $2c=m_{11}$ as in the last section
	\begin{itemize}\itemsep=0pt
	\item The fraction produces a simple pole at $\tilde{m}=-1$.
	\item $\Gamma(a+kc+l_{n-k})$ again produces possible poles at
	\begin{gather*}
	a+ k c  \in -\N_0,  \qquad k=0,\dots,n-1.
	\end{gather*}
	\item $\Gamma(b+kc)$ is changed and produces possible simple poles at
	\begin{gather*}
	b+kc  \in -\N_0,  \qquad k=0,\dots,n-1
	\end{gather*}
	\item $\Gamma((j-i+1)c+l_i-l_j)$ for all $i$, $j$ with fixed $k=j-i+1$ again produces a possible simple pole at
	\begin{gather*}
	kc  \in -\N,  \qquad k=2,\dots,n\tag*{\qed}
	\end{gather*}
\end{itemize}\renewcommand{\qed}{} 	
\end{proof}

\begin{ese}For $n=2$ the formula in the proof reads
	\begin{gather*}
	 \Sel(211)
	=\frac{1}{1+\tilde{m}}\frac{\Gamma(m_{11}+l_1-l_2)}{\Gamma(m_{11}/2+l_1-l_2)} \\
\hphantom{\Sel(211)=}{}
\times \frac{\Gamma(1+m_1)+l_2)\Gamma(1+m_{12})}
	{\Gamma(2+m_1+m_{12}+m_{11}/2+l_2)}
	\frac{\Gamma(1+m_1+m_{11}/2+l_1)\Gamma(1+m_{12}+m_{11}/2)}
	{\Gamma(2+m_1+m_{12}+nm_{11}/2+l_2)}.
	\end{gather*}
	This has simple poles at most at $m_{11}/2\in -\frac{1}{2}-\N_0$ and $m_{12}\in -\N$ and $m_{12}+m_{11}/2\in -\N$, and other poles involving $m_1$. We discuss two examples that are relevant later-on, see condition~\eqref{cond7}
	\begin{itemize}\itemsep=0pt
		\item For $m_{11}=2m$, $m_{12}=-m$ with $m\not\in \Z$ we have simple poles at \[m_{11}/2=m\in -\frac{1}{2}-\N_0.\]
		\item For $m_{11}=1$, $m_{12}=-m$ with $m\not\in \Z$ we have simple poles at
		\begin{gather*}
	m_{12}+1/2=-m+1/2  \in -\N\quad
			\Leftrightarrow\quad m \in\frac{1}{2}+\N.
		\end{gather*}
	\end{itemize}
\end{ese}

Unfortunately we have found no version of the previous lemma if the distinguished index is not the first index in~$\Sel$. For a specific choice of~$(m_{ij})$ associated to~$A_2$ we can use the Selberg integral associated to the Lie algebra~$A_2$ \cite{TV03,War09}, which gives an expression in terms of Gamma functions for a specific linear combination:

\begin{lem}[$n=2$]
 	Let $m_{11}=2m$ and $m_{12}=-m$. Consider the linear combination of Selberg integrals
	\begin{align*}
	S&=-\Sel(112)-\frac{\sin(\pi m)}{\sin(2\pi m)}\Sel(121).
	\end{align*}
	Then $S$ has an analytic continuation with at most the following poles
	\begin{itemize}\itemsep=0pt
		\item at $2m\in\pm\N$ a simple pole,
		\item simple poles at $m_1,m_2\in-\N$ and $m_1+m_2-m\in-2-\N_0$.
	\end{itemize}
\end{lem}
\begin{proof}
	The expression $S$ is a special case of \cite[Theorem~1.2]{War09}: In the notation of cit.~ļoc.~we set
\[k_1=1,\qquad k_2=1,\qquad \alpha=1,\qquad \beta_i=1+m_i,\qquad \gamma=\frac{m_{11}}{2}=m,\]
	then after a slight change in variables $z_i=1-t_i$, which reverses the order of the integration variables and causes the factor~$(-1)$, then the expression in question reads
	\begin{gather*}
	S =
	\frac{\Gamma(\beta_1)}{\Gamma(1+\beta_1-2\gamma)}
	\frac{\Gamma(\beta_1+\beta_2-\gamma)}{\Gamma(1+\beta_1+\beta_2)}
	\frac{\Gamma(\beta_2)}{\Gamma(1+\beta_2+\gamma)}\\
\hphantom{S=}{}\times
	\frac{\Gamma(1-2\gamma)\Gamma(\gamma)}{\Gamma(\gamma)}
	\frac{\Gamma(1)\Gamma(\gamma)}{\Gamma(\gamma)}
	\frac{\Gamma(1+\gamma)\Gamma(2\gamma)}{\Gamma(\gamma)}\\
\hphantom{S}{}
=
	\frac{\Gamma(1+m_1)}{\Gamma(2+m_1-m)}
	\frac{\Gamma(2+m_1+m_2-m)}{\Gamma(3+m_1+m_2)}
	\frac{\Gamma(1+m_2)}{\Gamma(2+m_2+m)}
	\Gamma(1-2m)(m/2)\Gamma(2m).
	\end{gather*}
	We again read off the possible poles from the Gamma-functions: In the first numerator $1+m_1\in-\N_0$, in the second numerator $2+m_1+m_2-m_{12}\in-\N_0$, in the third numerator $1+m_2\in-\N_0$, and in the rest $2m\in\pm\N$.
\end{proof}

The last assertion does not determine $\rFp(121)$ and $\rFp(112)$ individually, but if we assume $m<1$, then $\rFp(121)$ is subpolar on intervals, so it is analytic. Hence in this case we can combine the information on the poles of $\Sel(211)$ and $S$, $\Sel(112)$ to determine all possible poles of~$\rFp(112)$ and $\rFp(211)$ (and~$\rFp(121)$, but this is again analytic by subpolarity on intervals). Again, since~$\rFp$ is an integral over a contour not on~$0$, the poles depending on $m_1$, $m_2$ will disappear in the linear combinations of~$\Sel$.

\begin{cor}[later-on the case $A_2$, $m<1$]\label{cor_A2Method1}
	Suppose $m_{11}=2m$, $m_{12}=-m$ and suppose in addition $m<1$, then the previous two lemmas and the analyticity of $\Sel(121)$ give an analytic continuation of $\rFp(211)$, $\rFp(112)$ to all values of $m<1$ with at most simple poles at $2m\in-\N$.
\end{cor}
Observe that $\Sel(211)$ only contributes poles for negative half-integer $m$, which apparently are related to the power~$\zem_1^2$. On the other hand $\Sel(112)$ also contributes poles for positive half-integer~$m$, which apparently are related to the product~$\zem_1\zem_2$. The possible pole at $m=\frac{1}{2}$ must in fact be analytic, since subpolarity holds~-- it appears that all possible poles $m\in\frac{1}{2}+\N_0$ are artifacts from the sine-fraction appearing in the definition of~$S$.

As a second method (with different range of applicability), we now consider directly for~$\rFp(112)$ the analytic continuation via recursion in Corollary~\ref{cor_RecursionContinuation3}. Note that this could also be applied to the previous case, but then the additional boundary term in Lemma~\ref{lm_recursion} would come into consideration.
\begin{cor}[later-on the cases $A_2$, $m>-\frac{1}{2}$ and $A(1|0)$, $m<\frac{3}{2}$]\label{cor_A2Method2}
	We have the following analytic continuations:
\begin{itemize}\itemsep=0pt
	\item For $m_{11}=2m$, $m_{12}=-m$ with $m>-\frac{1}{2}$ we have \[m_{11}=2m>-1,\qquad
	m_{11}+m_{12}+m_{12}=2m-m-m=0>-2,\]
	so Corollary~{\rm \ref{cor_RecursionContinuation3}} applies and we have an analytic continuation with at most simple poles at $m\in\N$.
	\item For $m_{11}=1$, $m_{12}=-m$ with $m<\frac{3}{2}$ we have
	\[m_{11}=1>-1,\qquad
	m_{11}+m_{12}+m_{12}=1-m-m>-2,\]
	so Corollary~{\rm \ref{cor_RecursionContinuation3}} applies and we have an analytic continuation with at most a simple pole at $m=1$. Note that for $m<1$ the expression is subpolar.
\end{itemize}
\end{cor}

We are now ready to check the Serre relations with $c_{ij}=-1$. We calculate in the tensor algebra explicitly:
\begin{gather*}
[x_1,[x_1,x_2]_q]_q=x_1x_1x_2-q_{12}(q_{11}+1)x_1x_2x_1+q_{11}q_{12}^2 x_2x_1x_1, \\
 \sha [x_1,[x_1,x_2]_q]_q =x_1x_1x_2 -q_{12}(q_{11}+1)  x_1x_2x_1+q_{11}q_{12}^2  x_2x_1x_1 \\
\hphantom{\sha [x_1,[x_1,x_2]_q]_q =}{}
+q_{11}x_1x_1x_2 -q_{12}(q_{11}+1) q_{12} x_2x_1x_1+ q_{11}q_{12}^2  q_{21}x_1x_2x_1+q_{12}x_1x_2x_1 \\
\hphantom{\sha [x_1,[x_1,x_2]_q]_q =}{}
-q_{12}(q_{11}+1)  q_{21}x_1x_1x_2 +q_{11}q_{12}^2  q_{11}x_2x_1x_1+q_{12}q_{12}x_2x_1x_1 \\
\hphantom{\sha [x_1,[x_1,x_2]_q]_q =}{}
-q_{12}(q_{11}+1) q_{21}q_{11}x_1x_1x_2+q_{11}q_{12}^2  q_{11}q_{21}x_1x_2x_1+q_{11}q_{12}x_1x_2x_1 \\
\hphantom{\sha [x_1,[x_1,x_2]_q]_q =}{}
-q_{12}(q_{11}+1) q_{12}q_{11} x_2x_1x_1+ q_{11}q_{12}^2  q_{21}q_{21} x_1x_1x_2+q_{11}q_{12}q_{12}x_2x_1x_1 \\
\hphantom{\sha [x_1,[x_1,x_2]_q]_q =}{}
-q_{12}(q_{11}+1) q_{12}q_{11}q_{21} x_1x_2x_1+ q_{11}q_{12}^2  q_{21}q_{21}q_{11} x_1x_1x_2 \\
\hphantom{\sha [x_1,[x_1,x_2]_q]_q }{}
=(q_{11}+1)(q_{12}q_{21}-1)(q_{11} q_{12}q_{21}-1)x_1x_1x_2.
\end{gather*}
In the factorization we see the three possibilities for this relation to hold, namely: $q$-truncation, $q$-Cartan with $c_{ij}=0$ and $q$-Cartan with $c_{ij}=-1$.

Now by the results in Section~\ref{smallTheorems} this translates to an equation
\begin{gather*}
 \Fp(112)^{{\rm sym}}-q_{12}(q_{11}+1)\Fp(121)^{{\rm sym}}+q_{11}q_{12}^2 \Fp(211)^{{\rm sym}}\\
\qquad{} =(q_{11}+1)(q_{12}q_{21}-1)(q_{11}\;q_{12}q_{21}-1)\rFp(112)^{{\rm sym}}.
\end{gather*}
Combining the analytic continuations in Corollaries~\ref{cor_A2Method1} and~\ref{cor_A2Method2} we thus find:

\begin{cor}[later-on the case $A_2$]\label{cor_A2}
The function $\rFp(112)$ can be analytically continued to parameters $m_{11}=2m$, $m_{12}=-m$ with at most simple poles for $m\in-\frac{1}{2}\N$ and $m\in\N$. For $m\in-\frac{1}{2}-\N_0$ we have $m_{ii}\in -1-2\N_0$ and $q_{ii}=-1$, for which the factor above $(q_{11}+1)(q_{12}q_{21}\allowbreak -1)(q_{11} q_{12}q_{21}-1)$ has a double zero, and thus its product with $\rFp(112)$ is zero. As a consequence, for $m\not\in \Z$ two screening operators $\zem_1$, $\zem_2$ with such $(m_{ij})$ fulfill the quantum Serre relation
\[[\zem_1,[\zem_1,\zem_2]_q]_q=0.\]
\end{cor}

\begin{cor}[later-on the case $A(1,0)$]\label{cor_A10}
For $m<\frac{3}{2}$ the function $\rFp(112)$ can be analytically continued to parameters $m_{11}=1$, $m_{12}=-m$ with at most a simple pole at $m=1$. As a~consequence, for $m\not\in \Z$ two screening operators $\zem_1$, $\zem_2$ with such~$(m_{ij})$ fulfill the quantum Serre relation
	\[[\zem_1,[\zem_1,\zem_2]_q]_q=0.\]
\end{cor}

Here we make no assertion if the Serre relation holds for $m>\frac{3}{2}$.

\section{Formulation of the classification problem}\label{section5}
\begin{defin}\label{Mm} Let $\Lambda$ be a lattice of rank $r$, basis $\{a_1, \dots, a_r\}$, bilinear form $(\,,\,)$ and Cartan matrix $(c_{ij})$ and let $m_{ij}:=(a_i, a_j)$. Given a braiding matrix $(q_{ij})$, we say that the lattice $\Lambda$ described by $(m_{ij})$ \emph{realizes} $(q_{ij})$ iff
\begin{itemize}\itemsep=0pt
	\item the matrix elements $m_{ij}$ exponentiate to the matrix elements $q_{ij}$:
	\[{\rm e}^{{\rm i}\pi m_{ij}} = q_{ij},\]
	\item for each pair $(i,j)$ one of the following conditions hold: \begin{align}\label{cond7}
\text{A:} \ 2m_{ij} = c_{ij} m_{ii} \qquad \text{or} \qquad \text{B:} \ (1-c_{ij})m_{ii} = 2
\end{align}
(this condition originates in \cite{Sem14}, but not the next condition),
	\item a base change by precomposing with a Weyl reflection $\s_k(\alpha_i)=\alpha_i-c_{ki}\alpha_k$
	returns a new matrix $\r_k(m_{ij})$ with
	\[\r_k(m_{ij})_{ij}=(\s_k(\alpha_i),\s_k(\alpha_j))\]
	and as suggested by notation $\r_k(m_{ij})$ trivially fulfills condition (1) with the reflected braiding matrix $\r_k(q_{ij})$.
 We now demand in addition that all iterated reflections $\r_{k_1}\cdots\r_{k_n}(m_{ij})$ again satisfy condition~(\ref{cond7}).
\end{itemize}
\end{defin}

In complete analogy to Definition \ref{def_qCartanRoot} we define
\begin{defin}
	Let $(m_{ij})$ be a realization.
	\begin{enumerate}\itemsep=0pt
		\item A pair $(i,j)$ is called $m$-\emph{Cartan} if $m_{ij}$ satisfies (\ref{cond7}A), and $m$-\emph{truncation} if it satis\-fies~(\ref{cond7}B).
		\item A root $\alpha$ is called $m$-\emph{Cartan} resp.~$m$-\emph{truncation} if in any Weyl chamber containing $\alpha=\alpha_i$ as a simple root and any neighbour $j\sim i$ the pair $(i,j)$ is $m$-{Cartan} resp.~$m$-{truncation}.
		\item A root $\alpha$ is called \emph{only $m$-Cartan} if in any Weyl chamber containing $\alpha=\alpha_i$ as a simple root and any neighbour $j\sim i$ the pair $(i,j)$ is only $m$-{Cartan}.
	\end{enumerate}
\end{defin}	
We visualize a realization by a diagram decorated on the top as a $q$-diagram by $q_{ii}$, $q_{ij}q_{ji}$ and on the bottom by the realization $m_{ii}$, $m_{ij}+m_{ji}=2m_{ij}$, e.g.,
\[
	\begin{tikzpicture}
	\draw (0,0)--(1.6,0) (1.8,0)--(3.4,0);
	\draw (-0.1,0) circle[radius=0.1cm] node[anchor=south]{$ q_{11}$} node[anchor=north]{$ m_{11}$}
	(1.7,0) circle[radius=0.1cm] node[anchor=south]{$ q_{22} $}node[anchor=north]{$m_{22}$}
	(3.5,0) circle[radius=0.1cm] node[anchor=south]{$ q_{33} $}node[anchor=north]{$ m_{33}$};
	\draw (0.9,0) node[anchor=south]{$ q_{12}q_{21}$}node[anchor=north]{$ 2m_{12}$};
	\draw (2.7,0) node[anchor=south]{$ q_{23}q_{32}$} node[anchor=north]{$ 2m_{23}$};
	\end{tikzpicture}
\]
and no line connecting the vertices $1$, $3$ if $m_{13}=0$. Note that if $q_{ij}q_{ji}=1$ and thus $c_{ij}=0$, then in any realization the pair $(i,j)$ cannot be $m$-trunctation, must be $m$-Cartan, and thus $m_{ij}=0$.

We now summarize the strategy by which these notions allow to construct and classify realizations:
\begin{itemize}\itemsep=0pt
	\item Any pair $(i,j)$ that is $m$-Cartan resp.~$m$-truncation is surely $q$-Cartan resp.~$q$-truncation. Conversely, if a pair $(i,j)$ is only $q$-Cartan resp.~only $q$-truncation, than in any realization $(m_{ij})$ the pair has to be (only) $m$-Cartan resp.~(only) $m$-Cartan.
	\item If a simple root $\alpha_i$ is not $q$-Cartan, then $m_{ii}=\frac{2}{\operatorname{ord}(q_ii)}$. If a simple root $\alpha$ is $q$-Cartan, then (\ref{cond7}A) determines $m_{ij}$ for all $j\neq i$, hence one may proceed inductively. Moreover, for a subsystem generated by only $q$-Cartan simple roots $\alpha_{i},\alpha_{i'},\dots$ this determines $(m_{ij})$ for this subsystem to be a rescaled root lattice (see Lemma~\ref{cartanclassification}).
	\item If a root $\alpha$ is $q$-Cartan and $q$-truncation, then there might exist different realizations, depending on the assumption that it is $m$-Cartan or $m$-truncation. In more difficult cases one may argue with individual pairs $(i,j)$. See as example Remark~\ref{moresolsl21}.
	\item Conversely, suppose we are given a possible realization $(m_{ij})$ and want to proof that this is indeed a realization. If $m_{ii}=\frac{2}{\operatorname{ord}(q_ii)}$ and we already know $\alpha_i$ is $q$-truncation (in particular if it is fermionic), then~(\ref{cond7}B) holds. Otherwise we check condition (\ref{cond7}A). Then we have to go through all reflections and check the same conditions, possibly fixing additionally open parameters. As for the $q$-diagrams, the following fact reduces greatly the amount of computation:
\end{itemize}

\begin{prop}\label{prop_CartanRefl}
If the pairs $(k,i)$ and $(k,j)$ are $m$-Cartan $($for example if the root $\alpha_k$ is $m$-Cartan$)$ then we have ${\r_k(m_{ij})=(m_{ij})}$.
\end{prop}
\begin{proof}
	\begin{align*}
	\r_k(m_{ij})_{ij}&=(\s_k\alpha_i,\s_k\alpha_j)=(\alpha_i-c_{ki}\alpha_k,\alpha_j-c_{kj}\alpha_k)\\
	&=m_{ij}-c_{kj}m_{ik}- c_{ki}m_{kj}+c_{ki}c_{kj}m_{kk}\\
	&=m_{ij}-c_{kj}\cdot \frac{1}{2}c_{ki}m_{kk}- c_{ki}\cdot\frac{1}{2}c_{kj}m_{kk}+c_{ki}c_{kj}m_{kk}=m_{ij}.\tag*{\qed}
	\end{align*}\renewcommand{\qed}{}
\end{proof}

\begin{itemize}
\item For a reflection on a $m$-truncation root the result is less predictable. We derive in Propositions~\ref{prop_rank2reflection1} and~\ref{prop_rank2reflection2} a sufficient criterion, essentially from performing one reflection and give conditions when this is $m$-truncation of $m$-Cartan. In practice, these conditions are sufficient to fix~$(m_{ij})$ uniquely.
\item In particular for the Nichols algebras that do not follow into families, we proceed by induction: The realizations $(m_{ij})$ restrict on a subset of simple roots to a realization of the respective $q$-subdiagram.
\end{itemize}

\begin{ese}\label{sl(2|1)}
We now show an example of this procedure. We consider row 3 of Table~1 in~\cite{Hecklist}, described by the braiding matrices:
\[ \big(q_{ij}^\mathrm{I}\big)= \begin{bmatrix}
 q^2 & q^{-1}\vspace{1mm}\\
 q^{-1} & -1
\end{bmatrix}, \qquad \big(q_{ij}^\mathrm{II}\big)= \begin{bmatrix}
 -1 & -q\vspace{1mm}\\
 -q & -1
\end{bmatrix}, \]
and corresponding diagrams:
\[
\begin{tikzpicture}
	\draw (0,0)--(1.6,0);
	\draw (-0.1,0) circle[radius=0.1cm] node[anchor=south]{$ q^2$}
			(1.7,0) circle[radius=0.1cm] node[anchor=south]{$-1$};
	\draw (0.8,0) node[anchor=south]{$ q^{-2}$} ;
	\draw (3.6,0)--(5.2,0);
	\draw (3.5,0) circle[radius=0.1cm] node[anchor=south]{$ -1$}
			(5.3,0) circle[radius=0.1cm] node[anchor=south]{$-1$};
	\draw (4.4,0) node[anchor=south]{$ q^2$};
\node at (0.8,-0.4) {I};
\node at (4.4,-0.4) {II};
\end{tikzpicture}
\]
with $q \in \C^\times$, $q^2 \neq \pm 1$, simple roots $\{\alpha_1, \alpha_2\}$ and $\{\alpha_{12}, \alpha_2\}$ respectively, and a unique associated Cartan matrix \[\big(c_{ij}^\mathrm{I}\big) = \big(c_{ij}^\mathrm{II}\big) = \begin{bmatrix}
 \hphantom{-}2 & -1\\
 -1 & \hphantom{-}2
\end{bmatrix}. \]
This describes the Lie superalgebra $\mathfrak{sl}(2|1)$.
The set of positive roots is $\{\alpha_1, \alpha_2, \alpha_{12}\}$ where $\alpha_1$ is a $q$-Cartan root and $\alpha_2$, $\alpha_{12}$ are fermionic, thus $q$-truncation. All pairs $(i,j)$ are only $q$-Cartan or only $q$-truncation, since we assumed $q^2\neq -1$.

\begin{prop}\label{sl21mij} For this braiding and its reflections, the following are all realizations $(m_{ij})$:
\[\big(m_{ij}^{\mathrm{I}}\big) = \begin{bmatrix}
 2m & -m\\
 -m & 1
\end{bmatrix}, \qquad \big(m_{ij}^{\mathrm{II}}\big)= \begin{bmatrix}
 1 & -1+m\\
 -1+m & 1
\end{bmatrix}\]
for all $m= \frac{p'}{p}\in \mathbb{Q}$ with $(p',p)=1$ such that ${\rm e}^{{\rm i} \pi m} = q$.
\end{prop}
\begin{proof}
We check that condition (\ref{cond7}B) is satisfied for the pair $(2,1)$ in $\mathrm{I}$ and both pairs $(1,2)$ and $(2,1)$ in~$\mathrm{II}$:
\begin{gather*}
m_{22}^{\mathrm{I}} = \frac{2}{1-c_{21}^{\mathrm{I}}} = 1, \qquad
m_{11}^{\mathrm{II}} = \frac{2}{1-c_{12}^{\mathrm{II}}} = 1, \qquad
m_{22}^{\mathrm{II}} = \frac{2}{1-c_{21}^{\mathrm{II}}} = 1,
\end{gather*}
while condition (\ref{cond7}A) is satisfied for the pair $(1,2)$ in $\mathrm{I}$:
\begin{gather*}
m_{11}^{\mathrm{I}} = \frac{2m_{12}^{\mathrm{I}}}{c_{12}^{\mathrm{I}}} = 2m.
\end{gather*}
The reflection $\r_1$ preserves $\big(q_{ij}^{\mathrm{I}}\big)$ as well as $\big(m_{ij}^{\mathrm{I}}\big)$ by Proposition~\ref{prop_CartanRefl}, because~$\alpha_1$ is $m$-Cartan. We check that the other reflection
		\[\r_2(\alpha_1)=\alpha_1+\alpha_2,\qquad \r_2(\alpha_2)=-\alpha_2\]
maps $\big(m_{ij}^{\mathrm{I}}\big)$ to our choice of $\big(m_{ij}^{\mathrm{II}}\big)$:
\begin{gather*}
\r_2\big(m_{ij}^{\mathrm{I}}\big)
 = \begin{bmatrix}
(\alpha_1+\alpha_2,\alpha_1+\alpha_2) & (\alpha_1+\alpha_2,-\alpha_2)\\
\\
(-\alpha_2,\alpha_1+\alpha_2) & (-\alpha_2,-\alpha_2)
\end{bmatrix}
 =\begin{bmatrix}
2m-2m+1 &  m-1 \\
m-1 & 1
\end{bmatrix}
=\big(m_{ij}^{\mathrm{II}}\big).
\end{gather*}

We now prove conversely, that this is the only realization. Thereby we will find the typical dichotomy of arguments that we will also find in later cases:

Since $(1,2)$ in $\mathrm{I}$ is only $q$-Cartan, the realization has to be $m$-Cartan, which fixes
\begin{gather*}
m_{12}^{\mathrm{I}}=\frac{c_{12}^{\mathrm{I}}}{2}m_{11}^{\mathrm{I}}=-m
\end{gather*}
once we have defined $m$ via
\begin{gather*}
m_{11}^{\mathrm{I}}=:2m.
\end{gather*}
Since $(2,1)$ in $\mathrm{II}$ is only $q$-truncation, the realization has to be $m$-truncation, so
\begin{gather*}
m_{22}^{\mathrm{I}} = \frac{2}{1-c_{21}^{\mathrm{I}}} = 1.
\end{gather*}
Fixing $m_{ij}^{\mathrm{I}}$ already fixes the bilinear form, so for proving uniqueness of the realization, this is sufficient.
\end{proof}

\begin{oss}\label{moresolsl21} We now discuss the case $q^2=-1$ in the previous $q$-diagram, which is excluded in row~3. It appears in row 2, which corresponds to $\mathfrak{sl}_3$:
\[
	\begin{tikzpicture}
	\draw (0,0)--(1.6,0);
	\draw (-0.1,0) circle[radius=0.1cm] node[anchor=south]{$ -1$}
	(1.7,0) circle[radius=0.1cm] node[anchor=south]{$-1$};
	\draw (0.8,0) node[anchor=south]{$ -1 $} ;
	\end{tikzpicture}
\]	
	
 But in some sense, this diagram can be viewed as special case of both the $q$-diagrams appearing in rows~2 and~3, and there is an exceptional isomorphism $u_q(\mathfrak{sl}_3)^+\cong u_q(\mathfrak{sl}(2|1))$ for $q^2=-1$. We also find that in this case all pairs $i\sim j$ are both $q$-truncation and $q$-Cartan, which opens the possibility for different realizations $(m_{ij})$ in which different pairs are $m$-truncation or $m$-Cartan. Indeed we find two solutions, and they are special cases of the two different realizations we can construct for the diagrams in rows~2 and~3 respectively:

For this diagram, we find precisely the following two families of realizations $(m_{ij})$, each parametrized by odd $p',p'' \in \Z$,
\begin{itemize}\itemsep=0pt
	\item if we assume $(1,2)$ and $(2,1)$ in $\mathrm{I}$ to be $m$-truncation, we find the unique realizations
	\[
	 \big(m_{ij}^{\mathrm{I}}\big)= \begin{bmatrix}
 1 & -\frac{p''}{2}\\
 -\frac{p''}{2} & 1
\end{bmatrix},\]
	\item if we assume $(1,2)$ in $\mathrm{I}$ to be $m$-truncation and $(2,1)$ in $\mathrm{I}$ to be $m$-Cartan (or vice versa), we find the realizations
 \[
	\big(m_{ij}^{\mathrm{II}}\big)= \begin{bmatrix}
 1 & -\frac{p'}{2}\\
 -\frac{p'}{2} & p'
\end{bmatrix}.
\]
Reflection $\r_1$ maps this realizations $\big(m_{ij}^{\mathrm{II}}\big)$ to the previous $\big(m_{ij}^{\mathrm{I}}\big)$ and back, for $p''=2-p'$. Thus we have essentially one solution where $(q_{ij})$ is invariant under reflection, but $(m_{ij})$ is different in different Weyl chambers.
This solution should be viewed as an instance (or limiting case) of the solution in Proposition \ref{sl21mij}, which corresponds to $\mathfrak{sl}(2|1)$ at $q^2=-1$.
\item if we assume both pairs $(1,2)$ and $(2,1)$ in $\mathrm{I}$ to $m$-Cartan, the unique family of realizations is given by
 \[(m_{ij})= \begin{bmatrix}
 p' & -\frac{p'}{2}\\
 -\frac{p'}{2} & p'
\end{bmatrix}, \]
which is the rescaled root lattice of $\mathfrak{sl}_3$. By Proposition~\ref{prop_CartanRefl} all reflections leave $(m_{ij})$ invariant. This solution should be viewed as an instance of the generic solution for Cartan type in the next section, here $\mathfrak{sl}_3$, $q^2=-1$.
\end{itemize}
\end{oss}
\end{ese}

\section{Cartan type}\label{Cartan}
\subsection[$q$ diagram]{$\boldsymbol{q}$ diagram}\label{Cartanconse}

Let $\mathfrak{g}$ be a simple Lie algebra with simple roots $\alpha_1, \dots, \alpha_n$ and Killing form in the standard normalization $(\alpha_{i}, \alpha_{j})_\mathfrak{g} \in \{-3, -2, -1, 0, 2, 4, 6\}$. Let $q \in \C^\times$ be a primitive $\ell$-th root of unity with $\ell \in \Z$ and let $\operatorname{ord}\big(q^2\big) >d$ with $d$ half length of the long roots. Define a braiding mat\-rix~$(q_{ij})$ by
\[q_{ij} =q^{(\alpha_{i}, \alpha_{j})_\mathfrak{g}}.\]
The finite-dimensional Nichols algebra $\mathcal{B}(q_{ij})$ \cite{AS10, Lusz93, Rosso98} is called of \emph{Cartan type} We have that:
\begin{itemize}\itemsep=0pt
	\item $(q_{ij})$ is invariant under reflections $\r_k$,
	\item the Weyl groupoid is the Weyl group associated to $\mathfrak{g}$,
	\item the set of positive roots is the set of roots associated to $\mathfrak{g}$,
	\item the Cartan matrix $c_{ij}$ is the Cartan matrix for $\mathfrak{g}$.
\end{itemize}

\subsection[Construction of $(m_{ij})$]{Construction of $\boldsymbol{(m_{ij})}$}

\begin{defin}\label{mCartan}
Given $m \in \mathbb{Q}$ we define
\[m_{ij} := (\alpha_i, \alpha_j)_\mathfrak{g} m.\]
\end{defin}
Hence, the lattice $\Lambda$ the root lattice of $\mathfrak{g}$ rescaled by $m$.
\begin{oss} If we choose relatively prime integers $k$, $\ell$, such that $\frac{m}{2}= \frac{k}{\ell}$, then $q={\rm e}^{\pi\i m}$ is a primitive $\ell$-th root of unity. In literature on the logarithmic Kazhdan Lusztig conjecture, e.g.,~\cite{FT10, FGST06a}, one usually sets $m= \frac{p'}{p}$, so $q^2$ is a primitive $p$-th root of unity and $\ell=2p,p$ depending on the parity of~$p'$.
\end{oss}
\begin{lem}
The matrix $(m_{ij})$ realizes the braiding $(q_{ij})$ for all reflections, and every simple root is $m$-Cartan.
\end{lem}
\begin{proof}
Condition (\ref{cond7}) asks
\begin{alignat*}{3}
&2m_{ij} = c_{ij} m_{ii}\qquad && \text{or} \qquad (1-c_{ij})m_{ii} = 2,& \\
&2m_{ji} = c_{ji} m_{jj}\qquad && \text{or} \qquad (1-c_{ji})m_{jj} = 2.&
\end{alignat*} But from the last point of enumeration in Section~\ref{Cartanconse} we have $c_{ij} = \frac{2(\alpha_i, \alpha_j)_\mathfrak{g}}{(\alpha_i, \alpha_i)_\mathfrak{g}}$. Hence
\[
c_{ij}m_{ii} = \frac{2(\alpha_i, \alpha_j)_\mathfrak{g}}{(\alpha_i, \alpha_i)}(\alpha_i, \alpha_i)_\mathfrak{g} m = 2 m_{ij},
\] which is (\ref{cond7}A), saying that the roots are $m$-Cartan.
Since by Proposition~\ref{prop_CartanRefl}, any reflection on such a root leaves the $(m_{ij})$ invariant, condition~(\ref{cond7}) holds also after reflections.
\end{proof}

\begin{prop}\label{prop_CartanTrunc} If $\ell_i > 1- c_{ij}$ with $\ell_i := \operatorname{ord}\big(q^{2d_i}\big)$, then the pair $(i,j)$ is not $q$-truncation. Hence the only Nichols algebras of Cartan type, where roots $\alpha$ are both $q$-Cartan and $q$-truncation are
\begin{center}\renewcommand{\arraystretch}{1.2}
	\begin{tabular}{lll}
	$\mathfrak{g}$ & $q$ & both $q$-Cartan and $q$-truncation\\
	 \hline
	 $A_n$ & $q^2=-1$ & all roots \\
	 $B_n$, $C_n$, $F_4$ & $q^4=-1$ & long roots \\
	 $B_n$, $C_n$, $F_4$ & $q^2\in\mathbb{G}_3$ & short roots \\
	 $G_2$ & $q^6=-1$ & long roots \\
	 $G_2$ & $q^2\in \mathbb{G}_4$ & short roots \\
	\end{tabular}
\end{center}
\end{prop}
\begin{proof} Assume that $(i,j)$ is $m$-truncation $(1-c_{ij})m_{ii} = 2$, this implies: $q_{ii}^{(1-c_{ij})}={\rm e}^{{\rm i}\pi m_{ii}(1-c_{ij})}={\rm e}^{{\rm i}\pi \cdot 2}=1$. But $\operatorname{ord}(q_{ii})= \operatorname{ord}\big(q^{2d_i}\big) > 1- c_{ij}$ and we find a contradiction. The second claim follows by writing out these equations for long and short roots, and discarding the cases excluded by the conditions on the $q$-diagram ($q^2\not\in\mathbb{G}_2$ for $B_n$, $C_n$, $F_4$ and $q^2\not\in \mathbb{G}_2,\mathbb{G}_3$ for~$G_2$).
\end{proof}

\begin{lem}\label{cartanclassification}
If $(m_{ij})$ is a realization, such that all pairs $(i,j)$ are $m$-Cartan, then $(m_{ij})$ is the realization in Definition~{\rm \ref{mCartan}} for some~$m$.
\end{lem}
\begin{proof} If we fix $m_{ii}=:2m$ for some short root $\alpha_i$, then $m_{ij}$ for all $j$ is fixed by condition~(\ref{cond7}A) and so is $m_{jj}$ by the same condition with reversed indices. Hence up to rescaling there is a unique solution $(m_{ij})$ and Definition~\ref{mCartan} is such a solution.
\end{proof}
\begin{cor}The realization in Definition~{\rm \ref{mCartan}} is unique for all Nichols algebras of Cartan type except for the cases listed in Proposition~{\rm \ref{prop_CartanTrunc}}.
\end{cor}

As a counterexample, consider the case $\mathfrak{sl}_3$ and $\ell=2p=4$, where all pairs are both $m$-Cartan and $m$-truncation. Indeed, we have in this case two realizations, as discussed in Remark~\ref{moresolsl21}, corresponding to~$\mathfrak{sl}_3$ and to~$\mathfrak{sl}(2|1)$. In the latter realization, not all pairs are $m$-Cartan, despite being $q$-Cartan.

\subsection{Central charge}
Recall $\{a_1, \dots, a_r\}$ as basis of $\Lambda$ with $m_{ij} = (a_i,a_j)$ and $r=\operatorname{rank}(\mathfrak{g})$.

\begin{prop} The central charge of the system is
\begin{equation*}
c= \operatorname{rank}(\mathfrak{g}) -12\left(\frac{1}{r}\big|\rho^\vee\big|_\mathfrak{g}^2-2\big\langle\rho, \rho^\vee\big\rangle_\mathfrak{g}+r |\rho |_\mathfrak{g}^2\right),
\end{equation*}
where $\rho$ is half the sum of all positive roots.
\end{prop}
\begin{proof}
The central charge is
\[
c = r - 12(Q, Q),
\] where $Q=\sum_j c_ja_j$ is the unique combination such that for every $i$
\begin{gather*}
\frac{1}{2}(a_i,a_i) -(a_i,Q) =1,\\
\frac{1}{2}(a_i,a_i) -\sum_j c_j(a_i,a_j)_\Lambda =1.
\end{gather*}
Rewriting $a_i = -\sqrt{m}\alpha_i$, with $\alpha_i$ root of $\mathfrak{g}$, this set of equations bring us to
\begin{equation*}
Q= \sqrt{\frac{1}{m}}\rho^\vee-\sqrt{m}\rho
\end{equation*}
that on turn gives the central charge as in the statement.
\end{proof}

\begin{oss}
The central charge matches with the one of the affine Lie algebra $\hat{\mathfrak{g}}_k$ at level $k+h^\vee= \frac{1}{r}$ as in \cite{Arak07}. Conjecturally, the kernel of screening contains the Hamiltonian reduction of $\hat{\mathfrak{g}}_k$.
\end{oss}
\begin{oss} For rank $2$ and $m=p'/p$, the central charge is
\[ c= 1-3 \frac{(2p'-2p)^2}{2pp'} = 13 -6\frac{p}{p'}- 6\frac{p'}{p},
\] which is the central charge of the $(p,p')$ Virasoro model, see, e.g.,~\cite{TW13}.
\end{oss}
\subsection{Algebra relations}\label{CartanNArel}

According to the results in Section \ref{sec_Serre} we restrict ourselves here to the case of simple-laced $\g=A_n,D_n,E_6,E_7,E_8$. By \cite[Sections~4.1, and~4.4 and~4.5]{AA17} a set of defining relations is
\begin{itemize}\itemsep=0pt
	\item The commutation relations resp.~Serre relations for Cartan matrix entry $c_{ij}=0$ are $[x_i,x_j]_q$ for $\i\not\sim j$.
	\item The Serre relations for Cartan matrix entry $c_{ij}=-1$ are $[x_i,[x_i,x_j]_q]_q=0$ for $i\sim j$ for $q^2\neq -1$ (for $q^2=-1$ the Serre relations are implied by $x_i^2=0$).
	\item The truncation relations of root vector are $x_\alpha^{\ell_\alpha}=0$ for any root $\alpha\in\Phi^+$ and $\ell_\alpha=\operatorname{ord}\big(q^{2}\big)$, where the root vector $x_\alpha$ is defined by repeated reflections using Lusztig's isomorphism.
	\item For $q^2=-1$ additional relations $[x_j,[x_i,[x_j,x_k]_q]_q]_q=0$ for any subsystem $\alpha_i$, $\alpha_j$, $\alpha_k$ of type $A_3$ (for $q^2\neq -1$ these relations are a consequence of the Serre relations).
\end{itemize}

We now consider the realization of Cartan type $m_{ij}=(\alpha_i,\alpha_j)_\g m$. By Corollary~\ref{cor_A2}, the quantum Serre relations can be analytically continued and hold for all values of~$m$. By Corollary~\ref{cor_trun} the truncation relations of simple root vectors can be analytically continued for all values of~$m$, but they only hold for $m\geq 0$. We make no assertion about the additional relation for $q^2=-1$. The Nichols algebra without truncations relations is the Borel part of the Kac--DeConcini--Procesi quantum group $U_q^\mathcal{K}(\g)$ resp.~the distinguished pre-Nichols algebra~\cite{Ang16}. In particular we have proven:
\begin{cor}In the realization $m_{ij}=(\alpha_i,\alpha_j)\,m$ of the braiding $q_{ij}={\rm e}^{\i\pi(\alpha_i,\alpha_j)_\g m}$ associated to a simply-laced Lie algebra $\g$ at $q={\rm e}^{{\rm i}\pi m}$, $m\not\in 2\Z$, the Nichols algebra relations hold for the corresponding screenings as follows:
	\begin{itemize}\itemsep=0pt
		\item For $0<m<1$ the parameters $(m_{ij})$ are subpolar and all relations hold. Differently spoken, the algebra of screenings is a surjective image of the Borel part of the small quantum group~$u_q(\g)$.
		\item For $m<0$ the Serre relations hold. The truncation relations of simple root vectors fail. Differently spoken, for $q^2\neq 1$ the algebra of screenings is a surjective image of the Borel part of the Kac--DeConcini--Procesi quantum group~$U_q^\mathcal{K}(\g)$.
		\item for $m>1$ the Serre relations and the truncation relations of simple root vectors hold.
	\end{itemize}
We would conjecture that for $m<0$ also the additional relation holds and for $q^2=-1$ holds, so that also in this case we get the Borel part of $U_q^\mathcal{K}(\g)$, and that for $m>1$ also the truncation relations of non-simple root vectors hold, so that also in this case for $q^2\neq -1$ we get the Borel part of $u_q(\g)$. We make no assertion about the additional relation for $q^2=-1$ for $m>1$.
	We would conjecture that all surjections above are in fact isomorphism.
\end{cor}

\subsection{Examples: Cartan type realizations in rank 2}
\subsubsection*{Heckenberger row 2 (Cartan type $\boldsymbol{A_2}$)}
This case of the list is described by the braiding diagram:
\[
\begin{tikzpicture}
	\draw (0,0)--(1.6,0);
	\draw (-0.1,0) circle[radius=0.1cm] node[anchor=south]{$ q^2$}
			(1.7,0) circle[radius=0.1cm] node[anchor=south]{$ q^2$};
	\draw (0.8,0) node[anchor=south]{$ q^{-2}$} ;
\end{tikzpicture}
\]
with $q \in \C$ $q^2 \neq 1$ and simple roots $\lbrace\alpha_1, \alpha_2\rbrace$. The set of positive roots is given by $\lbrace \alpha_1, \alpha_2, \alpha_{12}\rbrace$ with unique associate Cartan matrix:
\begin{gather*}
(c_{ij})= \begin{bmatrix}
 2 & -1\\
 -1 & 2
\end{bmatrix}.
\end{gather*}
Definition \ref{mCartan} gives the following realization
\begin{gather*}
(m_{ij})= \begin{bmatrix}
 2m & -m\\
 -m & 2m
\end{bmatrix},\\
	\mbox{\begin{tikzpicture}
	\draw (0,0)--(1.6,0);
	\draw (-0.1,0) circle[radius=0.1cm]
	node[anchor=south]{$ q^2$}
	node[anchor=north]{$ 2m$}
	(1.7,0) circle[radius=0.1cm]
	node[anchor=south]{$ q^2$}
	node[anchor=north]{$ 2m$};
	\draw (0.8,0)
	node[anchor=south]{$ q^{-2}$}
	node[anchor=north]{$ -2m$} ;
	\end{tikzpicture}}
\end{gather*}

For $q^2\neq-1$, these are all solutions by Lemma \ref{cartanclassification}. For $q^2=-1$ there is a second family of solutions associated to $A(1,0)$, as discussed in Remark \ref{moresolsl21}:
\begin{gather*}
\big(m_{ij}^{\mathrm{I}}\big)= \begin{bmatrix}
2m & -m\\
-m & 1
\end{bmatrix},
\qquad
\big(m_{ij}^{\mathrm{II}}\big)= \begin{bmatrix}
1 & -1+m\\
-1+m & 1
\end{bmatrix},
\\
\mbox{\begin{tikzpicture}
	\draw (0,0)--(1.6,0);
	\draw (-0.1,0) circle[radius=0.1cm]
	node[anchor=south]{$ -1$}
	node[anchor=north]{$ 2m$}
	(1.7,0) circle[radius=0.1cm]
	node[anchor=south]{$ -1$}
	node[anchor=north]{$ 1$};
	\draw (0.8,0)
	node[anchor=south]{$ -1$}
	node[anchor=north]{$ -2m$} ;
	\end{tikzpicture}
	\hspace{1.6cm}
	\begin{tikzpicture}
	\draw (0,0)--(1.8,0);
	\draw (-0.1,0) circle[radius=0.1cm]
	node[anchor=south]{$ -1$}
	node[anchor=north]{$ 1$}
	(1.9,0) circle[radius=0.1cm]
	node[anchor=south]{$ -1$}
	node[anchor=north]{$ 1$};
	\draw (0.9,0)
	node[anchor=south]{$ -1$}
	node[anchor=north]{$ -2+2m$} ;
	\end{tikzpicture}}
\end{gather*}
for $m=\frac{p'}{2}$ for $p'$ odd. The algebra relations for this realization are discussed in Section~\ref{sec_SuperNA}.

\subsubsection*{Heckenberger row 4 (Cartan type $\boldsymbol{B_2}$)}
This case of the list is described by the braiding diagram:
\[
\begin{tikzpicture}
	\draw (0,0)--(1.6,0);
	\draw (-0.1,0) circle[radius=0.1cm] node[anchor=south]{$ q^2$}
			(1.7,0) circle[radius=0.1cm] node[anchor=south]{$ q^4$};
	\draw (0.8,0) node[anchor=south]{$ q^{-4}$} ;
\end{tikzpicture}
\]
with $q \in \C$, $q^2 \neq \pm 1 $ and simple roots $\lbrace\alpha_1, \alpha_2\rbrace$. The set of positive roots is given by $\lbrace \alpha_1, \alpha_2, \alpha_{12}, \allowbreak \alpha_{112}\rbrace$ with unique associate Cartan matrix:
\[(c_{ij})= \begin{bmatrix}
 \hphantom{-}2 & -2\\
 -1 & \hphantom{-}2
\end{bmatrix}. \]
Definition \ref{mCartan} gives the following realization
\begin{gather*} (m_{ij})= \begin{bmatrix}
 2m & -2m\\
 -2m & 4m
\end{bmatrix},
\\
\mbox{\begin{tikzpicture}
	\draw (0,0)--(1.6,0);
	\draw (-0.1,0) circle[radius=0.1cm]
	node[anchor=south]{$ q^2$}
	node[anchor=north]{$ 2m$}
	(1.7,0) circle[radius=0.1cm]
	node[anchor=south]{$ q^4$}
	node[anchor=north]{$ 4m$};
	\draw (0.8,0)
	node[anchor=south]{$ q^{-4}$}
	node[anchor=north]{$ -4m$} ;
	\end{tikzpicture}}
\end{gather*}

If all pairs are $m$-Cartan, in particular if
$\operatorname{ord}\big(q^2\big)>3$, $\operatorname{ord}\big(q^4\big)>2$, then this is the unique realization by Lemma~\ref{cartanclassification}. We now discuss the other cases:
\begin{enumerate}\itemsep=0pt
	\item When $q^4=-1$, the pair $(2,1)$ is $q$-Cartan and $q$-truncation. We find an additional family of solutions where $(2,1)$ is $m$-truncation:
\begin{gather*}
\big(m_{ij}^{\mathrm{I}}\big)= \begin{bmatrix}
 2m& - 2m\\
 - 2m & 1
\end{bmatrix}, \qquad \big(m_{ij}^{\mathrm{II}}\big)= \begin{bmatrix}
 -2m+1 & 2m -1\\
 2m-1 & 1
\end{bmatrix},\\
\mbox{\begin{tikzpicture}
\draw (0,0)--(1.6,0);
\draw (-0.1,0) circle[radius=0.1cm]
node[anchor=south]{$ \pm i$}
node[anchor=north]{$ 2m$}
(1.7,0) circle[radius=0.1cm]
node[anchor=south]{$ -1$}
node[anchor=north]{$ 1$};
\draw (0.8,0)
node[anchor=south]{$ -1$}
node[anchor=north]{$ -4m$} ;
\end{tikzpicture}
	\hspace{1.6cm}
	\begin{tikzpicture}
\draw (0,0)--(2.2,0);
\draw (-0.1,0) circle[radius=0.1cm]
node[anchor=south]{$ \pm i$}
node[anchor=north]{$ -2m+1$}
(2.3,0) circle[radius=0.1cm]
node[anchor=south]{$ -1$}
node[anchor=north]{$ 1$};
\draw (1.1,0)
node[anchor=south]{$ -1$}
node[anchor=north]{$\;\;\;\; 4m-2$} ;
\end{tikzpicture}}
\end{gather*}
for $m= \frac{p'}{4}$ and $p'$ odd,
with simple roots I: $\lbrace\alpha_1, \alpha_2\rbrace$ and II: $\lbrace\alpha_{12}, -\alpha_2\rbrace$.

This can be interpreted as a limiting case of the Lie superalgebra $B(1,1)$ described in case Heckenberger row~5, which for this choice of $q^2$ has the same $q$-diagram.

\item When $q^2=\zeta\in \mathbb{G}_3$, the pair $(1,2)$ is $q$-Cartan and $q$-truncation. We find an additional family of solutions where $(1,2)$ is $m$-truncation:
\begin{gather*}
\big(m_{ij}^{\mathrm{I}}\big)= \begin{bmatrix}
 \frac{2}{3} & -2m\\
 -2m & 4m
\end{bmatrix}, \qquad \big(m_{ij}^{\mathrm{II}}\big)= \begin{bmatrix}
 \frac{2}{3} & -\frac{4}{3} + 2m\\
 -\frac{4}{3} + 2m & \frac{8}{3} - 4m
\end{bmatrix},\\
\mbox{\begin{tikzpicture}
	\draw (0,0)--(1.6,0);
	\draw (-0.1,0) circle[radius=0.1cm]
	node[anchor=south]{$ \zeta$}
	node[anchor=north]{$ \frac{2}{3}$}
	(1.7,0) circle[radius=0.1cm]
	node[anchor=south]{$ \zeta^2$}
	node[anchor=north]{$ 4m$};
	\draw (0.8,0)
	node[anchor=south]{$ \zeta$}
	node[anchor=north]{$ -4m$} ;
	\end{tikzpicture}
	\hspace{1.7cm}
	\begin{tikzpicture}
	\draw (0,0)--(2,0);
	\draw (-0.1,0) circle[radius=0.1cm]
	node[anchor=south]{$ \zeta$}
	node[anchor=north]{$ \frac{2}{3}$}
	(2.1,0) circle[radius=0.1cm]
	node[anchor=south]{$\zeta^2$}
	node[anchor=north]{$ \;\;\;\;\;\;\; \frac{8}{3}-4m$};
	\draw (1,0)
	node[anchor=south]{$ \zeta$}
	node[anchor=north]{$ -\frac{8}{3}+4m$} ;
	\end{tikzpicture}}
\end{gather*}
for $m= \frac{2+3p'}{6}$, $p' \in \Z$, with simple roots I: $\lbrace\alpha_1, \alpha_2\rbrace$ and II: $\lbrace -\alpha_{1}, \alpha_{112}\rbrace$.

It can be interpreted as a limiting case of Heckenberger row 6 (a color Lie algebra), which for this choice of $q^2$ has the same $q$-diagram.
\end{enumerate}

\subsubsection*{Heckenberger row 11 ($\boldsymbol{G_2}$)}
This case of the list is described by the braiding diagram:
\[
\begin{tikzpicture}
	\draw (0,0)--(1.6,0);
	\draw (-0.1,0) circle[radius=0.1cm] node[anchor=south]{$ q^2$}
			(1.7,0) circle[radius=0.1cm] node[anchor=south]{$ q^{6}$};
	\draw (0.8,0) node[anchor=south]{$ q^{-6}$} ;
\end{tikzpicture}
\]
with $q^2 \neq \pm 1$, $q^2 \not\in \mathbb{G}_3$ and simple roots $\lbrace\alpha_1, \alpha_2\rbrace$. The set of positive roots is given by $\lbrace \alpha_1, \alpha_2, \alpha_{12}, \alpha_{112}, \alpha_{1112}, \alpha_{11122}\rbrace$ with unique associate Cartan matrix:
\[(c_{ij})= \begin{bmatrix}
 \hphantom{-}2 & -3\\
 -1 & \hphantom{-}2
\end{bmatrix}. \]
Definition~\ref{mCartan} gives the following realization
\begin{gather*}
(m_{ij})= \begin{bmatrix}
 \hphantom{-}2m & -3m\\
 -3m & \hphantom{-}6m
\end{bmatrix},\\
\mbox{\begin{tikzpicture}
	\draw (0,0)--(1.6,0);
	\draw (-0.1,0) circle[radius=0.1cm]
	node[anchor=south]{$ q^2$}
	node[anchor=north]{$ 2m$}
	(1.7,0) circle[radius=0.1cm]
	node[anchor=south]{$ q^6$}
	node[anchor=north]{$ 6m$};
	\draw (0.8,0)
	node[anchor=south]{$ q^{-6}$}
	node[anchor=north]{$ -6m$} ;
	\end{tikzpicture}}
\end{gather*}
	
If all pairs are $m$-Cartan, in particular if
$\operatorname{ord}\big(q^2\big)>4,\operatorname{ord}\big(q^6\big)>2$, then this is the unique realization by Lemma~\ref{cartanclassification}. We now discuss the other cases, where we ultimately find {\it no additional realizations}:
\begin{enumerate}\itemsep=0pt
	\item
When $q^2 \in \mathbb{G}_4$, the pair $(1,2)$ is $q$-Cartan and $q$-truncation. We now look for a possible additional realization, where $(1,2)$ is $m$-truncation: This assumption fixes $m_{11}=\frac{1}{4}$, and since $(2,1)$ is still $m$-Cartan, $m_{12}$ is still fixed as in the previous case. Hence the assumption uniquely determines the following possible realization $(m_{ij})=:\big(m_{ij}^\mathrm{I}\big)$, and we also compute the reflection $\r_1(m_{ij}^\mathrm{I})=:\big(m_{ij}^\mathrm{II}\big)$:
\[\big(m_{ij}^{\mathrm{I}}\big)= \begin{bmatrix}
 \frac{1}{2} & -3m\\
 -3m & 6m
\end{bmatrix}, \qquad \big(m_{ij}^{\mathrm{II}}\big)= \begin{bmatrix}
 \frac{1}{2} & -\frac{3}{2} + 3m\\
 -\frac{3}{2} + 3m & \frac{9}{2} -12m
\end{bmatrix} \]
with simple roots~I: $\lbrace\alpha_1, \alpha_2\rbrace$ and~II: $\lbrace -\alpha_{1}, \alpha_{1112}\rbrace$.
The pair $(2,1)$ in $\mathrm{II}$ is only $q$-Cartan, so it must be $m$-Cartan, which requires $m= \frac{1}{4}$. But for this value of $m$ the realization agrees with the previous realization. It appears here a second time, because for this realization $(1,2)$ is both $m$-Cartan and $m$-truncation

\item When $q^2\in \mathbb{G}_6$, the pair $(2,1)$ is $m$-Cartan and $m$-truncation, and we now look for a~possible additional realization where it is $m$-truncation. Again, this uniquely determines the following possible realization $(m_{ij})=:\big(m_{ij}^\mathrm{I}\big)$, and we also compute the reflection $\r_1\big(m_{ij}^\mathrm{I}\big)=:\big(m_{ij}^\mathrm{II}\big)$:
\[\big(m_{ij}^{\mathrm{I}}\big)= \begin{bmatrix}
 2m & -3m\\
 -3m & 1
\end{bmatrix}, \qquad \big(m_{ij}^{\mathrm{II}}\big)= \begin{bmatrix}
 1 - 4m & -1 + 3m\\
 -\frac{3}{2} + 3m & 1
\end{bmatrix} \]
with simple roots I: $\lbrace\alpha_1, \alpha_2\rbrace$ and II: $\lbrace \alpha_{12}, -\alpha_{2}\rbrace$. The pair $(1,2)$ in $\mathrm{II}$ is only $q$-Cartan, so it must be $m$-Cartan, which requires $m= \frac{1}{6}$. But for this value of $m$ again the realization agrees with the first realization.
\end{enumerate}

\section{Super Lie type}\label{SuperLie}

In the classification of finite-dimensional Nichols algebras of diagonal type in~\cite{Heck07} several infinite series occur, which are not of Cartan type, but which are linked to the root system of certain Lie superalgebras. Corresponding quantum super groups had been defined, e.g., in~\cite{KT91}. In~\cite[Example~9 and Theorem~24]{HY08} it is shown that an earlier definition of a generalized root system for Lie superalgebras \cite{Ser96} is a special case of the generalized root system in the sense of Section~\ref{sec_rootsys}. In~\cite{AAY11} the contragradient Lie superalgebras and their quantum supergroups are related to the corresponding Nichols algebras by the process of bozonization.

\subsection[$q$ diagram]{$\boldsymbol{q}$ diagram}

Let $\mathfrak{g} = \mathfrak{g_0} \oplus \mathfrak{g_1}$ be a simple Lie superalgebra of \emph{classical}, \emph{basic} type \cite{FSS96}, i.e., of type $A(m, n)$, $B(m, n)$, $C(n + 1)$, $D(m, n)$, $F(4)$, $G(3)$, $D(2, 1; \alpha)$. For these Lie superalgebras a (non degenerate or zero) Killing form $(\, ,\,)_\mathfrak{g} $ is defined.

We now choose a Weyl chamber $ \alpha_1, \dots, \alpha_{\f-1}, \alpha_\f,  \alpha_{\f+1}, \dots, \alpha_{n}$ with just one simple fermionic root $\alpha_\f$. We call it the \emph{standard chamber} according to~\cite{Kac77}. Given $\alpha$ positive root in the standard chamber, we define $f(\alpha)$ the multiplicity of $\alpha_\f$ in $\alpha$.

We can then split $\mathfrak{g}$ as the direct sum of vector spaces
\[\mathfrak{g} = \mathfrak{g'} \oplus \mathfrak{g''} \oplus \mathfrak{m},\]
where $\mathfrak{g'}$ and $\mathfrak{g''}$ are two bosonic connected component generated by the simple roots $ \alpha_1, \dots, \alpha_{\f-1}$ and $ \alpha_{\f+1}, \dots, \alpha_{n}$ respectively, while $\mathfrak{m}$ is the $\mathfrak{g'} \oplus \mathfrak{g''}$-module spanned by all other roots.

We have that $\mathfrak{m}$ contains $\mathfrak{g_1}$ and thus in particular contains the $\mathfrak{g'} \oplus \mathfrak{g''}$-submodule generated by the fermion $\alpha_\f$, i.e., the vector space of fermions $\gamma$, with $f(\gamma) =1$.
Moreover $\mathfrak{m}$ may contain bosonic roots $\delta$, with $f(\delta)\in 2\N$.
\begin{defin}We can write the inner product $\left( \,,\,\right)_\mathfrak{g}$ of two arbitrary simple roots as \[(\alpha_{i}, \alpha_{j})_\mathfrak{g} = (\alpha_{i}, \alpha_{j})_\mathfrak{g'}+(\alpha_{i}, \alpha_{j})_\mathfrak{g''}= \begin{cases} (\alpha_{i}, \alpha_{j})_\mathfrak{g'}& \text{if } i \leq \f ,\  j < \f, \\
0 & \text{if } i \leq \f \leq j, \\
(\alpha_{i}, \alpha_{j})_\mathfrak{g''}& \text{if } i \geq \f ,\  j > \f.
\end{cases}\]
In particular we assume $(\alpha_{f}, \alpha_{f})_\mathfrak{g}= (\alpha_{f}, \alpha_{f})_\mathfrak{g'} = (\alpha_{f}, \alpha_{f})_\mathfrak{g''}=0$.
\end{defin}

\begin{defin}\label{defbraidSuper} Let $q'$, $q''$ be primitive roots of unity.
Then to the above data in the standard chamber we associate the braiding matrix $(q_{ij})$
with
\[q_{ij} =
\begin{cases} (q')^{(\alpha_{i}, \alpha_{j})_\mathfrak{g'}}&\text{ if } i \leq \f ,\  j < \f, \\
(q'')^{(\alpha_{i}, \alpha_{j})_\mathfrak{g''}}& \text{if } i \geq \f ,\  j > \f, \\
  1 & \text{if } i>\f>j, \\
-1 & \text{if } i=\f=j.
\end{cases}\]
\end{defin}

Under certain conditions relating $q'$, $q''$, these braiding gives a finite-dimensional Nichols algebra~$\mathcal{B}(q_{ij})$, which we call of \emph{super Lie type}. We will continue our general considerations without having to specify these conditions on~$q'$,~$q''$. In the process of establishing a realization~$(m_{ij})$ depending on~$m'$,~$m''$ we will encounter additional conditions relating~$m'$,~$m''$. These additional conditions will in each case imply the conditions relating $q'$, $q''$ in Heckenberger's list for this specific Nichols algebra of super Lie type. For an explicit example, see Example~\ref{exm_superDReflection}, in general these conditions will arise for the exceptional cases in Corollary~\ref{3 conditions II} and will be spelled out when we go through all cases in Sections~\ref{subsec_rank2super} and~\ref{subsec_rankNsuper}.

The reflections will act on the braiding as follow:
\begin{itemize}\itemsep=0pt
	\item Reflections $\r_k$ around bosonic roots $\alpha_k$ leave $(q_{ij})$ invariant.
	\item Reflections $\r_k$ around fermionic roots $\alpha_k$ interchange fermionic and bosonic roots.
\end{itemize}
\begin{oss}
In the classification of Nichols algebras in \cite{Hecklist} and \cite{Heck06} we find that the fermion (as in the Lie superalgebra sense of the term) in the standard chamber $\alpha_\f$ has $q_{ff}= -1$, i.e., it is $q$-truncation. This is not true in general for every fermion as we can see in the following example.
\begin{ese}\label{fermnot1}
The case Heckenberger row~5 of Table~1 in~\cite{Hecklist} is described by two diagrams:
\begin{gather*}
\begin{tikzpicture}
	\draw (0,0)--(1.6,0);
	\draw (-0.1,0) circle[radius=0.1cm] node[anchor=south]{$ -1$}
			(1.7,0) circle[radius=0.1cm] node[anchor=south]{$ q^2$};
	\draw (0.8,0) node[anchor=south]{$ q^{-4}$};
	\draw (3.6,0)--(5.2,0);
	\draw (3.5,0) circle[radius=0.1cm] node[anchor=south]{$ -1$}
			(5.3,0) circle[radius=0.1cm] node[anchor=south]{$ -q^{-2}$};
	\draw (4.4,0) node[anchor=south]{$ q^{4}$};
\node at (0.8,-0.3) {$\mathrm{I}$};
\node at (4.3,-0.3) {$\mathrm{II}$};
\end{tikzpicture}
\end{gather*}
corresponding to the simple roots: \[ \mathrm{I}\colon \ \lbrace\alpha_1, \alpha_2\rbrace, \qquad \mathrm{II}\colon \ \lbrace -\alpha_{1}, \alpha_{12}\rbrace. \]
This is the Lie superalgebra ${B(1,1)}$ and $\alpha_{12}$ is a fermion with $q(\alpha_{12},\alpha_{12}) \neq -1$.\\
We will describe this example in detail later in this section.
\end{ese}
\end{oss}

\subsection[Construction of $(m_{ij})$]{Construction of $\boldsymbol{(m_{ij})}$}

\begin{defin}\label{mSuperLie} Given $p', p''\in \Z$ such that $(p',p)= (p'',p) =1$, we define $m':= \frac{p'}{p}$, $m'':= \frac{p''}{p} $ and in the standard chamber:
\[m_{ij}^S =
\begin{cases} (\alpha_i, \alpha_j)_{\mathfrak{g'}} m' &  \text{if } i \leq \f ,\  j < \f, \\
(\alpha_i, \alpha_j)_{\mathfrak{g''}} m'' &  \text{if } i \geq \f ,\  j > \f, \\
0 & \text{if } i>\f>j ,\\
1 & \text{if } i=\f=j.
\end{cases}\]
\end{defin}

We notice that if we restrict to $\mathfrak{g'}$ or $\mathfrak{g''}$, we get exactly the same result as in the Cartan type section for $p'$, $p$ respectively $p''$, $p$.
\begin{lem}
If we call $q'={\rm e}^{{\rm i} \pi m'}$ and $q''={\rm e}^{{\rm i} \pi m''}$, then $q_{ij}={\rm e}^{{\rm i} \pi m_{ij}}$ is the braiding defined in Definition~{\rm \ref{defbraidSuper}}.
\end{lem}
\begin{proof}
We have $m_{ij}=0$ if $\alpha_i$ and $\alpha_j$ are disconnected, so that $1={\rm e}^{{\rm i} \pi \cdot 0}$ and $m_{\f\f}=1$ for the fermionic root which gives $-1 = {\rm e}^{{\rm i} \pi \cdot 1}$ as demanded.
\end{proof}

\begin{lem}In an arbitrary chamber $C$ with simple roots ${\gamma_1, \dots, \gamma_{r}}$ we have \[{m_{ij}}^C= (\gamma_i, \gamma_j)_{\mathfrak{g'}}m' + (\gamma_i, \gamma_j)_{\mathfrak{g''}}m'' +f(\gamma_i)f(\gamma_j).\]
\end{lem}
\begin{proof}
We write $\gamma_i = \sum_{k} x_{ik} \alpha_k$ and $\gamma_j = \sum_{l} x_{jl} \alpha_l$ and we extend for linearity:
\begin{align*}
{m_{ij}}^C &=  \sum_{k,l} x_{ik} x_{jl} {m_{kl}}^S \\
&= \sum_{k,l \in \mathfrak{g'} \cup \lbrace f \rbrace} x_{ik} x_{jl} (\alpha_k, \alpha_l)_{\mathfrak{g'}}m' + \sum_{k,l \in \mathfrak{g''} \cup \lbrace f \rbrace} x_{ik} x_{jl} (\alpha_k, \alpha_l)_{\mathfrak{g''}}m'' + x_{i\f} x_{j\f}   \\
&=  (\gamma_i, \gamma_j)_{\mathfrak{g'}}m' + (\gamma_i, \gamma_j)_{\mathfrak{g''}}m''+f(\gamma_i) f(\gamma_j),
\end{align*}
where the last equality follows from the definition of $f(\gamma)$ as the multiplicity of $\alpha_\f$ in $\gamma$ and the fact that on each component $\mathfrak{g'}$ and $\mathfrak{g''}$ the roots are spanned as in a Lie algebra.
\end{proof}

\begin{cor}\label{Mcondotherchamb}A root $\gamma$ in an arbitrary chamber is
\begin{itemize}\itemsep=0pt
	\item[--] $m$-truncation if $(\gamma,\gamma)_\mathfrak{g'} m' + (\gamma,\gamma)_\mathfrak{g''} m'' + f(\gamma)f(\gamma) =1$,
	\item[--] $m$-Cartan if, for every simple root $\alpha_i$ in the standard chamber,
	\begin{gather*}
	2(\gamma,\alpha_i)_\mathfrak{g'} m' + 2(\gamma,\alpha_i)_\mathfrak{g''} m'' + 2f(\gamma)f(\alpha_i) \\
\qquad {} = c_{\gamma,\beta_i}((\gamma,\gamma)_\mathfrak{g'} m' + (\gamma,\gamma)_\mathfrak{g''} m'' + f(\gamma)f(\gamma)).
	\end{gather*}
\end{itemize}
\end{cor}

\begin{ese}\label{A(1,1)}
We consider as an example the Lie superalgebra ${A(1,1)}$ of rank~3.
The simple roots in the standard chamber are $\lbrace\alpha_1,   \alpha_2 = \alpha_\f,   \alpha_3\rbrace$ with inner product:
\[
(\alpha_i, \alpha_j)= \begin{bmatrix}
 \hphantom{-}2 & -1 & \hphantom{-}0\\
 -1 &   \hphantom{-}0 & -1 \\
 \hphantom{-}0 & -1 & \hphantom{-}2
\end{bmatrix}.
\]
Hence,
\[
\big(m_{ij}^S\big)= \begin{bmatrix}
 2m' & -m' & 0\\
 -m' &   1 & -m'' \\
 0 & -m'' & 2m''
\end{bmatrix}, \qquad
(q_{ij})= \begin{bmatrix}
 (q')^2 & (q')^{-1} & 1\\
  (q')^{-1} &   -1 & (q'')^{-1} \\
 1 & (q'')^{-1} & (q'')^2
\end{bmatrix}.
\]
 \end{ese}
We check, under which conditions $\big(m_{ij}^S\big)$ above is indeed a realization in all Weyl chambers.
\begin{lem}\label{onebossector}
	 If $\gamma= \sum_{i=1}^{f-1} a_i \alpha_i\in\mathfrak{g'}$ $($resp.\ for $\gamma \in \mathfrak{g''})$ the root $\gamma$ is $m$-Cartan.
	 \end{lem}
\begin{proof}
Suppose $\gamma \in \mathfrak{g'}$; then
\begin{alignat*}{3}
&(\gamma, \gamma)= (\gamma, \gamma)_\mathfrak{g'}, \qquad && (\gamma, \gamma)_\mathfrak{g''}=0,& \\
& (\gamma, \alpha_i)= (\gamma, \alpha_i)_\mathfrak{g'}, \qquad&& (\gamma, \alpha_i)_\mathfrak{g''}=0&
\end{alignat*}	
for an arbitrary simple root $\alpha_i$. Moreover $f(\gamma)=0$.
We then write out {\rm (\ref{cond7}A)}
\begin{gather*}
2(\gamma, \alpha_i) m'  = c_{\gamma,\alpha_i}((\gamma, \gamma) m').
\end{gather*}
This is true because of definition of $c_{\gamma\alpha_i}$ in the Lie algebra setting.
By linearity in the simple roots $\alpha_i$ it is possible to extend this result to an arbitrary root $\alpha = \sum b_i \alpha_i$.
\end{proof}	
\begin{lem}\label{isofermions} If $\gamma \neq \alpha_\f$ is isotropic, i.e., $(\gamma, \gamma) =(\gamma, \gamma)_\mathfrak{g'}= (\gamma, \gamma)_\mathfrak{g''}= 0$, and $f(\gamma) = \pm 1$ then~$\gamma$ is $m$-truncation.
\end{lem}
\begin{proof}
Condition (\ref{cond7}B) for a root to be $m$-truncation reads:
\begin{gather*}
(\gamma, \gamma)_\mathfrak{g'} + (\gamma, \gamma)_\mathfrak{g''} + f(\gamma)f(\gamma) = 1,
\end{gather*}
which is clearly true under these hypothesis.
\end{proof}	

We summarize these results in the following:
\begin{cor}\label{3 conditions II}
The matrix $(m_{ij})$ defined in Definition~{\rm \ref{mSuperLie}} corresponds to the given braiding~$(q_{ij})$ and the realization condition \eqref{cond7} holds for every root $\alpha$ in every Weyl chamber containing~$\alpha$ as simpe root, with the following possible exceptions:
\begin{enumerate}\itemsep=0pt
	\item[$1.$] $\alpha$ is a boson in $\mathfrak{g'}\cup \mathfrak{g''}$, i.e., $f(\alpha)$ is a strictly positive even integer.
	\item[$2.$] $\alpha$ is an isotropic fermion with $f(\alpha) \neq \pm 1$.
	\item[$3.$] $\alpha$ is a non-isotropic fermion.
	\item[$4.$] $\alpha$ is a fermion strong orthogonal to another fermion $\gamma$, i.e., in their real span $\langle \alpha, \gamma \rangle_\mathbb{R}$ there are no roots.
\end{enumerate}
\end{cor}

\begin{proof} If a boson $\alpha$ is contained in $\mathfrak{g'}$ or $\mathfrak{g''}$, then Lemma~\ref{onebossector} asserts that it must be $m$-Cartan. Otherwise, $\alpha$ must have the exceptional case (1) with $f(\alpha)>0$ even.

Let now $\alpha$ be a fermion which is not strong orthogonal to any other fermions. If it is isotropic and $f(\alpha) = \pm 1$, thanks to Lemma~\ref{isofermions}, it satisfies condition (\ref{cond7}B). If $f(\alpha) \neq \pm 1$ or it is non-isotropic, we have the exceptional case (2) and (3).

If $\alpha$ and $\gamma$ are two strong orthogonal fermions, then $c_{\alpha\beta}=0$.
\end{proof}

In the cases (1)--(4) we have to check explicitly condition~(\ref{cond7}A) or (\ref{cond7}B) by Corollary~\ref{Mcondotherchamb}. Note that it is possible that in a realization a fermion is $m$-Cartan but not $m$-trunctation. In case~(4) the condition simplifies to
In this case we have to check for which $m'$ and $m''$
\[m_{\alpha,\beta}=(\alpha,\beta)_\mathfrak{g'} m' + (\alpha,\beta)_\mathfrak{g''} m'' + f(\alpha)f(\beta) =0.\]

\begin{oss} Going through all cases, we do not find any boson with $f(\alpha)>2$ and any fermion with $f(\alpha)>1$. Thus, point (1) concerns then just bosons with $f(\alpha)=2$, for example in type $D(m,n)$ and $F(4)$, and point~(2) never occurs.
\end{oss}

As last general result we now state conversely a classification lemma:
\begin{lem}\label{superlieclassification}
If all the bosonic roots are $m$-Cartan, then the unique possible realizing solution for the given braiding is the matrix $(m_{ij})$ of Definition~{\rm \ref{mSuperLie}}. In particular this is the case if $\ell_i > 1- c_{ij}$ for $\forall i \neq \f$.
\end{lem}

\begin{proof}
Condition (\ref{cond7}) gives a unique solution $(m_{ij})$ in the standard chamber: the fermionic root is $m$-truncation and thus fixed to $m_{\f\f} =1$, while, since all the other roots are $m$-Cartan, restricting our study to the two bosonic sectors separately we end up in the same situation of Lemma~\ref{cartanclassification}.
Moreover the bilinear form is uniquely fixed by its values on one basis.
\end{proof}

\begin{ese}\label{exm_superDReflection}
We apply Lemma~\ref{3 conditions II} to Example~\ref{A(1,1)}: after reflecting the standard chamber set of roots around the fermion $\alpha_2$, we find for new simple roots: $\lbrace \alpha_{12}, -\alpha_2, \alpha_{23}\rbrace$ the matrix:
\[ \big(m_{ij}^{\r_2(S)}\big)= \begin{bmatrix}
1 & -1 + m' & -1+m' +m'' \\
-1 + m' & 1 & -1 + m'' \\
-1+m' +m'' & -1+m'' & 1	
\end{bmatrix}.
\]
Exception (4) of Lemma~\ref{3 conditions II} appears. We then have to ask $m_{23}=0$, i.e., $m' + m''=1$ and thus $q'q''=-1$. In that case $(m_{ij})$ is a~realizing solution.
This construction realizes the Nichols algebra~$\mathcal{B}(q_{ij})$ described by case row 8 of Table~2 in~\cite{Hecklist} when $q\neq \pm 1$.
\end{ese}

\subsection{Central charge}
We will compute the central charge of systems associated to a Lie superalgebra $\mathfrak{g}$ of rank $r$, with non degenerate Killing form $(\,,\,)$.
\begin{prop}
The central charge is $c= r - 12(Q,Q)$ with \[Q=\frac{\rho_{\mathfrak{g'}}^\vee}{\sqrt{m'}}-\rho_{\mathfrak{g'}}\sqrt{m'}+\frac{\rho_{\mathfrak{g''}}^\vee}{\sqrt{m''}} -\rho_{\mathfrak{g''}}\sqrt{m''}-\rho_{{\rm rest}}^\vee,\]
where we denoted by $\rho_{\mathfrak{g'}}$ the half sum of positive roots in~$\mathfrak{g'}$, $\rho_{\mathfrak{g''}}$ the half sum of positive roots in~$\mathfrak{g''}$ and $\rho_{{\rm rest}}$ the half sum of the remaining positive roots of~$\mathfrak{g}$.
\end{prop}

\begin{proof}By Proposition \ref{central_charge} the vector $Q$ is characterized uniquely by the following condition for all simple roots $\alpha_i$ of $\mathfrak{g}$
\[
\frac{1}{2}\big({-}\sqrt{m_i}\alpha_i,-\sqrt{m_i}\alpha_i\big) -\big({-}\sqrt{m_i}\alpha_i,Q\big) =1, \qquad \text{where}\quad m_i = \begin{cases} \dfrac{p'}{p} & \text{if } i<f,\\
1 & \text{if } i=f,\\
 \dfrac{p''}{p} &  \text{if } i>f.
\end{cases}
\]
Let $\lambda_j^\vee= \frac{\lambda_j}{d_j}$ be such that $(\alpha_i, \lambda_j^\vee) = \delta_{ij}$. Since $\rho_\mathfrak{g}= \sum\limits_{i=1}^n \lambda_i$, we have that $\rho_{\mathfrak{g'}}= \sum\limits_{i<f} \lambda_i$, $\rho_{\mathfrak{g''}}= \sum\limits_{i>f} \lambda_i$ and then $\rho_{{\rm rest}} = \lambda_\f$. We can thus rewrite $Q$ as
\begin{gather*}
Q=\frac{\rho_{\mathfrak{g'}}^\vee}{\sqrt{m'}}-\rho_{\mathfrak{g'}}\sqrt{m'}+\frac{\rho_{\mathfrak{g''}}^\vee}{\sqrt{m''}}-\rho_{\mathfrak{g''}}\sqrt{m''}-\rho_{{\rm rest}}^\vee
= \sum\limits_i\left( \frac{1}{\sqrt{m_i}}-\sqrt{m_i}d_i\right) \lambda_i^\vee.
\end{gather*}
Hence the previous equation becomes
\begin{gather*}
 \frac{1}{2}\big({-}\alpha_i \sqrt{m_i},-\alpha_i \sqrt{m_i}\big)-\big({-}\alpha_i \sqrt{m_i}, Q\big) \\
\qquad{}= \frac{1}{2} 2d_i m_i+ \sum\limits_j \sqrt{m_i} \left(\frac{1}{ \sqrt{m_j}}- \sqrt{m_j} d_j\right) \big(\alpha_i, \lambda_j^\vee\big) =1.\tag*{\qed}
\end{gather*}\renewcommand{\qed}{}
\end{proof}

\subsection{Algebra relations}\label{sec_SuperNA}

 According to the results in Section \ref{sec_Serre} we restrict ourselves here to the case of simple-laced $\g=A(n,m),D(m,n),D(2,1;\alpha)$. By~\cite[Sections~4.1,~4.4 and~4.5]{AA17} a set of defining relations is
 \begin{itemize}\itemsep=0pt
 	\item Commutation relation $[x_i,x_j]_q$ for $\i\not\sim j$.
 	\item Serre relations $[x_i,[x_i,x_j]_q]_q=0$ for $i\sim j$ for $i$ bosonic (for $\alpha_i$ fermionic the Serre relation is implied by $x_i^2=0$).
 	\item Root vector truncation relations $x_\alpha^{\ell_\alpha}=0$ for any root $\alpha\in\Phi^+$ and $\ell_\alpha=\operatorname{ord}\big(q^{2}\big)$ for~$\alpha$ bosonic or $\ell_\alpha=2$ for~$\alpha$ fermionic, where the root vector $x_\alpha$ is defined by repeated reflections using Lusztig's isomorphism.
 	\item For $j$ fermionic the additional relations $[x_j,[x_i,[x_j,x_k]_q]_q]_q=0$ for any subsystem $\alpha_i$, $\alpha_j$, $\alpha_k$ of type $A_3$.
 \end{itemize}

 We now consider the realization in Definition \ref{mSuperLie}. We first clarify, which parameters $m'$, $m''$ give subpolar $(m_{ij})$.
 \begin{lem}
 	Subpolarity for $(m_{ij})$ in Definition~{\rm \ref{mSuperLie}} holds under the condition
 \[\frac{1}{2d'}\geq m'>0,\qquad \frac{1}{2d''}\geq m''>0,\qquad \det (m_{ij})>0.\]
 \end{lem}
 \begin{proof}
 	Subpolarity in Section~\ref{smallTheorems} for all monomials holds under the assumptions $|\alpha_i|\leq 1$, which means $2d'm'\leq 1$, $2d''m''\leq 1$, and $(m_{ij})$ positive definite. By Sylvester's criterion, this is equivalent to $\det (m_{ij})>0$ and to the principal minor being positive definite. The principal minor is a~rescaling of the root lattices $\g'$, $\g''$, so it is positive definite for $m',m''>0$.
 \end{proof}

 \begin{ese}For type $A(n',n'')$ these conditions read
 \[\frac{1}{2}\geq m'>0,\qquad \frac{1}{2}\geq m''>0,\qquad \frac{n'}{n'+1}m'+\frac{n''}{n''+1}m''<1,\]
 where $m'+m''=1$.
 \end{ese}

By Corollary \ref{cor_A2}, for $\alpha_i$ bosonic the quantum Serre relations hold for all $m$ and by Corollary~\ref{cor_A10}, for $\alpha_i$ fermionic, the quantum Serre relations can be analytically continued and hold for $m<\frac{3}{2}$. By Corollary~\ref{cor_trun}, the truncation relations of simple root vectors $x_\alpha^{\ell_\alpha}=0$ can be analytically continued for all values of $m$, but they only hold for $m\geq 0$ for $\alpha$ bosonic and for all $m$ for $\alpha$ fermionic.

 The additional relation in degree $\alpha_i+2\alpha_j+\alpha_k$ with $m_{ij}=-m'$, $m_{jj}=1$, $m_{jk}=-m''$, $m_{ik}=0$ is subpolar for $m',m''<1$, because going through all subsets with multiplicities $J$, subpolarity amounts to the following inequalities
 \begin{gather*}
 -m'>-1,\\
 -m''>-1,\\
 -2m'+1>-2,\\
 -2m''+1>-2,\\
 -m'-m''>-2,\\
 -2m'+1-2m''>-3.
 \end{gather*}

The Nichols algebra without truncations relations of bosonic roots is the Borel part of the Kac--DeConcini--Procesi quantum super group $U_q^\mathcal{K}(\g)$ resp.~the distinguished pre-Nichols algebra~\cite{Ang16}. We hence find:

 \begin{cor}\label{cor_superNA}
 	In $(m_{ij})$ in Definition~{\rm \ref{mSuperLie}}, which is the realization of the braiding $(q_{ij})$ associated to a simply-laced Lie superalgebra $\g$ at $q={\rm e}^{{\rm i}\pi m}$, $m\not\in \frac{1}{2}\Z$ or the alternate realization of the braiding associated to a simply-laced Lie algebra $\g$ at $q={\rm e}^{{\rm i}\pi m}={\rm i}$, $m\in \frac{1}{2}+\Z$, the Nichols algebra relations hold for the corresponding screenings as follows:
 	\begin{itemize}\itemsep=0pt
 		\item For $\frac{1}{2}\geq m', m''>0$ and $\det(m_{ij})<0$ subpolarity holds, so that all relations hold. Differently spoken, the algebra of screenings is a surjective image of the Borel part of the small quantum super group $u_q(\g)$.
 		\item For $m'<0$ the Serre relations and the additional relation of a root $\alpha$ in $\g'$ with the fermionic simple root $\alpha_\f$ holds. The truncation relations of simple fermionic root vectors hold, the truncation relations of simple bosonic root vectors fail.
 		\item The analogous statment holds for $m''$ and simple roots in $\g''$.
 		\item For $\frac{3}{2}>m', m''>\frac{1}{2}$ the Serre relation holds and for $1>m',m''>\frac{1}{2}$ the additional relation holds, and we can make no assertion for larger values of $m'$, $m''$.
 	\end{itemize}
 	We would conjecture that for $m',m''<0$ also the truncation relations of non-simple root fermionic vectors hold, so that this case we get the Borel part Kac--DeConcini--Procesi quantum super group $U_q^\mathcal{K}(\g)$ where the truncation relations for one or both subsets of bosonic roots hold. We would conjecture that all surjections above are in fact isomorphism.
\end{cor}
In the example $A(n',n'')$ we have the additional condition $m'+m''=1$, so we expect (proven except for truncation of non-simple root vectors) for $1>m>0$ the small quantum supergroup and for $m'<0$, $m''>1$ (or vice versa) the Borel part of a version of the Kac--DeConcini--Procesi quantum super group $U_q^\mathcal{K}(\g)$ where the bosonic root vectors of~$A(n')$ fail and those of~$A(n'')$ hold and for the latter the additinal relation is in question.

\subsection{Examples in rank 2}\label{subsec_rank2super}
We now present the cases of Table~1 in~\cite{Hecklist} that come from Lie superalgebras of rank~$2$. We will check in every case whether the exceptions of Corollary~\ref{3 conditions II} appear.
In rank $2$, there is obviously always just one bosonic sector $\mathfrak{g'}$.

We also remark in each case how the simple roots in the standard chamber can be expressed using the standard basis $\epsilon_i$ and $\delta_i$ in \cite{Kac77}.

\subsubsection*{Heckenberger row 3}
The case row 3 of Table~1 in \cite{Hecklist}, studied in Example~\ref{sl(2|1)}, is realized by the Lie superalgebra lattice ${A(1,0)}$. This case is described by the diagrams
\begin{gather*}
\mbox{\begin{tikzpicture}
	\draw (0,0)--(1.6,0);
	\draw (-0.1,0) circle[radius=0.1cm] node[anchor=south]{$ q^2$}
			(1.7,0) circle[radius=0.1cm] node[anchor=south]{$ {-1}$};
	\draw (0.8,0) node[anchor=south]{$ q^{-2}$};
	\draw (3.6,0)--(5.2,0);
	\draw (3.5,0) circle[radius=0.1cm] node[anchor=south]{$ -1$}
			(5.3,0) circle[radius=0.1cm] node[anchor=south]{$ -1$};
	\draw (4.4,0) node[anchor=south]{$ q^2$};
\node at (0.8,-0.3) {$\mathrm{I}$};
\node at (4.4,-0.3) {$\mathrm{II}$};
\end{tikzpicture}}
\end{gather*}
with $q^2 \neq \pm 1$ and simple roots $\mathrm{I}\colon \lbrace\alpha_1, \alpha_2\rbrace$, $\mathrm{II}\colon \lbrace\alpha_{12}, -\alpha_2\rbrace$. The set of positive roots is given by $\lbrace \alpha_1, \alpha_2, \alpha_{12}\rbrace$ with unique associate Cartan matrix and inner products
\[(c_{ij})= \begin{bmatrix}
 \hphantom{-}2 & -1\\
 -1 & \hphantom{-}2
\end{bmatrix}, \qquad
(\alpha_i, \alpha_j)= \begin{bmatrix}
 \hphantom{-}2 & -1 \\
 -1 &  \hphantom{-}0
\end{bmatrix}.
\]

Therefore the matrix $(m_{ij})$ in the standard basis and after reflecting around $\alpha_2$ are given by
\[
\big(m_{ij}^{\mathrm{I}}\big)= \begin{bmatrix}
 2m & -m\\
 -m &  1
\end{bmatrix}, \qquad
\big(m_{ij}^{\mathrm{II}}\big)= \begin{bmatrix}
 1 & -1+m\\
 -1+m & 1
\end{bmatrix}.
\]
None of the exceptions of Lemma~\ref{3 conditions II} appears; therefore $(m_{ij})$ is a realizing solution for all~$m$. This result matches with Example~\ref{sl(2|1)}.
\begin{oss} As observed in Example~\ref{sl(2|1)}, if we allow the value $q^2=-1$ we obtain row 2 of Table~1 in~\cite{Hecklist}.
\end{oss}
\begin{oss}
The simple roots in the standard chamber of $A(1,0)$ can be expressed by
\[ \alpha_1 = \epsilon_1 - \epsilon_2, \qquad \alpha_2 = \alpha_\f = \epsilon_2 - \delta_1.\]
\end{oss}

\subsubsection*{Heckenberger row 5}
Row 5 of Table 1 in~\cite{Hecklist} is realized by the Lie superalgebra lattice ${B(1,1)}$. This case is described by the diagrams
\begin{gather*}
\mbox{\begin{tikzpicture}
	\draw (0,0)--(1.6,0);
	\draw (-0.1,0) circle[radius=0.1cm] node[anchor=south]{$ q^2$}
			(1.7,0) circle[radius=0.1cm] node[anchor=south]{$ {-1}$};
	\draw (0.8,0) node[anchor=south]{$ q^{-4}$};
	\draw (3.6,0)--(5.2,0);
	\draw (3.5,0) circle[radius=0.1cm] node[anchor=south]{$ -q^{-2}$}
			(5.3,0) circle[radius=0.1cm] node[anchor=south]{$ -1$};
	\draw (4.4,0) node[anchor=south]{$ q^{4}$};
\node at (0.8,-0.3) {$\mathrm{I}$};
\node at (4.4,-0.3) {$\mathrm{II}$};
\end{tikzpicture}}
\end{gather*}
with $q^2 \neq \pm 1$, $q^2 \not \in \mathbb{G}_4$ and simple roots $\mathrm{I}\colon \lbrace\alpha_1, \alpha_2\rbrace$, $\mathrm{II}\colon \lbrace\alpha_{12}, -\alpha_2\rbrace$.
The set of positive roots is given by $\lbrace \alpha_1, \alpha_2, \alpha_{12}, \alpha_{112}\rbrace$ with unique associate Cartan matrix
\[(c_{ij})= \begin{bmatrix}
 \hphantom{-}2 & -2\\
 -1 & \hphantom{-}2
\end{bmatrix} \]
and inner product
\[
(\alpha_i, \alpha_j)= \begin{bmatrix}
 \hphantom{-}2 & -2 \\
 -2 &  \hphantom{-}0
\end{bmatrix}.
\]
Therefore the matrix $(m_{ij})$ and its reflecting around $\alpha_1$ are given by
\[\big(m_{ij}^{\mathrm{I}}\big)= \begin{bmatrix}
 2m & -2m\\
 -2m& 1
\end{bmatrix}, \qquad
\big(m_{ij}^{\mathrm{II}}\big)= \begin{bmatrix}
 -2m+1 & 2m-1\\
 2m-1 & 1
\end{bmatrix}. \]
None of the exceptions of Lemma~\ref{3 conditions II} appears; therefore~$(m_{ij})$ is a realizing solution for all~$m$.
\begin{oss}\label{Heck5trunc}
When $q^2\in \mathbb{G}_3$, the root $\alpha_1$ is $q$-Cartan \textit{and} $q$-truncation.
When it is $m$-truncation we get
	\begin{gather*} \big(m_{ij}^{\mathrm{I}}\big)= \begin{bmatrix}
 \frac{2}{3} &-2m\\
 -2m & 1
\end{bmatrix}, \qquad
\big(m_{ij}^{\mathrm{II}}\big)= \begin{bmatrix}
 \frac{5}{3} -4m & 2m-1\\
 2m-1 & 1
\end{bmatrix}, \\
(m_{ij}^{\mathrm{III}})= \begin{bmatrix}
 \frac{2}{3} & 2m-\frac{4}{3}\\
 2m-\frac{4}{3} & \frac{11}{3}-8m
\end{bmatrix},
\end{gather*}
where III: $\lbrace -\alpha_{1}, \alpha_{112}\rbrace$ comes after reflecting around~$\alpha_1$.
The root $\alpha_{112}$ is never $m$-Cartan and it is $m$-truncation iff~$m=\frac{1}{3}$.
 But for this value of $m$, $\alpha_1$ is also $m$-Cartan and thus this is \emph{not} a new solution.
\end{oss}

\begin{oss} The roots can be expressed by \[ \alpha_1 =\epsilon_1, \qquad \alpha_2 = \alpha_\f = \delta_1 - \epsilon_1.\]
\end{oss}

\subsection{Arbitrary rank}\label{subsec_rankNsuper}
We generalize our study to arbitrary rank cases.
In every case we will see under which additional assumptions on $m'$, $m''$ the matrices $(m_{ij})$ in Definition~\ref{mSuperLie} are indeed realizing solutions.

\subsubsection*{$\boldsymbol{A(m,n)}$}
\[
\begin{tikzpicture}
	\draw (-0.6,0)--(1,0);
	\draw (-0.7,0) circle[radius=0.1cm] node[anchor=south]{$ q^2$}
			(1.1,0) circle[radius=0.1cm] node[anchor=south]{$ {q^2}$};
	\draw (0.2,0) node[anchor=south]{$ q^{-2}$};
	\draw (1.2,0)--(1.7,0);
			\draw (2.2,0) node{$\cdots$};
	\draw (2.7,0)--(3.2,0);
	\draw (3.3,0) circle[radius=0.1cm] node[anchor=south]{$ -1$};
	\draw (3.4,0)--(3.9,0);
		\draw (4.4,0) node{$\cdots$};
	\draw (4.9,0)--(5.4,0);
	\draw (5.5,0) circle[radius=0.1cm] node[anchor=south]{$ q^{-2}$};
	\draw (5.6,0)--(7.2,0);
	\draw (6.4,0) node[anchor=south]{$ q^2$};
	\draw (7.3,0) circle[radius=0.1cm] node[anchor=south]{$ q^{-2}$};
\end{tikzpicture}
\]
The simple roots in the standard chamber are \[\alpha_1, \dots, \alpha_\f=\alpha_{m+1}, \dots, \alpha_{m+n+1}\]
with inner product matrix
\[
(\alpha_i, \alpha_j)=\begin{bmatrix}
2 & -1 & & & & \\
-1 & \ddots & \ddots & & & \\
 & \ddots & \ddots & \ddots & & \\
 & & \ddots & 0& \ddots & \\
 & & & \ddots & \ddots & -1 \\
 & & & & -1 &  \hphantom{-} 2
\end{bmatrix}\]
We list all the positive roots. We denote by $\Delta_0$ the set of bosons and by $\Delta_1$ the set of fermions according to the literature~\cite{Kac77}.
\begin{gather*}
\Delta_0 = \left\lbrace \alpha_l+ \dots +\alpha_k, \text{ with } l,k<f \text{ or } l,k > f\right\rbrace, \\
\Delta_1 = \left\lbrace \alpha_l+ \dots +\alpha_k, \text{ with } l\leq f \leq k\right\rbrace.
\end{gather*}

We now apply the lemmas of the previous section to determine possible conditions on~$m'$ and~$m''$ such that the matrix~$(m_{ij})$ defined as in Definition~\ref{mSuperLie} is a realizing solution.
\begin{itemize}\itemsep=0pt
	\item All the bosons are either in $\mathfrak{g'}$ or $\mathfrak{g''}$. By Lemma \ref{onebossector}, we know they are always $m$-Cartan.
	\item All the fermions are isotropic and have $f(\alpha)= \pm 1$. By Lemma \ref{isofermions} we know that if they are not strong orthogonal to any other root they are $m$-truncation.
	\item We now focus on the case of strong orthogonal fermions. Let us consider two fermions:
\begin{gather*}
\gamma_1 = \alpha_{l_1}+ \dots + \alpha_{k_1} \qquad \text{with }  l_1\leq f \leq k_1, \\
\gamma_2 = \alpha_{l_2}+ \dots + \alpha_{k_2} \qquad \text{with }  l_2\leq f \leq k_2.
\end{gather*}	
	They are strong orthogonal if $l_1 \neq l_2$, $k_1 \neq k_2$.
	In this case we have to check that $m_{12}=(\gamma_1, \gamma_2)_\mathfrak{g'}m' + (\gamma_1, \gamma_2)_\mathfrak{g''}m'' + f(\gamma_1)f(\gamma_2) =0$.

	We thus compute the inner products in the two bosonic sides. We assume $l_1 < l_2$ and $k_1 < k_2$, because every other combination works analogously and gives the same result. Without loss of generality we can assume $l_2 = l_1+1$ and $k_2 = k_1+1$ and thus
\begin{gather*}
(\gamma_1, \gamma_2) = (\alpha_{l_1}, \gamma_2) + ( \alpha_{l_1+1}, \gamma_2) + \dots + (\alpha_\f, \gamma_2) + \dots + (\gamma_{k_1}, \gamma_2)\\
\hphantom{(\gamma_1, \gamma_2)}{}
= (\alpha_{l_1}, \alpha_{l_1+1})_\mathfrak{g'}
							 + ( \alpha_{l_1+1}, \alpha_{l_1+1})_\mathfrak{g'} + ( \alpha_{l_1+1}, \alpha_{l_1+2})_\mathfrak{g'}
							 + \cdots \\
\hphantom{(\gamma_1, \gamma_2)=}{}
							 +(\alpha_\f, \alpha_{\f-1})_\mathfrak{g'} + (\alpha_\f, \alpha_{f}) + (\alpha_\f, \alpha_{\f+1})_\mathfrak{g''}
							 +\cdots\\
\hphantom{(\gamma_1, \gamma_2)=}{}
+(\alpha_{k_1}, \alpha_{k_1}-1)_\mathfrak{g''} + (\alpha_{k_1},\alpha_{k_1})_\mathfrak{g''}+ (\alpha_{k_1}, \alpha_{k_1+1})_\mathfrak{g''}	.				
		\end{gather*}				
		The only term that contributes is $(\alpha_\f, \alpha_{\f-1})_\mathfrak{g'} + (\alpha_\f, \alpha_{f}) + (\alpha_\f, \alpha_{\f+1})_\mathfrak{g''}$ since the previous terms sum up to zero in $\mathfrak{g'}$, and the following terms sum up to zero in $\mathfrak{g''}$. Hence we have
		$(\gamma_1, \gamma_2)_{\g'}= (\gamma_1, \gamma_2)_{\g''}= -1.$
		Asking $m_{12}$ to be zero, means to ask \[-1 \cdot m' -1\cdot m'' + 1 = 0 \qquad \Rightarrow \qquad m' + m'' = 1.\]
\end{itemize}
To conclude, the only condition needed for the matrix $(m_{ij})$ to be a realizing solution is $m' + m'' = 1$. This condition matches with the formulation of $A(m,n)$ in terms of Nichols algebra diagram \cite[Table~C, row 2]{Heck06}, where $q'=q$ and $q''= -q^{-1}$. Indeed, if $m' + m'' = 1$ then $q'q''={\rm e}^{{\rm i}\pi m'}{\rm e}^{{\rm i}\pi m''} =-1$.
\begin{oss} We can write the simple roots in the standard chamber using as in \cite{Kac77} the standard basis $\epsilon_1, \dots, \epsilon_{m+1}, \delta_1, \dots, \delta_{n+1}$:
\begin{gather*}
 \{ \alpha_1 = \epsilon_1-\epsilon_2, \,\alpha_2= \epsilon_2-\epsilon_3, \,\dots,\,  \alpha_{m+1}= \epsilon_{m+1}-\delta_1, \\
 \alpha_{m+2}=\delta_1-\delta_2, \,\dots,\,\alpha_{m+n+1}=\delta_n-\delta_{n+1} \},
\end{gather*}
\end{oss}

\subsubsection*{$\boldsymbol{B(m,n)}$}
\begin{gather*}
\begin{tikzpicture}
	\draw (-0.6,0)--(1,0);
	\draw (-0.7,0) circle[radius=0.1cm] node[anchor=south]{$ q^{-4}$}
			(1.1,0) circle[radius=0.1cm] node[anchor=south]{$ q^{-4}$};
	\draw (0.2,0) node[anchor=south]{$ q^{4}$};
	\draw (1.2,0)--(1.7,0);
		\draw (2.2,0) node{$\cdots$};
	\draw (2.7,0)--(3.2,0);
	\draw (3.3,0) circle[radius=0.1cm] node[anchor=south]{$ -1$};
	\draw (3.4,0)--(3.9,0);
		\draw (4.4,0) node{$\cdots$};
	\draw (4.9,0)--(5.4,0);
	\draw (5.5,0) circle[radius=0.1cm] node[anchor=south]{$ q^{4}$};
	\draw (5.6,0)--(7.2,0);
	\draw (6.4,0) node[anchor=south]{$ q^{-4}$};
	\draw (7.3,0) circle[radius=0.1cm] node[anchor=south]{$ q^2$};
\end{tikzpicture}
\end{gather*}
The simple roots in the standard chamber are \[\alpha_1, \dots, \alpha_\f=\alpha_{n}, \dots, \alpha_{m+n}\]
with inner product matrix
\[
(\alpha_i, \alpha_j)=\begin{bmatrix}
4 & -2 & & & & \\
-2 & \ddots & \ddots & & & \\
 & \ddots & \ddots & \ddots & & \\
 & & \ddots & 0& \ddots & \\
 & & & \ddots & \ddots & -2 \\
 & & & & -2 & \hphantom{-} 2
\end{bmatrix}.\]
All the positive roots are
\begin{gather*}
\Delta_0 = \big\lbrace \alpha_l+ \dots +\alpha_k  \text{ with } l,k<\f,   \\
\hphantom{\Delta_0 = \big\lbrace}{}  \alpha_l+ \dots +\alpha_k \ \text{with } l,k > \f, \  k \neq m+n, \\
\hphantom{\Delta_0 = \big\lbrace}{} \alpha_l+ \dots +\alpha_{m+n} \ \text{with } l > \f,\\
\hphantom{\Delta_0 = \big\lbrace}{} \alpha_l+ \dots +2\alpha_{k}+\dots +2\alpha_{m+n} \ \text{with } l < \f,\  k\leq \f,\\
\hphantom{\Delta_0 = \big\lbrace}{}  \alpha_l+ \dots +2\alpha_{k}+\dots +2\alpha_{m+n} \ \text{with } l,k > \f \big\rbrace,\\
\Delta_1 = \big\lbrace \alpha_l+ \dots +\alpha_{m+n} \ \text{ with } l\leq \f, \\
\hphantom{\Delta_1 =\big\lbrace}{}
\alpha_l+ \dots +2\alpha_{k}+\dots +2\alpha_{m+n} \ \text{with } l < \f<k,\\
\hphantom{\Delta_1 =\big\lbrace}{}
  \alpha_l+ \dots +\alpha_{k} \  \text{with } l<\f<k, \  k\neq m+n \big\rbrace.
\end{gather*}

We now apply the lemmas of the previous section to determine possible conditions on~$m'$ and~$m''$ such that the matrix $(m_{ij})$ matrix defined as in Definition~\ref{mSuperLie} is a realizing solution.
\begin{itemize}\itemsep=0pt
	\item All the bosons which are not of the type $\gamma_{lk}:= \alpha_l+ \dots +2\alpha_{k}+\dots +2\alpha_{m+n}$, with $l < \f$, $k\leq \f$, are either in $\mathfrak{g'}$ or $\mathfrak{g''}$. By Lemma~\ref{onebossector}, we know they are always $m$-Cartan.
	\item For $\gamma_{lk}$, we need to explicitly ask condition~(\ref{cond7}).

The inner product is $(\gamma_{lk},\gamma_{lk})_{\g'}= -2$, $(\gamma_{lk},\gamma_{lk})_{\g''}= -4$.
	\begin{itemize}\itemsep=0pt
	\item $\gamma_{lk}$ is $m$-truncation if $2m' +4m'' =3$.
	\item $\gamma_{lk}$ is $m$-Cartan if $m'+m'' = 1$.
	\end{itemize}
	\item All the fermions which are not of the type $\gamma_{l}:= \alpha_l+ \dots +\alpha_{m+n}$, are isotropic and have $f(\alpha)= \pm 1$. By Lemma \ref{isofermions} we then know that if they are not strong orthogonal to any other root they are $m$-truncation.
	\item For $\gamma_{l}$, we need to explicitly ask condition (\ref{cond7}).

The inner product is $(\gamma_{l},\gamma_{l})_{\g'}=0$, $(\gamma_{l},\gamma_{l})_{\g''}= -1$.
	\begin{itemize}\itemsep=0pt
	\item $\gamma_{l}$ is $m$-truncation if $m'' =0$.
	\item $\gamma_{l}$ is $m$-Cartan if $m'+m'' = 1$.
	\end{itemize}
\item We now focus on the case of strong orthogonal fermions. Let us consider the fermions:
\begin{gather*} \big\lbrace \gamma_1 := \alpha_{l_1}+ \dots +\alpha_{m+n},\,
 \gamma_2:=\alpha_{l_2}+ \dots +2\alpha_{k_2}+\dots +2\alpha_{m+n},\\
 \qquad
  \gamma_3:= \alpha_{l_3}+ \dots +\alpha_{k_3} \big\rbrace.			
\end{gather*}
	The fermions $\gamma_1$ and $\gamma_2$ are strong orthogonal iff $l_1\neq l_2$;\\
	The fermions $\gamma_2$ and $\gamma_3$ are strong orthogonal iff $l_2\neq l_3$ or $k_2\neq k_3+1$;\\
	The fermions $\gamma_1$ and $\gamma_3$ are strong orthogonal iff $l_1\neq l_3$;\\
	Two fermions of type $\gamma_2$ are strong orthogonal for different $l_2$ and $k_2$;\\
	Two fermions of type $\gamma_3$ are strong orthogonal for different $l_3$ and $k_3$;
	Asking the condition $m_{ij}=0$ for those cases, we find again the condition $m'+m''=1$.
\end{itemize}		
In conclusion, the only condition needed for the matrix $(m_{ij})$ to be a realizing solution is $m' + m'' = 1$. If this condition is satisfied the bosons with $f(\alpha)=2$ as well as the non isotropic fermions are $m$-Cartan. If moreover $m'=m''=\frac{1}{2}$ then the bosons with $f(\alpha)=2$ are also $m$-truncation.

\begin{oss} We can write the simple roots in the standard chamber using as in~\cite{Kac77} the standard basis $\epsilon_1, \dots, \epsilon_{m}, \delta_1, \dots, \delta_{n}$:
\begin{gather*}
 \{ \alpha_1 = \delta_1-\delta_2,  \alpha_2= \delta_2-\delta_3, \dots,  \alpha_{n}= \delta_n-\epsilon_{1},
  \alpha_{n+1}=\epsilon_1-\epsilon_2, \dots, \alpha_{m+n}=\epsilon_m \}.
\end{gather*}
The bosons with $f(\alpha)=2$ will be of the form $\delta_i+\delta_j$, while the non isotropic fermions will be $\delta_i$.
\end{oss}

\subsubsection*{$\boldsymbol{C(n)}$}
\begin{gather*}
\begin{tikzpicture}
	\draw (-0.6,0)--(1,0);
	\draw (-0.7,0) circle[radius=0.1cm] node[anchor=south]{$ -1$}
			(1.1,0) circle[radius=0.1cm] node[anchor=south]{$ q^2$};
	\draw (0.2,0) node[anchor=south]{$ q^{-2}$};
	\draw (1.2,0)--(1.7,0);
		\draw (2.2,0) node{$\cdots$};
	\draw (2.7,0)--(3.2,0);
	\draw (3.3,0) circle[radius=0.1cm] node[anchor=south]{$ q^2$};
	\draw (3.4,0)--(5,0);
	\draw (5.1,0) circle[radius=0.1cm] node[anchor=south]{$ q^2$};
	\draw (4.2,0) node[anchor=south]{$ q^{-2}$};
	\draw (5.2,0)--(6.8,0);
	\draw (6,0) node[anchor=south]{$ q^{-4}$};
	\draw (6.9,0) circle[radius=0.1cm] node[anchor=south]{$ q^{4}$};
\end{tikzpicture}
\end{gather*}
The simple roots in the standard chamber are
\[\alpha_\f=\alpha_1, \dots,\alpha_{n}\]
with inner product matrix
\[
(\alpha_i, \alpha_j)=\begin{bmatrix}
\hphantom{-}0 & -1 & & & & \\
-1 & \hphantom{-} 2 & \ddots & & & \\
 & \ddots & \ddots & \ddots & & \\
 & & \ddots & \ddots & -1 & \\
 & & & -1 & \hphantom{-} 2 & -2 \\
 & & & & -2 &\hphantom{-} 4
\end{bmatrix}.\]
All the positive roots are
\begin{gather*}
\Delta_0 = \big\lbrace \alpha_l+ \dots +\alpha_k \ \text{with } l \neq 1, \ k \neq n, \\
\hphantom{\Delta_0 = \big\lbrace}{}
 \alpha_l+ \dots +2\alpha_{k}+\dots +2\alpha_{n-1}+\alpha_n \ \text{with } l \neq 1, \ k \neq n, \\
\hphantom{\Delta_0 = \big\lbrace}{}
\alpha_l+ \dots +\alpha_{n} \ \text{with } l \neq 1,\\
\hphantom{\Delta_0 = \big\lbrace}{}
 2\alpha_l+ \dots +2\alpha_{n-1}+\alpha_{n} \ \text{with } l \neq 1 \big\rbrace,\\
\Delta_1 =  \big\lbrace \alpha_1+ \dots +\alpha_{n}, \\
\hphantom{\Delta_1 =  \big\lbrace}{}
 \alpha_1+ \dots +\alpha_{k} \ \text{with } k \neq 1,\\
\hphantom{\Delta_1 =  \big\lbrace}{}
\alpha_1+ \dots +2\alpha_{k}+\dots +2\alpha_{n-1}+\alpha_n \  \text{with } k \neq n \big\rbrace.			
\end{gather*}

We now apply the lemmas of the previous section to determine possible conditions on $m'$ such that the matrix $(m_{ij})$ defined as in Definition~\ref{mSuperLie} is a realizing solution.
\begin{itemize}\itemsep=0pt
	\item Since there is just one bosonic side and there are no bosons with $f(\alpha)>0$ it is obvious that all the bosons are $m$-Cartan.
	\item All the fermions are isotropic, not strong orthogonal to each other, and have $f(\alpha)= \pm 1$. By Lemma~\ref{isofermions} we then know that they are $m$-truncation.
\end{itemize}
To conclude, the matrix $(m_{ij})$ is always a realizing solution.

\begin{oss} We can write the simple roots in the standard chamber using as in \cite{Kac77} the standard basis $\epsilon_1, \delta_ 1, \dots, \delta_{n-1}$:
\[  \lbrace \alpha_1 = \epsilon_1-\delta_1,  \alpha_2= \delta_1-\delta_2, \dots, \alpha_{n-1}=\delta_{n-2}-\delta_{n-1},  \alpha_n= 2\delta_{n-1} \rbrace. \]
\end{oss}

\subsubsection*{$\boldsymbol{D(m,n)}$}
\begin{gather*}
\begin{tikzpicture}
	\draw (-0.6,0)--(1,0);
	\draw (-0.7,0) circle[radius=0.1cm] node[anchor=south]{$ q^{-2}$}
			(1.1,0) circle[radius=0.1cm] node[anchor=south]{$ q^{-2}$};
	\draw (0.2,0) node[anchor=south]{$ q^2$};
	\draw (1.2,0)--(1.7,0);
	\draw (2.2,0) node{$\cdots$};
	\draw (2.7,0)--(3.2,0);
	\draw (3.3,0) circle[radius=0.1cm] node[anchor=south]{$ -1$};
	\draw (3.4,0)--(3.9,0);
	\draw (4.4,0) node{$\cdots$};
	\draw (4.9,0)--(5.4,0);
	\draw (5.5,0) circle[radius=0.1cm] node[anchor=south]{$ q^2$};
	\draw (5.6,0)--(7.2,0);
	\draw (6.4,0) node[anchor=south]{$ q^{-2}$};
	\draw (7.3,0) circle[radius=0.1cm] node[anchor=south]{$ q^2$};
	\draw (7.4,0)--(8.5,1.1);
	\draw (7.4,0)--(8.5,-1.1);
	\draw (8.56,1.17) circle[radius=0.1cm] node[anchor=south]{$ q^2$};
	\draw (8.2,1) node[anchor=east]{$ q^{-2}$};
	\draw (8.56,-1.17) circle[radius=0.1cm] node[anchor=north]{$ q^2$};
	\draw (8.2,-1) node[anchor=east]{$ q^{-2}$};
\end{tikzpicture}
\end{gather*}
The simple roots in the standard chamber are  \[\alpha_1, \dots,\alpha_{n}= \alpha_\f, \dots, \alpha_{n+m}\]
with inner product matrix
\[
(\alpha_i, \alpha_j)=\begin{bmatrix}
2 & -1 & & & & \\
-1 & \ddots & \ddots & & & \\
 & \ddots & 0 & \ddots & & \\
 & & \ddots & \hphantom{-}2& -1 & -1 \\
 & & & -1 & \hphantom{-}2 & \hphantom{-} 0 \\
 & & & -1 & \hphantom{-}0 & \hphantom{-}2
\end{bmatrix}.\]
All the positive roots are
\begin{gather*}
\Delta_0 =  \big\lbrace \alpha_l+ \dots +\alpha_k \ \text{with } l,k<\f,  \\
\hphantom{\Delta_0 =  \big\lbrace}{}  \alpha_l+ \dots +\alpha_k \ \text{with } l,k > \f, \\
\hphantom{\Delta_0 =  \big\lbrace}{}\alpha_l+ \dots +\alpha_{m+n-2}+\alpha_{m+n} \ \text{with } l > \f,\\
\hphantom{\Delta_0 =  \big\lbrace}{}\alpha_l+ \dots +2\alpha_{k}+\dots +2\alpha_{m+n-2} +\alpha_{m+n-1}+\alpha_{m+n} \ \text{with } l < \f, \  k\leq \f,\\
\hphantom{\Delta_0 =  \big\lbrace}{}\alpha_l+ \dots +2\alpha_{k}+\dots +2\alpha_{m+n-2} +\alpha_{m+n-1}+\alpha_{m+n} \ \text{with } l,k > \f, \\
\hphantom{\Delta_0 =  \big\lbrace}{}2\alpha_l+ \dots +2\alpha_{k}+\dots +2\alpha_{m+n-2} +\alpha_{m+n-1}+\alpha_{m+n} \ \text{with } l < \f,\  k\leq \f\big\rbrace,\\
\Delta_1 =  \big\lbrace \alpha_l+ \dots +\alpha_{k} \ \text{with } l\leq \f \leq k,\\
\hphantom{\Delta_1 =  \big\lbrace}{}
\alpha_l+ \dots +\alpha_{n+m-2}+\alpha_{n+m} \ \text{with } l\leq \f, \\
\hphantom{\Delta_1 =  \big\lbrace}{} \alpha_l+ \dots +2\alpha_{k}+\dots +2\alpha_{m+n-2}+\alpha_{n+m-1}+\alpha_{n+m} \ \text{with } l < \f<k\big\rbrace.
\end{gather*}
We now apply the lemmas of the previous section to determine possible conditions on~$m'$ and~$m''$ such that the matrix $(m_{ij})$ defined as in Definition~\ref{mSuperLie} is a realizing solution.
\begin{itemize}\itemsep=0pt
	\item All bosons except the IV or VI type in the list, are either in $\mathfrak{g'}$ or $\mathfrak{g''}$. Then, thanks to Lemma~\ref{onebossector}, we know they are always $m$-Cartan.
	\item The bosons of type IV have inner product $-2$ in $\mathfrak{g'}$ and $-4$ in $\mathfrak{g''}$:
	\begin{itemize}\itemsep=0pt
		\item it is $m$-truncation if $2m' +4m'' =3$,
	\item it is $m$-Cartan if $m'+m'' = 1$.
	\end{itemize}
	 The bosons of type VI have inner product $-4$ in $\mathfrak{g''}$:
	\begin{itemize}\itemsep=0pt
		\item it is $m$-truncation if $4m'' =3$,
	\item it is $m$-Cartan if $m'+m'' = 1$.
	\end{itemize}
	\item All fermions are isotropic and have $f(\alpha)= \pm 1$. By Lemma~\ref{isofermions} we then know that if they are not strong orthogonal to any other root they are $m$-truncation.
	\item There are many possibility for two fermions to be strong orthogonal. Asking the condition $m_{ij} = 0$ for those cases, we find again the condition $m' + m'' = 1$.
\end{itemize}
In conclusion, the only condition needed for the matrix $(m_{ij})$ to be a realizing solution is $m' + m'' = 1$. If this condition is satisfied the bosons with $f(\alpha)=2$ are $m$-Cartan. If moreover $m'=m''=\frac{1}{2}$ then the boson of type IV are also $m$-truncation. Instead if $m'=\frac{1}{4}$, $m''=\frac{3}{4}$ then the boson of type VI are also $m$-truncation.

\begin{oss} As in the previous cases the condition $m' + m'' = 1$ matches with the formulation of $D(m,n)$ in terms of Nichols algebra diagram \cite[Table~C, row~10]{Heck06}, where $q'= q$ and $q'= q^{-1}$.
\end{oss}
\begin{oss} We can write the simple roots in the standard chamber using as in \cite{Kac77} the standard basis $\epsilon_1, \dots, \epsilon_{m}, \delta_1, \dots, \delta_{n}$:
\begin{gather*}
 \{ \alpha_1 = \delta_1-\delta_2, \dots, \alpha_{n}= \delta_n-\epsilon_{1}, \alpha_{n+1}=\epsilon_1-\epsilon_2,  \dots,\\
\qquad  \alpha_{m+n-1}=\epsilon_{m-1}-\epsilon_m,  \alpha_{m+n}=\epsilon_{m-1}+\epsilon_m \}.
\end{gather*}
The bosons of type IV will be of the form $\delta_i+\delta_j$, while the one of type VI will be of the form~$2\delta_i$.
\end{oss}

\subsection{Sporadic cases}
\subsubsection[$G(3)$]{$\boldsymbol{G(3)}$}\label{G3}
\begin{gather*}
\begin{tikzpicture}
	\draw (0,0)--(1.6,0) (1.8,0)--(3.4,0);
	\draw (-0.1,0) circle[radius=0.1cm] node[anchor=south]{$ -1$}
			(1.7,0) circle[radius=0.1cm] node[anchor=south]{$ q^2$}
			(3.5,0) circle[radius=0.1cm] node[anchor=south]{$ q^6$};
	\draw (0.8,0) node[anchor=south]{$ q^{-2}$};
	\draw (2.6,0) node[anchor=south]{$ q^{-6}$};
	\end{tikzpicture}
\end{gather*}
The simple roots in the standard chamber are $\lbrace \alpha_\f =\alpha_1, \alpha_2, \alpha_3\rbrace$ with inner product
\[
(\alpha_i, \alpha_j)= \begin{bmatrix}
 \hphantom{-}0 & -1 & \hphantom{-}0\\
 -1 & \hphantom{-}2 & -3\\
 \hphantom{-}0 & -3 & \hphantom{-}6
\end{bmatrix}.
\]
There is only one bosonic part $\mathfrak{g'}$ and the positive roots are
\begin{gather*}
\lbrace  \alpha_1,  \alpha_2,  \alpha_3,  \alpha_{12},  \alpha_{23},  \alpha_{223},  \alpha_{123},  \alpha_{1223},
 \alpha_{12223},  \alpha_{2223},  \alpha_{22233},  \alpha_{1222233},  \alpha_{122233} \rbrace.
\end{gather*}

The matrix $(m_{ij})$ is given by
\[
\big(m_{ij}^\mathrm{I}\big)= \begin{bmatrix}
1 & -m & 0\\
 -m & 2m & -3m\\
0 & -3m & 6m
\end{bmatrix}.
\]
\begin{itemize}\itemsep=0pt
	\item Since there is just one bosonic side and there are no bosons with $f(\alpha)>0$, it is obvious that all the bosons satisfy are $m$-Cartan.
	\item All the fermions, except for $\alpha_{1223}$, are isotropic and have $f(\alpha)= \pm 1$. By Lemma~\ref{isofermions} we then know that they are $m$-truncation.
	\item The fermion $\alpha_{1223}$ is $m$-Cartan without further assumptions on $m$ by Corollary~\ref{Mcondotherchamb} and explicit computation
	(hence this is a rare example that a fermionic root in a Lie superalgebra can be $m$-Cartan and not $m$-truncation).
	\item There are no pairs of strong orthogonal fermions.
\end{itemize}
To conclude, the matrix $(m_{ij})$ is a realizing solution for all $m$.
This construction realize the Nichols algebra $\mathcal{B}(q_{ij})$ described row 7 of Table~2 in~\cite{Hecklist} when $q\neq \pm 1$, $q\not\in\mathbb{G}_3$.

For convenience we show explicitly all the reflections of the matrix~$(m_{ij})$:
Reflecting $\big(m_{ij}^\mathrm{I}\big)$ with $\r_1$ around $\alpha_1$ we find the following
\[
\big(m_{ij}^\mathrm{II}\big)= \begin{bmatrix}
 1 & -1+m & 0\\
 -1+m &  1 & -3m \\
 0 & -3m & 6m
\end{bmatrix}
\qquad \text{in basis } \{-\alpha_1,\alpha_{12},\alpha_3\}.
\]
Reflecting it with $\r_2$ around $\alpha_{12}$ we find the following
\[
\big(m_{ij}^\mathrm{III}\big)= \begin{bmatrix}
 2m & -m & -2m\\
 -m &  1 & -1+3m \\
 -2m & -1+3m & 1
\end{bmatrix}\qquad
\text{in basis } \{\alpha_2,-\alpha_{12},\alpha_{123}\}.
\]
Reflecting it around with $\r_3$ around $\alpha_{123}$ and permuting the indices we find
\[
\big(m_{ij}^\mathrm{IV}\big)= \begin{bmatrix}
6m & -3m & 0\\
-3m &  1 & -1+2m \\
0 & -1+2m & 1-2m.
\end{bmatrix}\qquad
\text{in basis } \{\alpha_{3},-\alpha_{123},\alpha_{1223}\}.
\]

\begin{oss}
If $q^2\in\mathbb{G}_6$, $\alpha_3$ is both $q$-Cartan and $q$-truncation. When it is $m$-truncation we find
\begin{gather*}
\begin{tikzpicture}
	\draw (0,0)--(1.4,0) (1.6,0)--(3,0);
	\draw (-0.1,0) circle[radius=0.1cm] node[anchor=south]{$ -1$} node[anchor=north]{$ 1$}
			(1.5,0) circle[radius=0.1cm] node[anchor=south]{$ \zeta$}node[anchor=north]{$ 2m$}
			(3.1,0) circle[radius=0.1cm] node[anchor=south]{$ -1$}node[anchor=north]{$ 1$};
	\draw (0.7,0) node[anchor=south]{$ \zeta^{-2}$}node[anchor=north]{$ -2m$};
	\draw (2.3,0) node[anchor=south]{$ -1$} node[anchor=north]{$ -6m$};
	\end{tikzpicture}
	\end{gather*}
	with $\zeta\in \mathbb{G}_6$.
This is a solution iff $m=\frac{1}{6}$. But for this value of $m$, $\alpha_3$ is also $m$-Cartan and thus this is not a new solution.
\end{oss}

\begin{oss}The roots can be expressed by
\[ \alpha_1 =\alpha_\f = \delta+ \epsilon_1, \qquad \alpha_2 = \epsilon_2, \qquad \alpha_3 = \epsilon_3-\epsilon_2.\]
\end{oss}

\subsubsection*{$\boldsymbol{F(4)}$}
\begin{gather*}
\begin{tikzpicture}
	\draw (0,0)--(1.6,0) (1.8,0)--(3.4,0) (3.6,0)--(5.2,0);
	\draw (-0.1,0) circle[radius=0.1cm] node[anchor=south]{$ q^{4}$}
			(1.7,0) circle[radius=0.1cm] node[anchor=south]{$ q^{4}$}
			(3.5,0) circle[radius=0.1cm] node[anchor=south]{$ q^2$}
			(5.3,0) circle[radius=0.1cm] node[anchor=south]{$ -1$};
	\draw (0.8,0) node[anchor=south]{$ q^{-4}$};
	\draw (2.6,0) node[anchor=south]{$ q^{-4}$};
	\draw (4.4,0) node[anchor=south]{$ q^{-2}$};
	\end{tikzpicture}
\end{gather*}
The simple roots in the standard chamber are $\lbrace \alpha_1, \alpha_2, \alpha_3, \alpha_4 =\alpha_\f \rbrace$ with inner product
\[
(\alpha_i, \alpha_j)= \begin{bmatrix}
 \hphantom{-}4 & -2 & & \\
 -2 & \hphantom{-}4 & -2 &\\
 & -2 & \hphantom{-}2 & -1\\
 & & -1 & \hphantom{-}0
\end{bmatrix}.
\]
There are $18$ roots: $9$ in the one bosonic part $\mathfrak{g'}$ of type $C_3$, furthermore $8$ fermionic roots
\begin{align*}
&\{ \alpha_{4},\; \alpha_{34},\;\alpha_{234},\;\alpha_{1234},\;
\alpha_{2334},\;\alpha_{12334},\;\alpha_{122334},\;\alpha_{1223334}\}
\end{align*}
and one bosonic root $\alpha_{12233344}$ with $f(\alpha)=2$.
\begin{itemize}\itemsep=0pt
	\item All bosons in $\mathfrak{g'}$ are automatically $m$-Cartan.
	 \item The boson $\alpha_{12233344}$ is $m$-Cartan without further assumptions by Corollary~\ref{Mcondotherchamb} and explicit computation.
	\item All fermions are isotropic and have $f(\alpha)= \pm 1$. By Lemma~\ref{isofermions} we then know they are $m$-truncation.
	\item We have two couples of strong orthogonal fermions:
	\[\lbrace \alpha_{34}, \alpha_{122334} \rbrace, \qquad \lbrace \alpha_{234}, \alpha_{12334} \rbrace\]
	which give the condition $m= \frac{1}{3}$.
\end{itemize}
To conclude, the condition for the matrix $(m_{ij})$ to be a realizing solution is $m= \frac{1}{3}$.

\subsubsection*{$\boldsymbol{D(2,1;\alpha}$)}
\begin{gather*}
\begin{tikzpicture}
\draw (5.5,0)--(4,-2.5)--(7,-2.5)--(5.5,0);
	\draw (5.5,0.1) circle[radius=0.1cm]
			(3.9,-2.5) circle[radius=0.1cm]
			(7.1,-2.5) circle[radius=0.1cm];
	\draw (5.5,-0.2) node[anchor=north]{$ -1$};
	\draw (4.3,-2.1) node[anchor=north]{$ -1 $};
	\draw (6.6,-2.1) node[anchor=north]{$ -1$};
	\draw (4.6,-1.5) node[anchor=west]{$ q'$};
	\draw (6.3,-1.5) node[anchor=east]{$ q''$};
	\draw (5.6,-2.5) node[anchor=south]{$ q'''$};
\draw (5.5,0.1) node[anchor=south]{$ 1$}
			(3.9,-2.5) node[anchor=east]{$ 1$}
			(7.1,-2.5) node[anchor=west]{$ 1$};
	\draw (4.6,-1.3) node[anchor=east]{$ m'-2$};
	\draw (6.3,-1.3) node[anchor=west]{$ m''-2$};
	\draw (5.5,-2.7) node[anchor=north]{$ m'''-2$};	
	\end{tikzpicture}
\end{gather*}

The simple roots in the standard chamber are $\lbrace \alpha_1, \alpha_2 = \alpha_\f, \alpha_3 \rbrace$ with inner product:
\[
(\alpha_i, \alpha_j)= \begin{bmatrix}
 \hphantom{-}2 & -2 & \hphantom{-}0\\
 -2 & \hphantom{-}0 & -2\\
 \hphantom{-}0 & -2 & \hphantom{-}2
\end{bmatrix}.
\]
The positive roots are
\[\lbrace \alpha_1,   \alpha_2 ,  \alpha_3,  \alpha_{12},  \alpha_{23}, \alpha_{123},  \alpha_{1223}\rbrace.\]
Reflecting the diagram around, say, $\alpha_2$ (the system is completely symmetric in the three roots), we obtain
\begin{gather*}
\begin{tikzpicture}
	\draw (0,0)--(1.4,0) (1.6,0)--(3,0);
	\draw (-0.1,0) circle[radius=0.1cm] node[anchor=south]{$ q'$} node[anchor=north]{$ m'$}
			(1.5,0) circle[radius=0.1cm] node[anchor=south]{$ -1$}node[anchor=north]{$ 1$}
			(3.1,0) circle[radius=0.1cm] node[anchor=south]{$ q'''$}node[anchor=north]{$ m'''$};
	\draw (0.7,0) node[anchor=south]{$ (q')^{-1}$}node[anchor=north]{$ -m'$};
	\draw (2.3,0) node[anchor=south]{$ (q''')^{-1}$} node[anchor=north]{$ -m'''$};
	\end{tikzpicture}
	\end{gather*}
	
Exception (4) of Lemma~\ref{3 conditions II} appears. Imposing that the first and the third roots are not connected we find the condition $m'+m''+m''' =2$.
This corresponds to the condition \mbox{$q'\cdot q''\cdot q'''=1$} of rows~9,~10,~11) in Table~2 of~\cite{Hecklist}. Hence these matrices $(m_{ij})$ are realizing solution.

\section{Other cases in rank 2}\label{section8}
\subsection[Construction of $(m_{ij})$]{Construction of $\boldsymbol{(m_{ij})}$}\label{Strange}

In this section we are going to present the examples of rank $2$ Nichols algebras that do not belong to the Cartan and super Lie study of the previous two sections. We use abbreviations such as $q_{112,112}:=q(\alpha_{112},\alpha_{112})$ and similarly $m_{112,112}:=(\alpha_{112},\alpha_{112})$.

\subsubsection*{Heckenberger row 6}
This corresponds to a $\Z_3$-(color-)Lie algebra \cite{AAB14, Yam07}. In Table~1 in~\cite{Hecklist} it is described by two diagrams:
\begin{gather*}
\begin{tikzpicture}
	\draw (0,0)--(1.6,0);
	\draw (-0.1,0) circle[radius=0.1cm] node[anchor=south]{$ \zeta$}
			(1.7,0) circle[radius=0.1cm] node[anchor=south]{$ q^2$};
	\draw (0.8,0) node[anchor=south]{$ q^{-2}$};
	\draw (3.6,0)--(5.2,0);
	\draw (3.5,0) circle[radius=0.1cm] node[anchor=south]{$ \zeta$}
			(5.3,0) circle[radius=0.1cm];
	\draw (5.5,0) node[anchor=south]{$ \zeta q^{-2}$};
	\draw (4.4,0) node[anchor=south]{$ \zeta^{-1}q^2$};
\node at (0.8,-0.2) {$\mathrm{I}$};
\node at (4.4,-0.2) {$\mathrm{II}$};
\end{tikzpicture}
\end{gather*}
where $\zeta \in \mathbb{G}_{3}$ and $q^2 \neq 1$, $\zeta$, $\zeta^2$ and with respectively simple roots: \[ \mathrm{I}\colon \ \lbrace\alpha_1, \alpha_2\rbrace, \qquad \mathrm{II}\colon \ \lbrace-\alpha_1, \alpha_{112}\rbrace. \]
There is just one associate Cartan matrix:
\[(c_{ij})= \begin{bmatrix}
\hphantom{-}2 & -2\\
 -1 & \hphantom{-}2
\end{bmatrix}.
\]
The set of positive roots is $\lbrace \alpha_1, \alpha_2, \alpha_{12}, \alpha_{112}\rbrace$ where $\alpha_2$ and $\alpha_{112}$ are only $q$-Cartan while the others are only $q$-truncation.
\begin{prop}
The following matrices $(m_{ij})$ are realizing solutions of the given braiding and its reflections:
\[\big(m_{ij}^{\mathrm{I}}\big)= \begin{bmatrix}
 \frac{2}{3} & -m\\
 -m & 2m
\end{bmatrix}, \qquad
\big(m_{ij}^{\mathrm{II}}\big)= \begin{bmatrix}
 \frac{2}{3} & -\frac{4}{3} +m\\
 -\frac{4}{3} +m & \frac{8}{3} -2m
\end{bmatrix}.\]
\end{prop}

\begin{proof}First we check that condition (\ref{cond7})B is satisfied for~$\alpha_1$:
\[m_{11} = \frac{2}{1-c_{12}} = \frac{2}{3}\]
and condition (\ref{cond7}A) is satisfied for $\alpha_{22}$ and $\alpha_{112}$:
\begin{gather*}
m_{22,22} = \frac{2m_{12}}{c_{21}} = 2m,\\
m_{112,112} = \frac{2m_{112,-1}}{c_{112,1}} = \frac{8}{3} -2m.
\end{gather*}
We then check that the reflection around $\alpha_1$ sends one matrix $(m_{ij})$ to the other as follows
\begin{gather*}
\big(m_{ij}^{\mathrm{II}}\big) = \begin{bmatrix}
 \frac{2}{3} & -\frac{4}{3} +m\\
 -\frac{4}{3} +m & \frac{8}{3} -2m
\end{bmatrix} =
\begin{bmatrix}
 -1 & 2\\
 \hphantom{-}0 & 1
\end{bmatrix}^{\rm T}
\begin{bmatrix}
 \frac{2}{3} & -m\\
 -m & 2m
\end{bmatrix}
\begin{bmatrix}
 -1 & 2\\
 0 & 1
\end{bmatrix} = \r_1 \big(m_{ij}^{\mathrm{I}}\big).\tag*{\qed}
\end{gather*}\renewcommand{\qed}{}
\end{proof}

\begin{oss}\label{Heck6trunc}
When $q^2\in \mathbb{G}_2$, the root $\alpha_2$ is $q$-Cartan \textit{and} $q$-truncation.
When we assume it is $m$-truncation, we get the additional solution:
\begin{gather*}
\big(m_{ij}^{\mathrm{I}}\big)= \begin{bmatrix}
 \frac{2}{3} & -m\\
 -m & 1
\end{bmatrix}, \qquad\!
\big(m_{ij}^{\mathrm{II}}\big)= \begin{bmatrix}
 \frac{2}{3} & -\frac{4}{3} +m\\
 -\frac{4}{3} +m & \frac{11}{3} -4m
\end{bmatrix}, \qquad\!
\big(m_{ij}^{\mathrm{III}}\big)= \begin{bmatrix}
 \frac{5}{3}-2m & m-1\\
 m-1 & 1
\end{bmatrix}.
\end{gather*}
with III: $\lbrace\alpha_{12}, -\alpha_2\rbrace$.
The root $\alpha_{112}$ is never $m$-truncation and it is $m$-Cartan iff $m=\frac{1}{2}$. But for this value of $m$, $\alpha_2$ is also $m$-Cartan and thus this is \emph{not} a new solution.
\end{oss}

Since this is the only other series in rank $2$ and the only other example with nontrivial behaviour of the Nichols algebra relations, we briefly discuss them: The relations are given in~\cite{Helb10} and \cite[Section~7.2]{AA17}. For $q\neq -1$ the truncation and Serre relations are the only defining relations for this Nichols algebra, for $q=-1$ there is an additional relation $[x_{112},x_{12}]_q=0$
\begin{prop}\label{prop_colorLieNA}\quad
		\begin{enumerate}\itemsep=0pt
	\item[$1.$] For $0<m<\frac{1}{2}$ the condition $|m_{ii}|<1$ holds and $(m_{ij})$ is positive definite, so in this we are in the subpolar region and all Nichols algebra relations hold.
	\item[$2.$] The truncation relation for the root $\alpha_1$ always holds. The truncation relation for $\alpha_2$ holds iff $m>0$. In chamber II the truncation relation for $\alpha_{112}$ holds if $m<\frac{4}{3}$. We make no assertion about the truncation relation for the non-simple root vectors.
	\item[$3.$] The quantum Serre relation always holds.
	\item[$4.$] We make no assertion about the additional relation.
\end{enumerate}
\end{prop}

Note that the classification Lemma \ref{superlieclassification} from Lie superalgebras extends to this case mutatis mutandis.

\subsubsection*{Heckenberger row 9}
This case of Table~1 in \cite{Hecklist} is described by three diagrams:
\begin{gather*}
\begin{tikzpicture}
	\draw (0,0)--(1.6,0);
	\draw (-0.1,0) circle[radius=0.1cm] node[anchor=south]{$ -\zeta^2$}
			(1.7,0) circle[radius=0.1cm] node[anchor=south]{$ -\zeta^{2}$};
	\draw (0.8,0) node[anchor=south]{$ \zeta$};
	\draw (3.6,0)--(5.2,0);
	\draw (3.5,0) circle[radius=0.1cm] node[anchor=south]{$ -\zeta^2$}
			(5.3,0) circle[radius=0.1cm] node[anchor=south]{$ -1$};
	\draw (4.4,0) node[anchor=south]{$ \zeta^{3}$};
	\draw (7.2,0)--(8.8,0);
	\draw (7.1,0) circle[radius=0.1cm] node[anchor=south]{$ -\zeta^{-1}$}
			(8.9,0) circle[radius=0.1cm] node[anchor=south]{$ -1$};
	\draw (8,0) node[anchor=south]{$ -\zeta^3$};
\node at (0.8,-0.2) {$\mathrm{I}$};
\node at (4.4,-0.2) {$\mathrm{II}$};
\node at (8,-0.2) {$\mathrm{III}$};
\end{tikzpicture}
\end{gather*}
where $\zeta \in \mathbb{G}_{12}$ and with respectively simple roots:
\[ \mathrm{I}\colon \ \lbrace\alpha_1, \alpha_2\rbrace, \qquad \mathrm{II}\colon \ \lbrace-\alpha_1, \alpha_{112}\rbrace, \qquad \mathrm{III}\colon \ \lbrace\alpha_{12}, -\alpha_{122}\rbrace. \]
The associate Cartan matrices are
\[\big(c_{ij}^{\mathrm{I}}\big)= \begin{bmatrix}
 \hphantom{-}2 & -2\\
 -2 & \hphantom{-}2
\end{bmatrix}, \qquad
\big(c_{ij}^{\mathrm{II}}\big)= \begin{bmatrix}
 \hphantom{-}2 & -2\\
 -1 & \hphantom{-}2
\end{bmatrix}, \qquad
\big(c_{ij}^{\mathrm{III}}\big)= \begin{bmatrix}
 \hphantom{-}2 & -3\\
 -1 & \hphantom{-}2
\end{bmatrix}.
\]
The set of positive roots is $\lbrace \alpha_1, \alpha_2, \alpha_{12}, \alpha_{112}, \alpha_{122}\rbrace$ where $\alpha_{12}$ is only $q$-Cartan while the others are only $q$-truncation.
\begin{prop}
The following $m_{ij}$ matrices are realizing solutions of the given braiding and its reflections:
\[\big(m_{ij}^{\mathrm{I}}\big)= \begin{bmatrix}
 \frac{2}{3} & -\frac{7}{12}\vspace{1mm}\\
 -\frac{7}{12} & \frac{2}{3}
\end{bmatrix}, \qquad
\big(m_{ij}^{\mathrm{II}}\big)= \begin{bmatrix}
 \frac{2}{3} & -\frac{3}{4}\vspace{1mm}\\
 -\frac{3}{4} & 1
\end{bmatrix}, \qquad
\big(m_{ij}^{\mathrm{III}}\big)= \begin{bmatrix}
 \frac{1}{6} & -\frac{1}{4}\vspace{1mm}\\
 -\frac{1}{4} & 1
\end{bmatrix}. \]
\end{prop}

\begin{proof}
First we check that condition (\ref{cond7}B) is satisfied for all the roots:
\begin{gather*}
m_{11} = \frac{2}{1-c_{12}} = \frac{2}{3}, \qquad
m_{22} = \frac{2}{1-c_{21}} = \frac{2}{3}, \\
m_{112,112} = \frac{2}{1-c_{112,1}} = 1, \qquad
m_{122,122} = \frac{2}{1-c_{122,12}} = 1
\end{gather*}
and condition (\ref{cond7}A) is satisfied for the root $\alpha_{12}$:
\[m_{12,12} = \frac{2m_{-122,12}}{c_{12,112}} = \frac{1}{6}.\]
We then check that the reflections send the matrices $(m_{ij})$ above to each other:
\begin{gather*}
\big(m_{ij}^{\mathrm{II}}\big) =\begin{bmatrix}
 \hphantom{-}\frac{2}{3} & -\frac{3}{4}\vspace{1mm}\\
 -\frac{3}{4} & \hphantom{-}1
\end{bmatrix} =
\begin{bmatrix}
 -1 & 2\\
\hphantom{-}0 & 1
\end{bmatrix}^{\rm T}
\begin{bmatrix}
 \hphantom{-}\frac{2}{3} & -\frac{7}{12}\vspace{1mm}\\
 -\frac{7}{12} & \hphantom{-}\frac{2}{3}
\end{bmatrix}
\begin{bmatrix}
 -1 & 2\\
 \hphantom{-}0 & 1
\end{bmatrix} = \r_1 \big(m_{ij}^{\mathrm{I}}\big),\\
\big(m_{ij}^{\mathrm{III}}\big)=\begin{bmatrix}
 \hphantom{-}\frac{1}{6} & -\frac{1}{4}\vspace{1mm}\\
 -\frac{1}{4} & \hphantom{-}1
\end{bmatrix} =
\begin{bmatrix}
 1 & \hphantom{-}0\\
 1 & -1
\end{bmatrix}^{\rm T}
\begin{bmatrix}
\hphantom{-}\frac{2}{3} & -\frac{3}{4}\vspace{1mm}\\
 -\frac{3}{4} & \hphantom{-}1
\end{bmatrix}
\begin{bmatrix}
 1 & \hphantom{-}0\\
 1 & -1
\end{bmatrix} = \r_{122} \r_2 \big(m_{ij}^{\mathrm{I}}\big).\tag*{\qed}
\end{gather*}\renewcommand{\qed}{}
\end{proof}

\begin{cor}By formula \eqref{centralcharge} for rank~$2$, we have that the central charge of the system is $c= -126$.
\end{cor}

\begin{cor}Since $\big(m_{ij}^\mathrm{I}\big)$ is positive definite and has diagonal entries $|m_{ii}|\leq 1$, by Lemma~{\rm \ref{pos+small}} the screening algebra is the Nichols algebra.
\end{cor}

We conclude this case with a picture illustrating how the set of simple roots behave under reflections. We write~I,~II,~III, to indicate to which diagram do the simple roots in each case belong:
\begin{equation*}
 \begin{tikzcd}
 && \ar{ld}[swap]{\r_1} \substack{\{\alpha_1, \alpha_2\}^\mathrm{I} } \ar{rd}{\r_2}&&\\
 &\ar{d}[swap]{\r_{112}}\substack{\{-\alpha_1, \alpha_{112}\}^\mathrm{II} }&& \substack{\{\alpha_{122}, -\alpha_2\}^\mathrm{II} } \ar{d}{\r_{122}} \\
 &\ar{rd}[swap]{\text{sign swap}}\substack{\{\alpha_{12}, -\alpha_{112}\}^\mathrm{III} } && \substack{\{-\alpha_{122}, \alpha_{12}\}^\mathrm{III}}\ar{ld}{{\r_{12}}} \\
 && \substack{\{\alpha_{112}, -\alpha_{12}\}^\mathrm{III} }&&\\
 \end{tikzcd}
\end{equation*}

\subsubsection*{Heckenberger row 10}
This case of Table~1 in~\cite{Hecklist} is described by three diagrams:
\begin{gather*}
\begin{tikzpicture}
	\draw (0,0)--(1.6,0);
	\draw (-0.1,0) circle[radius=0.1cm] node[anchor=south]{$ -\zeta$}
			(1.7,0) circle[radius=0.1cm] node[anchor=south]{$ \zeta^{3}$};
	\draw (0.8,0) node[anchor=south]{$ \zeta^{-2}$};
	\draw (3.6,0)--(5.2,0);
	\draw (3.5,0) circle[radius=0.1cm] node[anchor=south]{$ \zeta^3$}
			(5.3,0) circle[radius=0.1cm] node[anchor=south]{$ -1$};
	\draw (4.4,0) node[anchor=south]{$ \zeta^{-1}$};
	\draw (7.2,0)--(8.8,0);
	\draw (7.1,0) circle[radius=0.1cm] node[anchor=south]{$ -\zeta^{2}$}
			(8.9,0) circle[radius=0.1cm] node[anchor=south]{$ -1$};
	\draw (8,0) node[anchor=south]{$ \zeta$};
\node at (0.8,-0.2) {$\mathrm{I}$};
\node at (4.4,-0.2) {$\mathrm{II}$};
\node at (8,-0.2) {$\mathrm{III}$};
\end{tikzpicture}
\end{gather*}
where $\zeta \in \mathbb{G}_{9}$ and with respectively simple roots: \[ \mathrm{I}\colon \ \lbrace\alpha_1, \alpha_2\rbrace, \qquad \mathrm{II}\colon \ \lbrace-\alpha_{2}, \alpha_{122}\rbrace, \qquad \mathrm{III}\colon \ \lbrace\alpha_{12}, -\alpha_{122}\rbrace. \]
The associate Cartan matrices are:
\[\big(c_{ij}^{\mathrm{I}}\big)= \begin{bmatrix}
 \hphantom{-}2 & -2\\
 -2 & \hphantom{-}2
\end{bmatrix}, \qquad
\big(c_{ij}^{\mathrm{II}}\big)= \begin{bmatrix}
 \hphantom{-}2 & -2\\
 -1 & \hphantom{-}2
\end{bmatrix}, \qquad
\big(c_{ij}^{\mathrm{III}}\big)= \begin{bmatrix}
 \hphantom{-}2 & -4\\
 -1 & \hphantom{-}2
\end{bmatrix}.
\]
The set of positive roots is $\lbrace \alpha_1, \alpha_2, \alpha_{12}, \alpha_{112}, \alpha_{122}, \alpha_{11122}\rbrace$ where $\alpha_1$ and $\alpha_{12}$ are only $q$-Cartan while the others are only $q$-truncation.

\begin{prop}
The following $m_{ij}$ matrices are realizing solutions of the given braiding and its reflections:
\[\big(m_{ij}^{\mathrm{I}}\big)= \begin{bmatrix}
 \hphantom{-}\frac{5}{9} & -\frac{5}{9}\vspace{1mm}\\
 -\frac{5}{9} & \hphantom{-}\frac{2}{3}
\end{bmatrix}, \qquad
\big(m_{ij}^{\mathrm{II}}\big)= \begin{bmatrix}
\hphantom{-} \frac{2}{3} & -\frac{7}{9}\vspace{1mm}\\
 -\frac{7}{9} & \hphantom{-}1
\end{bmatrix}, \qquad
\big(m_{ij}^{\mathrm{III}}\big)= \begin{bmatrix}
 \hphantom{-}\frac{1}{9} & -\frac{2}{9}\vspace{1mm}\\
 -\frac{2}{9} & \hphantom{-}1
\end{bmatrix}. \]
\end{prop}
\begin{proof}
We check that the roots $\lbrace \alpha_2, \alpha_{112}, \alpha_{122}, \alpha_{11122}\rbrace$ satisfy condition~(\ref{cond7}B), while the root $\alpha_1$ and $\alpha_{12}$ satisfy condition~(\ref{cond7}A).
We check that the reflections send one~$(m_{ij})$ to the other.
\end{proof}
\begin{cor}
By formula \eqref{centralcharge} for rank~$2$, we have that the central charge of the system is~$-\frac{1088}{5}$.
\end{cor}

\begin{cor}Since $\big(m_{ij}^I\big)$ is positive definite and has diagonal entries $|m_{ii}|\leq 1$, by Lemma~{\rm \ref{pos+small}} the screening algebra is the Nichols algebra.
\end{cor}

\subsubsection*{Heckenberger row 12}
This case of Table~1 in \cite{Hecklist} is described by three diagrams:
\begin{gather*}
\begin{tikzpicture}
	\draw (0,0)--(1.6,0);
	\draw (-0.1,0) circle[radius=0.1cm] node[anchor=south]{$ \zeta^2$}
			(1.7,0) circle[radius=0.1cm] node[anchor=south]{$ \zeta^{-1}$};
	\draw (0.8,0) node[anchor=south]{$ \zeta$};
	\draw (3.6,0)--(5.2,0);
	\draw (3.5,0) circle[radius=0.1cm] node[anchor=south]{$ \zeta^2$}
			(5.3,0) circle[radius=0.1cm] node[anchor=south]{$ -1$};
	\draw (4.4,0) node[anchor=south]{$ -\zeta^{-1}$};
	\draw (7.2,0)--(8.8,0);
	\draw (7.1,0) circle[radius=0.1cm] node[anchor=south]{$ \zeta$}
			(8.9,0) circle[radius=0.1cm] node[anchor=south]{$ -1$};
	\draw (8,0) node[anchor=south]{$ -\zeta$};
\node at (0.8,-0.2) {$\mathrm{I}$};
\node at (4.4,-0.2) {$\mathrm{II}$};
\node at (8.0,-0.2) {$\mathrm{III}$};
\end{tikzpicture}
\end{gather*}
where $\zeta \in \mathbb{G}_{8}$ and with respectively simple roots: \[ \mathrm{I}\colon \ \lbrace\alpha_1, \alpha_2\rbrace, \qquad \mathrm{II}\colon \ \lbrace -\alpha_{1}, \alpha_{1112}\rbrace, \qquad \mathrm{III}\colon \ \lbrace\alpha_{112}, -\alpha_{1112}\rbrace. \]
There is just one associate Cartan matrix:
\[(c_{ij})= \begin{bmatrix}
 \hphantom{-}2 & -3\\
 -1 & \hphantom{-}2
\end{bmatrix}.
\]
The set of positive roots is $\lbrace \alpha_1, \alpha_2, \alpha_{12}, \alpha_{112}, \alpha_{1112}, \alpha_{11122}\rbrace$ where $\alpha_2$ and $\alpha_{112}$ are only $q$-Cartan while the others are only $q$-truncation.
\begin{prop}
The following $m_{ij}$ matrices are realizing solutions of the given braiding and its reflections:
\[\big(m_{ij}^{\mathrm{I}}\big)= \begin{bmatrix}
 \hphantom{-}\frac{1}{2} & -\frac{7}{8}\vspace{1mm}\\
 -\frac{7}{8} & \hphantom{-}\frac{7}{4}
\end{bmatrix}, \qquad
\big(m_{ij}^{\mathrm{II}}\big)= \begin{bmatrix}
 \hphantom{-}\frac{1}{2} & -\frac{5}{8}\vspace{1mm}\\
 -\frac{5}{8} & \hphantom{-}1
\end{bmatrix}, \qquad
\big(m_{ij}^{\mathrm{III}}\big)= \begin{bmatrix}
 \hphantom{-}\frac{1}{4} & -\frac{3}{8}\vspace{1mm}\\
 -\frac{3}{8} & \hphantom{-}1
\end{bmatrix}. \]
\end{prop}
\begin{proof}
We check that the roots $\lbrace \alpha_1, \alpha_{12}, \alpha_{1112}, \alpha_{11122}\rbrace$ satisfy condition~(\ref{cond7}B), while the root~$\alpha_2$ and $\alpha_{112}$ satisfy condition~(\ref{cond7}A).
We check that the reflections send one~$(m_{ij})$ to the other.
\end{proof}
\begin{cor}
By formula \eqref{centralcharge} for rank~$2$, we have that the central charge of the system is~$-\frac{874}{7}$.
\end{cor}

\begin{cor}Since $\big(m_{ij}^{\mathrm{II}}\big)$ is positive definite and has diagonal entries $|m_{ii}|\leq 1$, by Lem\-ma~{\rm \ref{pos+small}} the screening algebra is the Nichols algebra.
\end{cor}

\subsubsection*{Heckenberger row 13}
This case of Table~1 in \cite{Hecklist} is described by four diagrams:
\begin{gather*}
\begin{tikzpicture}
	\draw (0,0)--(1.6,0);
	\draw (-0.1,0) circle[radius=0.1cm] node[anchor=south]{$ \zeta^6$}
			(1.7,0) circle[radius=0.1cm] node[anchor=south]{$ -\zeta^{-4}$};
	\draw (0.8,0) node[anchor=south]{$ -\zeta^{-1}$};
	\draw (3.1,0)--(4.7,0);
	\draw (3.0,0) circle[radius=0.1cm] node[anchor=south]{$ \zeta^6$}
			(4.8,0) circle[radius=0.1cm] node[anchor=south]{$ \zeta^{-1}$};
	\draw (3.9,0) node[anchor=south]{$ \zeta$};
	\draw (6.2,0)--(7.8,0);
	\draw (6.1,0) circle[radius=0.1cm] node[anchor=south]{$ -\zeta^{-4}$}
			(7.9,0) circle[radius=0.1cm] node[anchor=south]{$ -1$};
	\draw (7,0) node[anchor=south]{$ \zeta^5$};
	\draw (9.3,0)--(10.9,0);
	\draw (9.2,0) circle[radius=0.1cm] node[anchor=south]{$ \zeta$}
			(11.0,0) circle[radius=0.1cm] node[anchor=south]{$ -1$};
	\draw (10.1,0) node[anchor=south]{$ \zeta^{-5}$};
\node at (0.8,-0.2) {$\mathrm{I}$};
\node at (3.9,-0.2) {$\mathrm{II}$};
\node at (7.0,-0.2) {$\mathrm{III}$};
\node at (10.1,-0.2) {$\mathrm{IV}$};
\end{tikzpicture}
\end{gather*}
where $\zeta \in \mathbb{G}_{24}$ and with respectively simple roots: \[ \mathrm{I}\colon \ \lbrace\alpha_1, \alpha_2\rbrace, \qquad \mathrm{II}\colon \ \lbrace-\alpha_1, \alpha_{1112}\rbrace, \qquad \mathrm{III}\colon \ \lbrace -\alpha_{2}, \alpha_{122}\rbrace, \qquad \mathrm{IV}\colon \ \lbrace\alpha_{12}, -\alpha_{122}\rbrace. \]
The associate Cartan matrices are
\begin{gather*}
\big(c_{ij}^{\mathrm{I}}\big)= \begin{bmatrix}
 \hphantom{-}2 & -3\\
 -2 & \hphantom{-}2
\end{bmatrix}, \qquad
\big(c_{ij}^{\mathrm{II}}\big)= \begin{bmatrix}
 \hphantom{-}2 & -3\\
 -1 & \hphantom{-}2
\end{bmatrix}, \\
\big(c_{ij}^{\mathrm{III}}\big)= \begin{bmatrix}
 \hphantom{-}2 & -2\\
 -1 & \hphantom{-}2
\end{bmatrix},\qquad
\big(c_{ij}^{\mathrm{IV}}\big)= \begin{bmatrix}
 \hphantom{-}2 & -5\\
 -1 & \hphantom{-}2
\end{bmatrix}.
\end{gather*}
The set of positive roots is $\lbrace \alpha_1, \alpha_2, \alpha_{12}, \alpha_{112}, \alpha_{122}, \alpha_{1112}, \alpha_{11122}, \alpha_{1111222}\rbrace$ where $\alpha_{12}$ and $\alpha_{1112}$ are the only $q$-Cartan roots while the others are only $q$-truncation.
\begin{prop}
The following $(m_{ij})$ matrices are realizing solutions of the given braiding and its reflections:
\begin{gather*}
\big(m_{ij}^{\mathrm{I}}\big)=  \begin{bmatrix}
 \frac{1}{2} & -\frac{13}{24}\vspace{1mm}\\
 -\frac{13}{24} & \frac{2}{3}
\end{bmatrix}, \qquad
\big(m_{ij}^{\mathrm{II}}\big)= \begin{bmatrix}
 \frac{1}{2} & -\frac{23}{24}\vspace{1mm}\\
 -\frac{23}{24} & \frac{23}{12}
\end{bmatrix}, \\
\big(m_{ij}^{\mathrm{III}}\big)=  \begin{bmatrix}
 \hphantom{-}1 & -\frac{19}{24}\vspace{1mm}\\
 -\frac{19}{24} & \hphantom{-}\frac{2}{3}
\end{bmatrix}, \qquad
\big(m_{ij}^{\mathrm{IV}}\big)= \begin{bmatrix}
 \hphantom{-}1 & -\frac{5}{24}\vspace{1mm}\\
 -\frac{5}{24} & \hphantom{-}\frac{1}{12}
\end{bmatrix}.
\end{gather*}
\end{prop}
\begin{proof}
We check that the roots $\alpha_{12}$ and $\alpha_{1112}$ satisfy condition (\ref{cond7}A), while the others satisfy condition (\ref{cond7}B).
We check that the reflections send one $(m_{ij})$ to the other.
\end{proof}
\begin{cor}
By formula \eqref{centralcharge} for rank~$2$, we have that the central charge of the system is~$-\frac{7826}{23}$.
\end{cor}

\begin{cor}Since $\big(m_{ij}^{\mathrm{I}}\big)$ is positive definite and has diagonal entries $|m_{ii}|\leq 1$, by Lem\-ma~{\rm \ref{pos+small}} the screening algebra is the Nichols algebra.
\end{cor}

\subsubsection*{Heckenberger row 14}
This case of Table~1 in \cite{Hecklist} is described by two diagrams:
\begin{gather*}
\begin{tikzpicture}
	\draw (0,0)--(1.6,0);
	\draw (-0.1,0) circle[radius=0.1cm] node[anchor=south]{$ \zeta$}
			(1.7,0) circle[radius=0.1cm] node[anchor=south]{$ {-1}$};
	\draw (0.8,0) node[anchor=south]{$ \zeta^{2}$};
	\draw (3.6,0)--(5.2,0);
	\draw (3.5,0) circle[radius=0.1cm] node[anchor=south]{$ -\zeta^{-2}$}
			(5.3,0) circle[radius=0.1cm] node[anchor=south]{$ -1$};
	\draw (4.4,0) node[anchor=south]{$ \zeta^{-2}$};
\node at (0.8,-0.2) {$\mathrm{I}$};
\node at (4.4,-0.2) {$\mathrm{II}$};
\end{tikzpicture}
\end{gather*}
where $\zeta \in \mathbb{G}_{5}$ and with respectively simple roots: \[ \mathrm{I}\colon \ \lbrace\alpha_1, \alpha_2\rbrace, \qquad \mathrm{II}\colon \ \lbrace \alpha_{12}, -\alpha_{2}\rbrace. \]
The associate Cartan matrices are
\[\big(c_{ij}^{\mathrm{I}}\big)= \begin{bmatrix}
 \hphantom{-}2 & -3\\
 -1 & \hphantom{-}2
\end{bmatrix}, \qquad
\big(c_{ij}^{\mathrm{II}}\big)= \begin{bmatrix}
 \hphantom{-}2 & -4\\
 -1 & \hphantom{-}2
\end{bmatrix}.
\]
The set of positive roots is $\lbrace \alpha_1, \alpha_2, \alpha_{12}, \alpha_{112}, \alpha_{1112}, \alpha_{1111222}, \alpha_{11122}, \alpha_{11111222}\rbrace$ where $\alpha_1$, $\alpha_{12}$, $\alpha_{112}$ and $\alpha_{11122}$ are only $q$-Cartan while the others are only $q$-truncation.
\begin{prop}
The following $m_{ij}$ matrices are realizing solutions of the given braiding and its reflections:
\[\big(m_{ij}^{\mathrm{I}}\big)= \begin{bmatrix}
 \hphantom{-}\frac{2}{5} & -\frac{3}{5}\vspace{1mm}\\
 -\frac{3}{5} & \hphantom{-}1
\end{bmatrix}, \qquad
\big(m_{ij}^{\mathrm{II}}\big)= \begin{bmatrix}
 \hphantom{-}\frac{1}{5} & -\frac{2}{5}\vspace{1mm}\\
 -\frac{2}{5} & \hphantom{-}1
\end{bmatrix}. \]
\end{prop}
\begin{proof}
We check that the roots $\alpha_1$, $\alpha_{12}$, $\alpha_{112}$ and $\alpha_{11122}$ satisfy condition (\ref{cond7}A), while the others satisfy condition (\ref{cond7}B). We check that the reflections send one $(m_{ij})$ to the other.
\end{proof}
\begin{cor}
By formula \eqref{centralcharge} for rank~$2$, we have that the central charge of the system is~$-364$.
\end{cor}

\begin{cor}
Since $\big(m_{ij}^{\mathrm{I}}\big)$ is positive definite and has diagonal entries $|m_{ii}|\leq 1$, by Lem\-ma~{\rm \ref{pos+small}} the screening algebra is the Nichols algebra.
\end{cor}

\subsubsection*{Heckenberger row 17}
This case of Table~1 in \cite{Hecklist} is described by two diagrams:
\[
\begin{tikzpicture}
	\draw (0,0)--(1.6,0);
	\draw (-0.1,0) circle[radius=0.1cm] node[anchor=south]{$ -\zeta$}
			(1.7,0) circle[radius=0.1cm] node[anchor=south]{$ {-1}$};
	\draw (0.8,0) node[anchor=south]{$ -\zeta^{-3}$};
	\draw (3.6,0)--(5.2,0);
	\draw (3.5,0) circle[radius=0.1cm] node[anchor=south]{$ -\zeta^{-2}$}
			(5.3,0) circle[radius=0.1cm] node[anchor=south]{$ -1$};
	\draw (4.4,0) node[anchor=south]{$ -\zeta^{3}$};
\node at (0.8,-0.2) {$\mathrm{I}$};
\node at (4.4,-0.2) {$\mathrm{II}$};
\end{tikzpicture}
\]
where $\zeta \in \mathbb{G}_{7}$ and with respectively simple roots: \[
\mathrm{I}\colon \ \lbrace\alpha_1, \alpha_2\rbrace, \qquad \mathrm{II}\colon \ \lbrace \alpha_{12}, -\alpha_{2}\rbrace. \]
The associate Cartan matrices are
\[\big(c_{ij}^{\mathrm{I}}\big)= \begin{bmatrix}
 \hphantom{-}2 & -3\\
 -1 & \hphantom{-}2
\end{bmatrix}, \qquad
\big(c_{ij}^{\mathrm{II}}\big)= \begin{bmatrix}
 \hphantom{-}2 & -5\\
 -1 & \hphantom{-}2
\end{bmatrix}.
\]
The set of positive roots is
\begin{gather*}
\lbrace  \alpha_1, \alpha_2, \alpha_{12}, \alpha_{112}, \alpha_{1112}, \alpha_{11122}, \alpha_{1111222}, \alpha_{111112222}, \alpha_{111111122222},\\
\qquad {}\alpha_{11111222}, \alpha_{1111111122222}, \alpha_{11111112222}\rbrace,
\end{gather*}
where $\lbrace \alpha_1, \alpha_{12}, \alpha_{112}, \alpha_{11122}, \alpha_{1111222}, \alpha_{11111222} \rbrace$ are only $q$-Cartan while the others are only $q$-truncation.
\begin{prop}
The following $(m_{ij})$ matrices are realizing solutions of the given braiding and its reflections:
\[\big(m_{ij}^{\mathrm{I}}\big)= \begin{bmatrix}
 \hphantom{-}\frac{6}{14} & -\frac{9}{14}\vspace{1mm}\\
 -\frac{9}{14} & \hphantom{-}1
\end{bmatrix}, \qquad
\big(m_{ij}^{\mathrm{II}}\big)= \begin{bmatrix}
 \hphantom{-}\frac{2}{14} & -\frac{5}{14}\vspace{1mm}\\
 -\frac{5}{14} & \hphantom{-}1
\end{bmatrix}. \]
\end{prop}
\begin{proof}
We check that the roots $\lbrace \alpha_1{,} \alpha_{12}{,} \alpha_{112}{,} \alpha_{11122}{,} \alpha_{1111222}{,} \alpha_{11111222} \rbrace$ satisfy condi\-tion~(\ref{cond7}A), while the others satisfy condition~(\ref{cond7}B). We check that the reflections send one $m_{ij}$-matrix to the other.
\end{proof}

\begin{cor}By formula \eqref{centralcharge} for rank~$2$, we have that the central charge of the system is~$-962$.
\end{cor}

\begin{cor}Since $\big(m_{ij}^{\mathrm{I}}\big)$ is positive definite and has diagonal entries $|m_{ii}|\leq 1$, by Lem\-ma~{\rm \ref{pos+small}} the screening algebra is the Nichols algebra.
\end{cor}

\subsection{Classification}\label{Classification}
In this section we are going to prove the following
\begin{teo}\label{classifthm}
For all finite-dimensional diagonal Nichols algebras of rank~$2$, the realiza\-tions~$(m_{ij})$ constructed in Sections~{\rm \ref{Cartan}},~{\rm \ref{SuperLie}} or~{\rm \ref{Strange}} are all realizations.
\end{teo}

In order to prove it, we are going to go through Table~1 in~\cite{Hecklist}, see which roots are $q$-truncation, $q$-Cartan and compute for every diagram the corresponding $(m_{ij})$ from some necessary conditions in the following Lemma. We will see that for every case this already fixes~$(m_{ij})$ uniquely, and of course we recover what we computed in the previous section.

To prove this result we will need the following two Propositions giving necessary conditions for a realization. In essence, it lists the conditions on a general $(m_{ij})$ of rank $2$, such that after one reflection condition~\eqref{cond7} holds.
\begin{prop}\label{prop_rank2reflection1}
We consider a diagram
\[
\begin{tikzpicture}
	\draw (0,0)--(1.6,0);
	\draw (-0.1,0) circle[radius=0.1cm] node[anchor=south]{$ q_{ii}$}
			(1.7,0) circle[radius=0.1cm] node[anchor=south]{$ q_{jj}$};
	\draw (0.8,0) node[anchor=south]{$ q_{ij} q_{ji}$} ;
\end{tikzpicture}
\]
where we assume that both $\{\alpha_i, \alpha_j\}$ are $q$-truncation, and apply a reflection $\r_{i}$ around the root~$\alpha_i$
\begin{align*}
\s_{i}\colon\quad &\alpha_i \longmapsto -\alpha_i, \\
&\alpha_j \longmapsto \alpha
\end{align*}
arriving to a new diagram with simple roots $\{-\alpha_i,   \alpha:= \alpha_j-c_{ij}\alpha_i\}$. We have
\begin{enumerate}\itemsep=0pt
	\item[$1)$] if $\alpha$ is $m$-truncation then \begin{equation}\label{TR-TR->TR}
	m_{ij} = \frac{c_{ij}}{1-c_{ij}} - \frac{1}{c_{ij}(1-a_{\beta,-\alpha_i})} + \frac{1}{c_{ij}(1-c_{ji})},
	\end{equation}
	\item[$2)$] if $\alpha$ is $m$-Cartan then \begin{equation}\label{TR-TR->CA}
	m_{ij} = \frac{c_{ij}}{1-c_{ij}} + \frac{\left( \frac{1}{1-c_{ji}} -\frac{c_{ij}}{(1-c_{ij})c_{\beta \alpha_i}}\right)}{(-\frac{1}{c_{\beta \alpha_i}} + c_{ij})}.
	\end{equation}
\end{enumerate}
\end{prop}

\begin{proof}
Since $\{\alpha_i, \alpha_j\}$ are only $q$-truncation, thus $m$-truncation, we have the relations
\[ m_{ii} = \frac{2}{1-c_{ij}},\qquad m_{jj} = \frac{2}{1-c_{ji}}.\]

1.~If $\beta$ is $m$-truncation then $m_{\beta \beta} = \frac{2}{1-c_{\beta, -\alpha_i}}$. But for definition of $\beta$ we have
\[ m_{\beta \beta} = m_{jj} - 2c_{ij} m_{ij} + c_{ij}^2 m_{ii}.\]
	Gathering all the information together we get
	\begin{align*}
	\frac{2}{1-c_{\beta, -\beta_i}} = \frac{2}{1-c_{ji}} - 2c_{ij} m_{ij} + c_{ij}^2 \frac{2}{1-c_{ij}}
	\end{align*}
and from this the final result.

2.~This case is completely analogous, with the only difference that $\beta$ is $m$-Cartan and thus $m_{\beta \beta} = \frac{2 m_{\beta, -i}}{c_{\beta \alpha_i}}$ we will then have
	\begin{gather*}
	m_{\beta \beta} = \frac{2 m_{\beta, -i}}{c_{\beta \alpha_i}} = - 2 \frac{m_{ij}}{c_{\beta \alpha_i}} + \frac{2 c_{ij} (\frac{2}{1-c_{ij}})}{c_{\beta \alpha_i}}, \\
	m_{\beta \beta} = \frac{2}{1-c_{ji}} - 2c_{ij} m_{ij} + c_{ij}^2 \frac{2}{1-c_{ij}}.
	\end{gather*}
The two equations together give the thesis.
\end{proof}

\begin{prop}\label{prop_rank2reflection2}
We consider a diagram
\[
\begin{tikzpicture}
	\draw (0,0)--(1.6,0);
	\draw (-0.1,0) circle[radius=0.1cm] node[anchor=south]{$ q_{ii}$}
			(1.7,0) circle[radius=0.1cm] node[anchor=south]{$ q_{jj}$};
	\draw (0.8,0) node[anchor=south]{$ q_{ij} q_{ji}$} ;
\end{tikzpicture}
\]
where we assume that $\{\alpha_i, \alpha_j\}$ are the first $q$-Cartan and the latter $q$-truncation. We apply a~reflection around the $q$-truncation root $\alpha_j$,
\begin{align*}
\s_{j}\colon \quad &\alpha_j \longmapsto -\alpha_j, \\
&\alpha_i \longmapsto \beta
\end{align*}
arriving to a new diagram $\r_j(q_{ij})$ associated to the roots: $\{\beta:= \alpha_i-c_{ji}\alpha_j,  -\alpha_j\}$.
In this diagram we have the necessary conditions
\begin{enumerate}\itemsep=0pt
	\item[$1)$] if $\beta$ is $m$-truncation then \begin{equation}\label{CA-TR->TR}
	m_{ij} = \frac{c_{ij}}{1-c_{ij}c_{ji}}\left( \frac{1}{1-c_{\beta,-\alpha_j}}- \frac{c_{ji}^2}{1-c_{ji}}\right),
	\end{equation}
	\item[$2)$] if $\beta$ is $m$-Cartan then \begin{equation}\label{CA-TR->CA}
	m_{ij} = \frac{a_{ij}c_{ji}}{1-c_{ji}}\cdot\frac{c_{ji}c_{\beta, -\alpha_j}-2}{c_{ji}c_{ij}c_{\beta,-j}-c_{\beta, -j}-c_{ij}}.
	\end{equation}
\end{enumerate}
\end{prop}

\subsubsection*{Heckenberger row 2}
We have $d=1$ and then $\ell_1 = \ell_2 = \frac{\ell}{gdc(\ell, 2)}$. Therefore $\ell \neq 2$ and since $c_{ij}=-1$ we have the following:
If $\ell > 4$ or $\ell=3$ then by classification Lemma~\ref{cartanclassification} we get a unique solution, presented in Section~\ref{Cartan} Heckenberger row~2.
If $\ell = 4$ then $q_{ii}=q^2= -1$ and the roots are both $q$-Cartan and $q$-truncation:
\begin{itemize}\itemsep=0pt
	\item If both are $m$-Cartan, we find a unique solution, by Lemma~\ref{cartanclassification} presented in Section~\ref{Cartan} Heckenberger row~2, in the limit case $q^2=-1$.
	\item If one of the two is $m$-truncation, we find a unique solution, presented in Section~\ref{SuperLie}, Heckenberger row~3, in the limit case $q^2 = -1$. This result is a consequence of Lemma~\ref{superlieclassification}.
	\item If both are only $m$-truncation we recognize the matrix
	$\left[\begin{smallmatrix}
 1 & -\frac{p'}{2}\\
 -\frac{p'}{2} & 1
\end{smallmatrix} \right]$, which is the other Weyl chamber in Example~\ref{sl(2|1)}.
\end{itemize}

\subsubsection*{Heckenberger row 3}
We have $d=1$ and then $\ell_1 = \ell_2 = \frac{\ell}{gdc(\ell, 2)}$. Therefore $\ell \neq 2$ and since $c_{12}=-1$ we have the following:
If $\ell > 4$ or $\ell=3$ then by classification Lemma~\ref{superlieclassification} we get a unique solution, presented in Section~\ref{SuperLie} case Heckenberger row~3.
If $\ell = 4$, $\alpha_1$ is both $q$-Cartan and $q$-truncation.
\begin{itemize}\itemsep=0pt
	\item If it is $m$-Cartan, we find again the unique solution presented in Section~\ref{SuperLie} Heckenberger row~3, in the limit case $q^2=-1$.
	This result is a consequence of Lemma~\ref{superlieclassification}.
	\item If it is $m$-truncation we recognize again the matrix $\left[\begin{smallmatrix}
 1 & -\frac{p'}{2}\\
 -\frac{p'}{2} & 1
\end{smallmatrix}\right] $ which is the other Weyl chamber in Example~\ref{sl(2|1)}.
\end{itemize}

\subsubsection*{Heckenberger row 4}
We have $d=d_2=2$ and then $\ell_1 = \frac{\ell}{gdc(\ell,2)}$, $\ell_2 = \frac{\ell}{gdc(\ell,4)}$. Moreover $\ell \neq 2,4$, because $q^2 \neq \pm 1$, and since $c_{12}=-2$, $c_{21}=-1$ we have the following:
If $\ell > 8$ or $\ell=5,7$ then by classification Lemma~\ref{cartanclassification} we get a unique solution, presented in Section~\ref{Cartan} Heckenberger row~4.
If $\ell = 8$ then the long root $\alpha_2$ is both $q$-Cartan and $q$-truncation, while $\alpha_1$ is only $q$-Cartan.
\begin{itemize}\itemsep=0pt
	\item If $\alpha_2$ is $m$-Cartan, we find again the unique solution presented in Section~\ref{Cartan}, Heckenberger row~4, by Lemma~\ref{cartanclassification}.
	\item If $\alpha_2$ is $m$-truncation, we find the unique solution presented in Section~\ref{SuperLie}, Heckenberger row~5, in the limit case $q^2=i$, by Lemma~\ref{superlieclassification}.
	
\end{itemize}
If $\ell = 3,6$ then the short root $\alpha_1$ is both $q$-Cartan and $q$-truncation, while $\alpha_2$ is only $q$-Cartan.
\begin{itemize}\itemsep=0pt
	\item If $\alpha_1$ is $m$-Cartan, we find a unique solution, presented in Section~\ref{Cartan} Heckenberger row 4, again thanks to Lemma \ref{cartanclassification}.
	\item If $\alpha_1$ is $m$-truncation, we find a family of solution, presented in Section~\ref{Strange}, Heckenberger row~6, up to rescaling. The uniqueness follows from Lemma \ref{superlieclassification}.
\end{itemize}
\subsubsection*{Heckenberger row 5} We have $d=1$ and then $\ell_1 = \frac{\ell}{gdc(\ell,2)}$. Moreover $\ell \neq 2,4$, because $q^2 \neq \pm 1$, and since $c_{12}=-2$ we have the following: If $\ell > 6 $ or $\ell = 5$ then by classification Lemma~\ref{superlieclassification} we get a unique solution, presented in Section~\ref{SuperLie} Heckenberger row~5.
If $\ell = 3, 6$ then the bosonic root $\alpha_1$ is both $q$-Cartan and $q$-truncation.
\begin{itemize}\itemsep=0pt
	\item If $\alpha_1$ is $m$-Cartan, we find again the unique solution presented in Section~\ref{SuperLie} Heckenberger row~5, by Lemma~\ref{superlieclassification}.
	\item If $\alpha_1$ is $m$-truncation, we recognize the matrix $\left[\begin{smallmatrix}
 \frac{2}{3} & -2m\\
 -2m & 1
\end{smallmatrix} \right]$
of Remark~\ref{Heck5trunc} which is a~solution only for $m=\frac{1}{3}$.
\end{itemize}

\subsubsection*{Heckenberger row 6}
We have $d=1$ and then $\ell_2 = \frac{\ell}{gdc(\ell,2)}$. Moreover $\ell \neq 2, 3, 6$, because $q^2 \neq 1$, $\zeta$, $\zeta^2$, with $\zeta \in \mathbb{G}_3$. Since $c_{12}=-1$ we have the following: If $\ell > 6$ or $\ell = 5$ then by classification Lemma~\ref{superlieclassification} we get a unique solution, presented in Section~\ref{Strange} Heckenberger row~6. If $\ell = 4$ then the root $\alpha_2$ is both $q$-Cartan and $q$-truncation.
\begin{itemize}\itemsep=0pt
	\item If $\alpha_2$ is $m$-Cartan, we find again the unique solution presented in Section~\ref{Strange} Heckenberger row~6, by Lemma~\ref{superlieclassification}.
	\item If $\alpha_2$ is $m$-truncation, we recognize the matrix $\left[\begin{smallmatrix}
 \frac{2}{3} & -m\\
 -m & 1
\end{smallmatrix}\right] $
of Remark~\ref{Heck6trunc} which is a solution only for $m=\frac{1}{2}$.
\end{itemize}

\subsubsection*{Heckenberger row 7}
We apply formula (\ref{TR-TR->TR}) to the reflection $\r_1$ and $\r_2$, since the simple roots $\alpha_1$ and $\alpha_2$ as well as the ones after reflections are only $q$-truncation and thus $m$-truncation. From the first reflection we obtain $m_{12}=-\frac{2}{3}$, while from the latter $m_{12}=-\frac{1}{2}$. Since these results do not match, it means that there is no possible formulation of the Nichols algebra braiding in terms of the matrix~$(m_{ij})$.
\begin{oss} We have $q$-truncation roots $\alpha_i$, $\alpha_j$, with $q_{ii}= \zeta$, $q_{jj}=\zeta^{-1}$, both third roots of unity and it is not possible to realize both of them with $m_{ii}=m_{jj}=\frac{2}{3}$. This is another way to see that this case is not realizable.
\end{oss}

\subsubsection*{Heckenberger row 8}
We apply formula (\ref{TR-TR->TR}) to the reflections $\r_1$ and $\r_2$, since the simple roots $\alpha_1$ and $\alpha_2$ as well as the ones after reflections are only $q$-truncation and thus $m$-truncation. From the first reflection we obtain $m_{12}=-\frac{3}{4}$, while from the latter $m_{12}=-\frac{7}{12}$. Since these results do not match, it means that there is no possible formulation of the Nichols algebra braiding in terms of the matrix~$(m_{ij})$.

\subsubsection*{Heckenberger row 9}
We apply formula (\ref{TR-TR->TR}) to the reflection $\r_1$ or $\r_2$, since the simple roots $\alpha_1$ and $\alpha_2$ as well as the ones after reflections are only $q$-truncation and thus $m$-truncation. The resulting $m_{12}$ shows that this is the matrix $(m_{ij})$ appearing in Section~\ref{Strange}. This is thus the only possible solution.

\subsubsection*{Heckenberger row 10}
We apply formula (\ref{CA-TR->TR}) to the reflection $\r_{2}$, since the simple root $\alpha_1$ is only $q$-Cartan and thus $m$-Cartan, while $\alpha_2$ as well as the ones after reflections are only $q$-truncation and thus $m$-truncation. The resulting~$m_{12}$ shows that this is the~$m_{ij}$ appearing in Section~\ref{Strange}. This is thus the only possible solution.

\subsubsection*{Heckenberger row 11}
We have $d=d_2=3$ and then $\ell_1 = \frac{\ell}{gdc(\ell,2)}$, $\ell_2=\frac{\ell}{gdc(\ell,6)}$. Moreover $\ell \neq 2, 3, 4, 6$ because $q^2 \neq \pm 1$, $q^2 \not\in \mathbb{G}_3$. Since $c_{12}=-3$ and $c_{21}=-1$ we have the following:
If $\ell > 12$ or $\ell =5, 7, 9, 10, 11$ then by classification Lemma~\ref{cartanclassification} we get a unique solution, presented in Section~\ref{Cartan} Heckenberger row~11.
If $\ell = 12$ then the root $\alpha_2$ is both $q$-Cartan and $q$-truncation, while the root $\alpha_1$ is only $q$-Cartan.
\begin{itemize}\itemsep=0pt
	\item If $\alpha_2$ is $m$-Cartan, we find again the unique solution presented in Section~\ref{Cartan} Heckenberger row~11, by Lemma~\ref{cartanclassification}.
	\item If $\alpha_2$ is $m$-truncation, we recognize the matrix
	$\left[\begin{smallmatrix}
 2m & -3m\\
 -3m & 1
\end{smallmatrix}\right]$,
which is a solution only for $m=\frac{1}{6}$.
\end{itemize}
If $\ell = 8$ then the root $\alpha_1$ is both $q$-Cartan and $q$-truncation, while the root $\alpha_2$ is only $q$-Cartan.
\begin{itemize}\itemsep=0pt
	\item If $\alpha_1$ is $m$-Cartan, we find again the unique solution presented in Section~\ref{Cartan} Heckenberger row~11, by Lemma~\ref{cartanclassification}.
	\item If $\alpha_1$ is $m$-truncation, we recognize the matrix $\left[\begin{smallmatrix}
 \frac{1}{2} & -3m\\
 -3m & 6m
\end{smallmatrix}\right]$
which is a solution only for $m=\frac{1}{4}$.
\end{itemize}

\subsubsection*{Heckenberger row 12}
We apply formula (\ref{CA-TR->TR}) to the reflections $\r_1$, since the simple roots $\alpha_1$ as well as the ones after reflections are only $q$-truncation and thus $m$-truncation, while $\alpha_2$ is only $q$-Cartan, and thus $m$-Cartan. The result is $m_{12}=-\frac{7}{8}$, which matches with the one of Section~\ref{Strange}.

\subsubsection*{Heckenberger row 13}
We apply formula (\ref{TR-TR->TR}) to the reflection $\r_1$ or $\r_2$, since the simple roots $\alpha_1$ and $\alpha_2$ as well as the ones after reflections are only $q$-truncation and thus $m$-truncation. The resulting~$m_{12}$ shows that this is the $m_{ij}$ appearing in Section~\ref{Strange}. This is thus the only possible solution.

\subsubsection*{Heckenberger row 14}
We apply formula (\ref{CA-TR->CA}) to the reflections $\r_2$, since the simple roots $\alpha_1$ as well as the ones after reflections are only $q$-Cartan and thus $m$-Cartan, while $\alpha_2$ is only $q$-truncation, and thus $m$-truncation. The result is $m_{12}=-\frac{3}{5}$, which matches with the one of Section~\ref{Strange}.

\subsubsection*{Heckenberger row 15}
We apply formula (\ref{TR-TR->TR}) to the reflections $\r_1$ and (\ref{TR-TR->CA}) to~$\r_2$ since the simple roots~$\alpha_1$ and~$\alpha_2$ as well as the ones after $\r_1$ are only $q$-truncation and thus $m$-truncation, while the ones after~$\r_2$ are only $q$-Cartan, and thus $m$-Cartan. From the first reflection we obtain $m_{12}=-\frac{4}{5}$, while from the latter $m_{12}=-\frac{11}{20}$. Since these results do not match, it means that there is no possible formulation of the Nichols algebra braiding in terms of the matrix~$(m_{ij})$.

\subsubsection*{Heckenberger row 16}
The root $\alpha_1$ is $q$-Cartan so we can not start with the system of simple roots $\alpha_1$, $\alpha_2$ if we want to compare the results of the reflections around them.
We then start with the simple roots $\alpha_{122}$ and $-\alpha_2$ which are only $q$-truncation and thus $m$-truncation.
After reflection $\r_{122}$ we obtain a only $q$-Cartan, and thus $m$-Cartan, simple root. While after reflection $\r_{2}$ we obtain a only $q$-truncation, and thus $m$-truncation, simple root. We then apply (\ref{TR-TR->CA}) to $\r_{122}$ and (\ref{TR-TR->TR}) to $\r_{2}$ obtaining two different results. Hence there is no possible formulation of the Nichols algebra braiding in terms of the matrix $(m_{ij})$.

\subsubsection*{Heckenberger row 17}
We apply formula (\ref{CA-TR->TR}) to the reflections $\r_2$, since the simple roots $\alpha_2$ as well as the ones after reflections are only $q$-truncation and thus $m$-truncation, while $\alpha_1$ is only $q$-Cartan, and thus $m$-Cartan. The result is $m_{12}=-\frac{5}{14}$, which matches with the one of Section~\ref{Strange}.

\section{Rank 3}\label{rank3}
We now rise the rank by one and construct all matrices $(m_{ij})$ which realize finite-dimensional diagonal Nichols algebras of rank~$3$, listed in Table 2 of~\cite{Hecklist}.

For Cartan type we will refer to the study of Section~\ref{Cartan}.
For super Lie type we will explicitly compute the realizing solutions.

For the other cases, we will see that the matrices $(m_{ij})$ matrices are completely fixed by the lower rank: this will imply uniqueness of the solution and make it not just a construction result but also a classification one.

In particular for these latter cases we will proceed as follows:
\begin{itemize}\itemsep=0pt
 	\item Given a $q$-diagram in rank $3$, we will consider it as two rank $2$ $q$-diagrams joined in the middle node. We will then associate to both sides the matrices $(m_{ij})$ realizing them, found in the rank $2$ study.
For these matrices $(m_{ij})$ to be compatible, some restriction on the parameter of which they depend will possibly appear.
 	\item We will then reflect the $q$-diagram on its $q$-truncation roots and proceed again as in the first point for the new diagram.
 	We reflect until we arrive not just to an already found $q$-diagram, but also when the realization $(m_{ij})$ is repeated (the matrix $(m_{ij})$ can be different also if associated to the same $q$-diagram).
 	\item We will then have to make sure that all the conditions found on the parameters are compatible and acceptable, in order for the rank $3$ matrices $(m_{ij})$ to be realizing solutions.
\end{itemize}

The $q$-diagrams and the associated realizing solutions are listed in Table~\ref{App3} of the appendix.

\subsubsection*{Heckenberger row 1}
This case belongs to the Cartan section. In particular it corresponds to the Lie algebras $A_3$ and it is described by the following $q$-diagram with corresponding realization $(m_{ij})$:
\[
\begin{tikzpicture}
	\draw (0,0)--(1.4,0) (1.6,0)--(3,0);
	\draw (-0.1,0) circle[radius=0.1cm] node[anchor=south]{$ q^2$} node[anchor=north]{$ 2m$}
			(1.5,0) circle[radius=0.1cm] node[anchor=south]{$ q^2 $}node[anchor=north]{$ 2m$}
			(3.1,0) circle[radius=0.1cm] node[anchor=south]{$ q^2 $}node[anchor=north]{$ 2m$};
	\draw (0.7,0) node[anchor=south]{$ q^{-2}$}node[anchor=north]{$ -2m$};
	\draw (2.3,0) node[anchor=south]{$ q^{-2}$} node[anchor=north]{$ -2m$};
\end{tikzpicture}
\]
\begin{oss}\label{rk3H1sp} When $q^2\in \mathbb{G}_2$ the roots are both $q$-Cartan and $q$-truncation and the $q$-diagram reads
\[
\begin{tikzpicture}
	\draw (0,0)--(1.4,0) (1.6,0)--(3,0);
	\draw (-0.1,0) circle[radius=0.1cm] node[anchor=south]{$ -1$} node[anchor=north]{$ $}
			(1.5,0) circle[radius=0.1cm] node[anchor=south]{$ -1$}node[anchor=north]{$ $}
			(3.1,0) circle[radius=0.1cm] node[anchor=south]{$ -1$}node[anchor=north]{$ $};
	\draw (0.7,0) node[anchor=south]{$ -1$}node[anchor=north]{$ $};
	\draw (2.3,0) node[anchor=south]{$ -1$} node[anchor=north]{$ $};
	\end{tikzpicture}
\]
\end{oss}
We have the following extra solutions:
\begin{itemize}\itemsep=0pt
	\item[--] When $\alpha_1$ is $m$-truncation and $\alpha_2$, $\alpha_3$ are $m$-Cartan we find
\[
\begin{tikzpicture}
	\draw (0,0)--(1.4,0) (1.6,0)--(3,0);
	\draw (-0.1,0) circle[radius=0.1cm] node[anchor=south]{$ -1$} node[anchor=north]{$ 1$}
			(1.5,0) circle[radius=0.1cm] node[anchor=south]{$ -1$}node[anchor=north]{$ 2m$}
			(3.1,0) circle[radius=0.1cm] node[anchor=south]{$ -1$}node[anchor=north]{$ 2m$};
	\draw (0.7,0) node[anchor=south]{$ -1$}node[anchor=north]{$ -2m$};
	\draw (2.3,0) node[anchor=south]{$ -1$} node[anchor=north]{$ -2m$};
	\end{tikzpicture}
\]
	which is one chamber of the Lie superalgebra $A(2,0)$ described in Heckenberger row~4.
	\item[--] When $\alpha_1$, $\alpha_2$ are $m$-truncation and $\alpha_3$ is $m$-Cartan we find
\[
\begin{tikzpicture}
	\draw (0,0)--(1.4,0) (1.6,0)--(3,0);
	\draw (-0.1,0) circle[radius=0.1cm] node[anchor=south]{$ -1$} node[anchor=north]{$ 1$}
			(1.5,0) circle[radius=0.1cm] node[anchor=south]{$ -1$}node[anchor=north]{$ 1$}
			(3.1,0) circle[radius=0.1cm] node[anchor=south]{$ -1$}node[anchor=north]{$ 2m$};
	\draw (0.7,0) node[anchor=south]{$ -1$}node[anchor=north]{$ m'$};
	\draw (2.3,0) node[anchor=south]{$ -1$} node[anchor=north]{$ -2m$};
	\end{tikzpicture}
\]
	which is a $m$-solution just for $m=\frac{1}{2}$ and $m'=-1$. But for these values of $m$, $m'$ the roots $\alpha_1$, $\alpha_2$ are also $m$-Cartan and thus this is not a new solution.
	\item[--] When $\alpha_2$ is $m$-truncation and $\alpha_1$, $\alpha_3$ are $m$-Cartan we find
\[
\begin{tikzpicture}
	\draw (0,0)--(1.4,0) (1.6,0)--(3,0);
	\draw (-0.1,0) circle[radius=0.1cm] node[anchor=south]{$ -1$} node[anchor=north]{$ 2m'$}
			(1.5,0) circle[radius=0.1cm] node[anchor=south]{$ -1$}node[anchor=north]{$ 1$}
			(3.1,0) circle[radius=0.1cm] node[anchor=south]{$ -1$}node[anchor=north]{$ 2m''$};
	\draw (0.7,0) node[anchor=south]{$ -1$}node[anchor=north]{$ -2m'$};
	\draw (2.3,0) node[anchor=south]{$ -1$} node[anchor=north]{$ -2m''$};
	\end{tikzpicture}
\]
	This is a solution either for $m'=\frac{1}{2}$ for which we end up again in the previous point, or for $m'=1-m''$, which gives us one chamber of the Lie superalgebra $A(1,1)$ described in Heckenberger row~8.
\item[--] When $\alpha_1$, $\alpha_3$ are $m$-truncation and $\alpha_2$ is $m$-Cartan we find
\[
\begin{tikzpicture}
	\draw (0,0)--(1.4,0) (1.6,0)--(3,0);
	\draw (-0.1,0) circle[radius=0.1cm] node[anchor=south]{$ -1$} node[anchor=north]{$ 1$}
			(1.5,0) circle[radius=0.1cm] node[anchor=south]{$ -1$}node[anchor=north]{$ 2m$}
			(3.1,0) circle[radius=0.1cm] node[anchor=south]{$ -1$}node[anchor=north]{$ 1$};
	\draw (0.7,0) node[anchor=south]{$ -1$}node[anchor=north]{$ -2m$};
	\draw (2.3,0) node[anchor=south]{$ -1$} node[anchor=north]{$ -2m$};
	\end{tikzpicture}
\]
	which is another chamber of the Lie superalgebra $A(1,1)$ described in Heckenberger row~8.
\item[--] When the roots are all $m$-truncation we find
\[
\begin{tikzpicture}
	\draw (0,0)--(1.4,0) (1.6,0)--(3,0);
	\draw (-0.1,0) circle[radius=0.1cm] node[anchor=south]{$ -1$} node[anchor=north]{$ 1$}
			(1.5,0) circle[radius=0.1cm] node[anchor=south]{$ -1$}node[anchor=north]{$ 1$}
			(3.1,0) circle[radius=0.1cm] node[anchor=south]{$ -1$}node[anchor=north]{$ 1$};
	\draw (0.7,0) node[anchor=south]{$ -1$}node[anchor=north]{$ m'$};
	\draw (2.3,0) node[anchor=south]{$ -1$} node[anchor=north]{$ m''$};
	\end{tikzpicture}
\]
	This is a solution either for $m'=-m''-2$ which is again a chamber of the Lie superalgebra $A(1,1)$, or for $m'=m''=-1$ for which the roots are also $m$-Cartan and thus does not give a new solution.
\end{itemize}
\subsubsection*{Heckenberger row 2}
This case belongs to the Cartan section. In particular it corresponds to the Lie algebras $B_3$ and it is described by the following $q$-diagram with corresponding realization $(m_{ij})$:
\[
\begin{tikzpicture}
	\draw (0,0)--(1.4,0) (1.6,0)--(3,0);
	\draw (-0.1,0) circle[radius=0.1cm] node[anchor=south]{$ q^4$} node[anchor=north]{$4r $}
			(1.5,0) circle[radius=0.1cm] node[anchor=south]{$ q^4 $}node[anchor=north]{$4m $}
			(3.1,0) circle[radius=0.1cm] node[anchor=south]{$ q^2 $}node[anchor=north]{$ 2m $};
	\draw (0.7,0) node[anchor=south]{$ q^{-4}$}node[anchor=north]{$ -4m$};
	\draw (2.3,0) node[anchor=south]{$ q^{-4}$} node[anchor=north]{$ -4m$};
\end{tikzpicture}
\]
\begin{oss}\label{rk3H2sp} When $q^2\in \mathbb{G}_4$ the roots $\alpha_1$, $\alpha_2$ are both $q$-Cartan and $q$-truncation and the $q$-diagram reads
\[
\begin{tikzpicture}
	\draw (0,0)--(1.4,0) (1.6,0)--(3,0);
	\draw (-0.1,0) circle[radius=0.1cm] node[anchor=south]{$ -1$} node[anchor=north]{$ $}
			(1.5,0) circle[radius=0.1cm] node[anchor=south]{$ -1 $}node[anchor=north]{$ $}
			(3.1,0) circle[radius=0.1cm] node[anchor=south]{$ i $}node[anchor=north]{$ $};
	\draw (0.7,0) node[anchor=south]{$ -1$}node[anchor=north]{$ $};
	\draw (2.3,0) node[anchor=south]{$ -1$} node[anchor=north]{$ $};
\end{tikzpicture}
\]
	For all the possible combinations of $m$-truncation and $m$-Cartan roots, no new solution is found. In some cases we find the Lie superalgebra $B(2,1)$ described in Heckenberger row~5.
	\end{oss}
	
\begin{oss} When $q^2\in \mathbb{G}_3$ the root $\alpha_3$ is both $q$-Cartan and $q$-truncation and the $q$-diagram reads
\[
\begin{tikzpicture}
	\draw (0,0)--(1.4,0) (1.6,0)--(3,0);
	\draw (-0.1,0) circle[radius=0.1cm] node[anchor=south]{$ \zeta^2$} node[anchor=north]{$ $}
			(1.5,0) circle[radius=0.1cm] node[anchor=south]{$ \zeta^2$}node[anchor=north]{$ $}
			(3.1,0) circle[radius=0.1cm] node[anchor=south]{$\zeta$}node[anchor=north]{$ $};
	\draw (0.7,0) node[anchor=south]{$ \zeta^{-2}$}node[anchor=north]{$ $};
	\draw (2.3,0) node[anchor=south]{$ \zeta^{-2}$} node[anchor=north]{$ $};
\end{tikzpicture}
\]
	with $\zeta\in \mathbb{G}_3$.
	The case when it is $m$-truncation is a solution only for $m=\frac{1}{3}$ for which the root is also $m$-Cartan and thus does not give a new solution.
	\end{oss}
\subsubsection*{Heckenberger row 3}
This case belongs to the Cartan section. In particular it corresponds to the Lie algebras $C_3$ and it is described by the following $q$-diagram with corresponding realization $(m_{ij})$:
\[
 \begin{tikzpicture}
	\draw (0,0)--(1.4,0) (1.6,0)--(3,0);
	\draw (-0.1,0) circle[radius=0.1cm] node[anchor=south]{$ q^2$} node[anchor=north]{$ 2m $}
			(1.5,0) circle[radius=0.1cm] node[anchor=south]{$ q^2 $}node[anchor=north]{$ 2m $}
			(3.1,0) circle[radius=0.1cm] node[anchor=south]{$ q^4 $}node[anchor=north]{$ 4m $};
	\draw (0.7,0) node[anchor=south]{$ q^{-2}$}node[anchor=north]{$ -2m$};
	\draw (2.3,0) node[anchor=south]{$ q^{-4}$} node[anchor=north]{$ -4m $};
\end{tikzpicture}
\]
\begin{oss}
If $q^2\in \mathbb{G}_4$, $ \alpha_3$ is both $q$-Cartan and $q$-truncation and the $q$-diagram reads
\[
\begin{tikzpicture}
	\draw (0,0)--(1.4,0) (1.6,0)--(3,0);
	\draw (-0.1,0) circle[radius=0.1cm] node[anchor=south]{$ i$} node[anchor=north]{$2m $}
			(1.5,0) circle[radius=0.1cm] node[anchor=south]{$ i $}node[anchor=north]{$2m $}
			(3.1,0) circle[radius=0.1cm] node[anchor=south]{$ -1 $}node[anchor=north]{$ 1 $};
	\draw (0.7,0) node[anchor=south]{$ -i$}node[anchor=north]{$ -2m$};
	\draw (2.3,0) node[anchor=south]{$ -1$} node[anchor=north]{$-4m $};
\end{tikzpicture}
\]
The case when it is $m$-truncation is a solution iff $m=\frac{1}{4}$ for which it is actually also $m$-Cartan. So this is not a new solution.
\end{oss}
\subsubsection*{Heckenberger row 4}
Row 4 of Table~2 in \cite{Hecklist} corresponds to the Lie superalgebra $A(2,0)$.

The simple roots in the standard chamber are $ \lbrace \alpha_1 =\alpha_\f, \alpha_2, \alpha_3\rbrace.$
We then have just a bosonic part $\mathfrak{g'}$. The inner products is given by
\[
(\alpha_i, \alpha_j)= \begin{bmatrix}
\hphantom{-}0 & -1 & \hphantom{-}0\\
 -1 & \hphantom{-}2 & -1\\
 \hphantom{-}0 & -1 & \hphantom{-}2
\end{bmatrix}
\]
and therefore
\[
\begin{tikzpicture}
	\draw (0,0)--(1.4,0) (1.6,0)--(3,0);
	\draw (-0.1,0) circle[radius=0.1cm] node[anchor=south]{$ -1$} node[anchor=north]{$ 1$}
			(1.5,0) circle[radius=0.1cm] node[anchor=south]{$ q^2$}node[anchor=north]{$ 2m$}
			(3.1,0) circle[radius=0.1cm] node[anchor=south]{$ q^2$}node[anchor=north]{$ 2m$};
	\draw (0.7,0) node[anchor=south]{$ q^{-2}$}node[anchor=north]{$ -2m$};
	\draw (2.3,0) node[anchor=south]{$ q^{-2}$} node[anchor=north]{$ -2m$};
	\end{tikzpicture}
\]

Reflecting around $\alpha_1$ we find the following
\[
\begin{tikzpicture}
	\draw (0,0)--(1.4,0) (1.6,0)--(3,0);
	\draw (-0.1,0) circle[radius=0.1cm] node[anchor=south]{$ -1$} node[anchor=north]{$ 1$}
			(1.5,0) circle[radius=0.1cm] node[anchor=south]{$ -1$}node[anchor=north]{$ 1$}
			(3.1,0) circle[radius=0.1cm] node[anchor=south]{$ q^2$}node[anchor=north]{$ 2m$};
	\draw (0.7,0) node[anchor=south]{$ q^2$}node[anchor=north]{$ -2+2m$};
	\draw (2.3,0) node[anchor=south]{$ q^{-2}$} node[anchor=north]{$ -2m$};
	\end{tikzpicture}
\]

Reflecting around the second root we find a symmetric result.
The roots satisfy condition~(\ref{cond7}) for all $m$ and therefore this $(m_{ij})$ is a realizing solution.

\subsubsection*{Heckenberger row 5}
Row 5 of Table~2 in \cite{Hecklist} corresponds to the Lie superalgebra $B(2,1)$.

The simple roots in the standard chamber are $ \lbrace \alpha_1 =\alpha_\f, \alpha_2, \alpha_3\rbrace.$
We then have just a bosonic part $\mathfrak{g'}$. The inner products is given by
\[
(\alpha_i, \alpha_j)= \begin{bmatrix}
 \hphantom{-}0 & -2 & \hphantom{-}0\\
 -2 & \hphantom{-}4 & -2\\
 \hphantom{-}0 & -2 & \hphantom{-}2
\end{bmatrix}
\]
and therefore
\[
\begin{tikzpicture}
	\draw (0,0)--(1.4,0) (1.6,0)--(3,0);
	\draw (-0.1,0) circle[radius=0.1cm] node[anchor=south]{$ -1$} node[anchor=north]{$ 1$}
			(1.5,0) circle[radius=0.1cm] node[anchor=south]{$ q^{4}$}node[anchor=north]{$ 4m$}
			(3.1,0) circle[radius=0.1cm] node[anchor=south]{$ q^2$}node[anchor=north]{$ 2m$};
	\draw (0.7,0) node[anchor=south]{$ q^{-4}$}node[anchor=north]{$ -4m$};
	\draw (2.3,0) node[anchor=south]{$ q^{-4}$} node[anchor=north]{$ -4m$};
	\end{tikzpicture}
\]

Reflecting around $\alpha_1$ we find the following
\[
\begin{tikzpicture}
	\draw (-0.3,0)--(1.4,0) (1.6,0)--(3,0);
	\draw (-0.4,0) circle[radius=0.1cm] node[anchor=south]{$ -1$} node[anchor=north]{$ 1$}
			(1.5,0) circle[radius=0.1cm] node[anchor=south]{$ -1$}node[anchor=north]{$ 1$}
			(3.1,0) circle[radius=0.1cm] node[anchor=south]{$ q^2$}node[anchor=north]{$ 2m$};
	\draw (0.5,0) node[anchor=south]{$ q^{4}$}node[anchor=north]{$ -2+4m$};
	\draw (2.3,0) node[anchor=south]{$ q^{-4}$} node[anchor=north]{$ -4m$};
	\end{tikzpicture}
\]

and after another reflection around the second root we find the following
\[
\begin{tikzpicture}
	\draw (-0.2,0)--(1.2,0) (1.4,0)--(3.3,0);
	\draw (-0.3,0) circle[radius=0.1cm] node[anchor=south]{$ q^{4}$} node[anchor=north]{$ 4m$}
			(1.3,0) circle[radius=0.1cm] node[anchor=south]{$ -1$}node[anchor=north]{$ 1$}
			(3.4,0) circle[radius=0.1cm] node[anchor=south]{$ -q^{-2}$};
	\draw (0.5,0) node[anchor=south]{$ q^{-4}$}node[anchor=north]{$-4m$};
	\draw (2.2,0) node[anchor=south]{$ q^{4}$} node[anchor=north]{$-2+4m$};
	\draw (3.6,0) node[anchor=north]{$1-2m$};
	\end{tikzpicture}
\]

The roots satisfy condition (\ref{cond7}) for all $m$ and therefore this $(m_{ij})$ is a realizing solution.
\begin{oss}
If $q^2\in \mathbb{G}_4$ then the root $\alpha_2$ is both $q$-Cartan and $q$-truncation. This case has been already studied in details in Heckenberger row 2 Remark \ref{rk3H2sp}.
\end{oss}
\begin{oss}
If $q^2\in \mathbb{G}_3$ then the root $\alpha_3$ is both $q$-Cartan and $q$-truncation.
When it is $m$-truncation we get
\[
\begin{tikzpicture}
	\draw (0,0)--(1.4,0) (1.6,0)--(3,0);
	\draw (-0.1,0) circle[radius=0.1cm] node[anchor=south]{$ -1$} node[anchor=north]{$ 1$}
			(1.5,0) circle[radius=0.1cm] node[anchor=south]{$ \zeta^{2}$}node[anchor=north]{$ 4m$}
			(3.1,0) circle[radius=0.1cm] node[anchor=south]{$ \zeta$}node[anchor=north]{$ \frac{2}{3}$};
	\draw (0.7,0) node[anchor=south]{$ \zeta^{-2}$}node[anchor=north]{$ -4m$};
	\draw (2.3,0) node[anchor=south]{$ \zeta^{-2}$} node[anchor=north]{$ -4m$};
	\end{tikzpicture}
\]

This is a solution iff $m=\frac{1}{3}$. But for this value of $m$, $\alpha_3$ is also $m$-Cartan and thus this is not a new solution.
\end{oss}

\subsubsection*{Heckenberger row 6}
Row 6 of Table~2 in \cite{Hecklist} corresponds to the Lie superalgebra $C(3)$.

The simple roots in the standard chamber are $ \lbrace \alpha_1 =\alpha_\f, \alpha_2, \alpha_3\rbrace.$
We then have just a bosonic part $\mathfrak{g'}$. The inner products is given by
\[
(\alpha_i, \alpha_j)= - \begin{bmatrix}
 \hphantom{-}0 & -1 & \hphantom{-}0\\
 -1 & \hphantom{-}2 & -2\\
 \hphantom{-}0 & -2 & \hphantom{-}4
\end{bmatrix}
\]
and therefore
\[
\begin{tikzpicture}
	\draw (0,0)--(1.4,0) (1.6,0)--(3,0);
	\draw (-0.1,0) circle[radius=0.1cm] node[anchor=south]{$ -1$} node[anchor=north]{$ 1$}
			(1.5,0) circle[radius=0.1cm] node[anchor=south]{$ q^2$}node[anchor=north]{$ 2m$}
			(3.1,0) circle[radius=0.1cm] node[anchor=south]{$ q^4$}node[anchor=north]{$ 4m$};
	\draw (0.7,0) node[anchor=south]{$ q^{-2}$}node[anchor=north]{$ -2m$};
	\draw (2.3,0) node[anchor=south]{$ q^{-4}$} node[anchor=north]{$ -4m$};
	\end{tikzpicture}
\]

Reflecting around $\alpha_1$ we find the following
\[
\begin{tikzpicture}
	\draw (-0.3,0)--(1.4,0) (1.6,0)--(3,0);
	\draw (-0.4,0) circle[radius=0.1cm] node[anchor=south]{$ -1$} node[anchor=north]{$ 1$}
			(1.5,0) circle[radius=0.1cm] node[anchor=south]{$ -1$}node[anchor=north]{$ 1$}
			(3.1,0) circle[radius=0.1cm] node[anchor=south]{$ q^4$}node[anchor=north]{$ 4m$};
	\draw (0.5,0) node[anchor=south]{$ q^2$}node[anchor=north]{$ -2+2m$};
	\draw (2.3,0) node[anchor=south]{$ q^{-4}$} node[anchor=north]{$ -4m$};
	\end{tikzpicture}
\]

Reflecting around $\alpha_{12}$ we find the following
\[
\begin{tikzpicture}
	\draw (5.5,0)--(4,-2.5)--(7,-2.5)--(5.5,0);
	\draw (5.5,0.1) circle[radius=0.1cm]
			(3.9,-2.5) circle[radius=0.1cm]
			(7.1,-2.5) circle[radius=0.1cm];
	\draw (5.5,-0.2) node[anchor=north]{$ -1$};
	\draw (4.3,-2) node[anchor=north]{$ q^2$};
	\draw (6.6,-2.1) node[anchor=north]{$ -1$};
	\draw (4.5,-1.5) node[anchor=west]{$ q^{-2}$};
	\draw (6.4,-1.5) node[anchor=east]{$ q^{4}$};
	\draw (5.6,-2.5) node[anchor=south]{$ q^{-2}$};
	\draw (5.5,0.1) node[anchor=south]{$ 1$}
			(3.9,-2.5) node[anchor=east]{$ 2m$}
			(7.1,-2.5) node[anchor=west]{$ 1$};
	\draw (4.6,-1.3) node[anchor=east]{$ -2m$};
	\draw (6.3,-1.3) node[anchor=west]{$ -2+4m$};
	\draw (5.5,-2.7) node[anchor=north]{$ -2m$};
\end{tikzpicture}
\]

The roots satisfy condition (\ref{cond7}) for all $m$ and therefore this $(m_{ij})$ is a realizing solution.
\begin{oss}
If $q^2\in\mathbb{G}_4$, $\alpha_3$ is both $q$-Cartan and $q$-truncation. When it is $m$-truncation we find
\[
\begin{tikzpicture}
	\draw (0,0)--(1.4,0) (1.6,0)--(3,0);
	\draw (-0.1,0) circle[radius=0.1cm] node[anchor=south]{$ -1$} node[anchor=north]{$ 1$}
			(1.5,0) circle[radius=0.1cm] node[anchor=south]{$ i$}node[anchor=north]{$ 2r$}
			(3.1,0) circle[radius=0.1cm] node[anchor=south]{$ -1$}node[anchor=north]{$ 1$};
	\draw (0.7,0) node[anchor=south]{$ -i$}node[anchor=north]{$ -2m$};
	\draw (2.3,0) node[anchor=south]{$ -1$} node[anchor=north]{$ -4m$};
	\end{tikzpicture}
\]
This is a solution iff $m=\frac{1}{4}$. But for this value of $m$, $\alpha_3$ is also $m$-Cartan and thus this is not a new solution.
\end{oss}
\begin{oss}
The simple roots in the standard chamber can be expressed according to \cite{Kac77} by
\[ \alpha_1 =\alpha_\f = \epsilon_1 - \delta_1, \qquad \alpha_2 = \delta_1 - \delta_2, \qquad \alpha_3 = 2\delta_2.\]
\end{oss}

\subsubsection*{Heckenberger row 7}
Row 7 of Table~2 in~\cite{Hecklist} corresponds to the Lie superalgebra $G(3)$ and it has been already explicitly treated as sporadic case of super Lie type in Section~\ref{G3}.

\subsubsection*{Heckenberger row 8}
Row 8 of Table~2 in~\cite{Hecklist} corresponds to the Lie superalgebra $A(1,1)$.

The simple roots in the standard chamber are $\lbrace\alpha_1,  \alpha_2 = \alpha_\f,  \alpha_3\rbrace$.
We then have two bosonic parts $\mathfrak{g'}$ and $\mathfrak{g''}$. The inner products is given by
\[
(\alpha_i, \alpha_j)= \begin{bmatrix}
\hphantom{-}2 & -1 & \hphantom{-}0\\
 -1 & \hphantom{-}0 & -1 \\
 \hphantom{-}0 & -1 & \hphantom{-}2
\end{bmatrix}
\]
and therefore
\[
\begin{tikzpicture}
	\draw (0,0)--(1.4,0) (1.6,0)--(3.1,0);
	\draw (-0.1,0) circle[radius=0.1cm] node[anchor=south]{$ q^2$} node[anchor=north]{$ 2m'$}
			(1.5,0) circle[radius=0.1cm] node[anchor=south]{$ -1$}node[anchor=north]{$ 1$}
			(3.2,0) circle[radius=0.1cm] node[anchor=south]{$ q^{-2}$}node[anchor=north]{$ 2m''$};
	\draw (0.7,0) node[anchor=south]{$ q^{-2}$}node[anchor=north]{$ -2m'$};
	\draw (2.25,0) node[anchor=south]{$ q^2$} node[anchor=north]{$ -2m''$};
	\end{tikzpicture}
\]

Reflecting around $\alpha_2$ we find the following
\[
\begin{tikzpicture}
	\draw (-0.3,0)--(1.4,0) (1.6,0)--(3.4,0);
	\draw (-0.4,0) circle[radius=0.1cm] node[anchor=south]{$ -1$} node[anchor=north]{$ 1$}
			(1.5,0) circle[radius=0.1cm] node[anchor=south]{$ -1$}node[anchor=north]{$ 1$}
			(3.5,0) circle[radius=0.1cm] node[anchor=south]{$ -1$}node[anchor=north]{$ 1$};
	\draw (0.5,0) node[anchor=south]{$ q^2$}node[anchor=north]{$ -2+2m'$};
	\draw (2.5,0) node[anchor=south]{$ q^{-2}$} node[anchor=north]{$ -2+2m''$};
	\end{tikzpicture}
\]

Other reflections give different matrices $(m_{ij})$ as shown in Table~\ref{App3}.
However, exception~(4) of Lemma~\ref{3 conditions II}, already appears. Indeed to the latter diagram is associated the following:
\[
m_{ij}^C= \begin{bmatrix}
1 & -1 + m' & -1+m' +m'' \\
-1 + m' & 1 & -1 + m'' \\
-1+m' +m'' & -1+m'' & 1	
\end{bmatrix}.
\]
We then have to ask $m_{13}^C=0$, i.e., $m' + m''=1$. In this case these matrices $(m_{ij})$ are realizing solution.

\begin{oss}
The simple roots in the standard chamber can be expressed according to~\cite{Kac77} by
\[\alpha_1 = \epsilon_1 - \epsilon_2,\qquad \alpha_2 = \alpha_\f = \epsilon_2 - \delta_1, \qquad \alpha_3 = \delta_1 - \delta_2,\]
with vectors $\epsilon_i$ generating $\mathfrak{g'}$ and $\delta_i$ generating $\mathfrak{g''}$.
\end{oss}

\subsubsection*{Heckenberger row 9--10--11}
Rows 9, 10, 11 of Table~2 in~\cite{Hecklist} correspond to the Lie superalgebra $D(2,1;\alpha)$ and it has been already explicitly treated as sporadic case of super Lie type in Section~\ref{G3}.

\subsubsection*{Heckenberger row 12}
The first diagram is a composition of the diagrams of rank $2$: $\# 2$ with $q= -\zeta^{-1}$ and $\# 6$ with $q= -\zeta^{-1}$, with $\zeta \in \mathbb{G}_3$.
\[
\begin{tikzpicture}
	\draw (0,0)--(1.4,0) (1.6,0)--(3,0);
	\draw (-0.1,0) circle[radius=0.1cm] node[anchor=south]{$ -\zeta^{-1}$} node[anchor=north]{$ 2m'$}
		 (1.5,0) circle[radius=0.1cm] node[anchor=south]{$ -\zeta^{-1} $} node[anchor=north]{$ 2m'$}
		 (3.1,0) circle[radius=0.1cm] node[anchor=south]{$ \zeta $};
	\draw (0.7,0) node[anchor=south]{$ -\zeta$} node[anchor=north]{$ -2m'$};
	\draw (2.3,0) node[anchor=south]{$ -\zeta$};
	\draw (1.5,-0.4) node[anchor=north]{$ 2m''$};
	\draw (2.4,-0.4) node[anchor=north]{$ -2m''$};
	\draw (3.1,-0.4) node[anchor=north]{$ \frac{2}{3} $};
\end{tikzpicture}
\]

For them to be joint in the middle circle we find $m'=m''=:m$.

The only $q$-truncation root is the third. Reflecting on it we find the same diagram and as matching condition $2m=\frac{8}{3}-2m$, i.e., $m=\frac{2}{3}$. But $q={\rm e}^{{\rm i} \pi m} \in \mathbb{G}_6$. So this case is not realizable.
\subsubsection*{Heckenberger row 13}
This case has two sub cases: $\zeta \in \mathbb{G}_3$ and $\zeta \in \mathbb{G}_6$ and diagram:
\[
\begin{tikzpicture}
	\draw (0,0)--(1.4,0) (1.6,0)--(3,0);
	\draw (-0.1,0) circle[radius=0.1cm] node[anchor=south]{$ \zeta$} node[anchor=north]{$ 2m'$}
		 (1.5,0) circle[radius=0.1cm] node[anchor=south]{$ \zeta $} node[anchor=north]{$ 2m'$}
		 (3.1,0) circle[radius=0.1cm] node[anchor=south]{$ -1 $};
	\draw (0.7,0) node[anchor=south]{$ \zeta^{-1}$} node[anchor=north]{$ -2m'$};
	\draw (2.3,0) node[anchor=south]{$ \zeta^{-2}$};
	\draw (1.5,-0.4) node[anchor=north]{$ 2m''$};
	\draw (2.4,-0.4) node[anchor=north]{$ -4m''$};
	\draw (3.1,-0.4) node[anchor=north]{$ 1 $};
\end{tikzpicture}
\]
\begin{enumerate}\itemsep=0pt
	\item Suppose $\zeta \in \mathbb{G}_3$.
	The first diagram is a composition of the diagrams of rank $2$: $\# 2$ with $q= \zeta$ and $\# 5$ with $q=\zeta$. For them to be joint in the middle circle we find $m'=m''=:m$.

The only $q$-truncation root is the third. Reflecting on it we find a diagram composition of $\# 4$ with $q= -\zeta^{-1}$ and $\# 5$ with $q= \zeta$. As matching condition we find $m=-2m+1$, i.e., $m=\frac{1}{3}$.
This case is thus realizable by the unique solution with parameter $m=\frac{1}{3}$.
	\item Suppose $\zeta \in \mathbb{G}_6$. We proceed analogously, but after reflecting around the third root we find a diagram which is composition of $\# 6$ with $q= \zeta$ and $\# 5$ with $q=\zeta$. The condition now is $m=\frac{1}{6}$ which is an acceptable condition.
	This case is thus realizable by the unique solution with parameter $m=\frac{1}{6}$.
\end{enumerate}

\subsubsection*{Heckenberger row 14}
This case is not realizable, since one of the diagrams contains diagram $\# 7$ of rank 2 which is on turn not realizable.

\subsubsection*{Heckenberger row 15}
The first diagram is a composition of the diagrams of rank $2$: $\# 3$ with $q= \zeta$ and $\# 5$ with $q=\zeta$, where $\zeta \in \mathbb{G}_3$.
\[
\begin{tikzpicture}
	\draw (0,0)--(1.4,0) (1.6,0)--(3,0);
	\draw (-0.1,0) circle[radius=0.1cm] node[anchor=south]{$ {-1}$} node[anchor=north]{$ 1$}
		 (1.5,0) circle[radius=0.1cm] node[anchor=south]{$ \zeta $} node[anchor=north]{$ 2m'$}
		 (3.1,0) circle[radius=0.1cm] node[anchor=south]{$ -1 $};
	\draw (0.7,0) node[anchor=south]{$ \zeta^{-1}$} node[anchor=north]{$ -2m'$};
	\draw (2.3,0) node[anchor=south]{$ \zeta$};
	\draw (1.5,-0.4) node[anchor=north]{$ 2m''$};
	\draw (2.4,-0.4) node[anchor=north]{$ -4m''$};
	\draw (3.1,-0.4) node[anchor=north]{$ 1 $};
\end{tikzpicture}
\]
For them to be joint in the middle circle we find $m'=m''=:m$.
After the reflections around $\r_{12}\circ \r_{1}$ we find the condition $m= \frac{1}{3}$ which is acceptable and gives a unique realizable solution.

\subsubsection*{Heckenberger row 16}
The first diagram is a composition of the diagrams of rank $2$: $\# 3$ with $q= \zeta$ and $\# 6$ with $q= -\zeta$, where $\zeta \in \mathbb{G}_3$.
\[
\begin{tikzpicture}
	\draw (0,0)--(1.4,0) (1.6,0)--(3,0);
	\draw (-0.1,0) circle[radius=0.1cm] node[anchor=south]{$ {-1}$} node[anchor=north]{$ 1$}
		 (1.5,0) circle[radius=0.1cm] node[anchor=south]{$ \zeta $} node[anchor=north]{$ 2m'$}
		 (3.1,0) circle[radius=0.1cm] node[anchor=south]{$ -\zeta $};
	\draw (0.7,0) node[anchor=south]{$ \zeta^{-1}$} node[anchor=north]{$ -2m'$};
	\draw (2.3,0) node[anchor=south]{$ -\zeta^{-1}$};
	\draw (1.5,-0.4) node[anchor=north]{$ \frac{2}{3}$};
	\draw (2.4,-0.4) node[anchor=north]{$ -2m''$};
	\draw (3.3,-0.4) node[anchor=north]{$ 2m'' $};
\end{tikzpicture}
\]
 For them to be joint in the middle circle we find $m'=\frac{1}{3}$.
 After reflecting on the second root we find the condition $m'' = \frac{5}{6}$.
This case is thus realizable by the unique solution with parameters $m'=\frac{1}{3}$ and $m''=\frac{5}{6}$.

\subsubsection*{Heckenberger row 17}
This case is not realizable, since one of the diagrams contains diagram $\# 7$ of rank 2 which is on turn not realizable.

\subsubsection*{Heckenberger row 18}
The first diagram is a composition of the diagrams of rank $2$: $\# 2$ with $q= \zeta$ and $\# 6$ with $q= \zeta$, with $\zeta \in \mathbb{G}_9$:
\[
\begin{tikzpicture}
	\draw (0,0)--(1.4,0) (1.6,0)--(3,0);
	\draw (-0.1,0) circle[radius=0.1cm] node[anchor=south]{$ \zeta$} node[anchor=north]{$ 2m' $}
		 (1.5,0) circle[radius=0.1cm] node[anchor=south]{$ \zeta $} node[anchor=north]{$ 2m'$}
		 (3.1,0) circle[radius=0.1cm] node[anchor=south]{$ \zeta^{-3} $};
	\draw (0.7,0) node[anchor=south]{$ \zeta^{-1} $} node[anchor=north]{$ -2m'$};
	\draw (2.3,0) node[anchor=south]{$ \zeta^{-1} $};
	\draw (1.5,-0.4) node[anchor=north]{$ 2m''$};
	\draw (2.4,-0.4) node[anchor=north]{$ -2m''$};
	\draw (3.1,-0.4) node[anchor=north]{$ \frac{2}{3} $};
\end{tikzpicture}
\]
For them to be joint in the middle circle we find $m'=m''=:m$.
The only $q$-truncation root is the third. Reflecting on it we find the same diagram and as matching condition $m=-\frac{8}{3}+2r$, i.e., $m=\frac{8}{9}$.
This case is thus realizable by the unique solution with parameter $m=\frac{8}{9}$.

\section[Rank $\geq 4$]{Rank $\boldsymbol{\geq 4}$}\label{section10}

In rank $\geq 4$ we do not list all diagram, but we give an effective way to determine all possible realizations from the list of rank~$3$ realizations:

Determining all possible realizations is a simple matter of covering a $q$-diagram with smaller $q$-diagrams, looking up their realizations (which are typically unique or depend on one parameter) and choosing the parameters such that $(m_{ij})$ agrees on the overlap of the subdiagrams. Typically the result is a unique possible realization.

Verifying on the other hand that a possible realization is indeed a realization can in rank $\geq 4$ be in principle done as follows: The entry $m_{ij}^{\r_k(C)}$ after a reflection on $\alpha_k$ from $m_{ij}^C$ is entirely determine in the rank~$3$ Nichols subalgebra and root system generated by $\alpha_i$, $\alpha_j$, $\alpha_k$. Hence in principle we go through all simple roots $\alpha_k$ in all chamber, which are not $m$-Cartan (otherwise the diagram and its realization remains unchanged), and compare our choices of $\big(m_{ij}^{C}\big)$, $\big(m_{ij}^{\r_k(C)}\big)$ as follows:
\begin{itemize}\itemsep=0pt
	\item If $\alpha_k$ is connected to $\alpha_i$ and $\alpha_j$ determine its reflection of this rank $3$ subdiagram from the list and verify that it coincides with the choice of the realization $\big(m_{ij}^{\r_k(C)}\big)$. If $\alpha_k$ is a branch point, this has to be verified for all combinations of $\alpha_i$ and $\alpha_j$. If $\alpha_k$ is only connected to one vertex~$\alpha_i$, this has to be verified only for the rank $2$ subdiagram.
	\item For each $\alpha_i$, $\alpha_j$ not connected to~$\alpha_k$, verify that $m_{ij}^{C}$ coincides with the choice of realization~$m_{ij}^{\r_k(C)}$.
\end{itemize}
In practical examples, we have chosen large subdiagrams that correspond in many Weyl chambers, and we have tried to mostly have an overlap of rank $2$ between these subdiagrams, so that most verifications above are true by construction.

\begin{ese}
	We consider rank $4$ row 18 with $q^3=1$, $q\neq 1$:
\[
	\begin{tikzpicture}
	\draw (0,0)--(1.4,0) (1.6,0)--(3,0) (3.2,0)--(4.6,0);
	\draw (-0.1,0) circle[radius=0.1cm]
	node[anchor=south]{$q^{-1}$};
	\draw (1.5,0) circle[radius=0.1cm]
	node[anchor=south]{$q^{-1}$};
	\draw (3.1,0) circle[radius=0.1cm]
	node[anchor=south]{$q$};
	\draw (4.7,0) circle[radius=0.1cm]
	node[anchor=south]{$-1$};
	\draw (0.7,0)
	node[anchor=south]{$q$};
	\draw (2.3,0)
	node[anchor=south]{$q$};
	\draw (3.9,0)
	node[anchor=south]{$q^{-1}$};
	\end{tikzpicture}
\]
	
	In the first diagram we consider the subdiagram on the nodes $1$, $2$, $3$ of Cartan type $B_3$ (we slightly rewrite the $q$'s to make this visible) and the subdiagram on $2$, $3$, $4$ of super Lie type~$C(3)$. Each has a unique family of realizations depending on a parameter $m_1$ resp.~$m_2$. The overlap between the diagrams (Cartan type $B_2$) has decorations $2m_1$, $-2m_1$, $m_1$ resp.~$2m_2$, $-2m_2$, $m_2$. Hence the only possible realization is for $m_1=m_2=:m$
\[
	\begin{tikzpicture}
	\draw (0,0)--(1.4,0) (1.6,0)--(3,0) (3.2,0)--(4.6,0);
	\draw (-0.1,0) circle[radius=0.1cm]
	node[anchor=south]{$q^2$}
	node[anchor=north]{$2m$};
	\draw (1.5,0) circle[radius=0.1cm]
	node[anchor=south]{$q^2$}
	node[anchor=north]{$2m$};
	\draw (3.1,0) circle[radius=0.1cm]
	node[anchor=south]{$q$}
	node[anchor=north]{$m$};
	\draw (4.7,0) circle[radius=0.1cm]
	node[anchor=south]{$-1$}
	node[anchor=north]{$1$};
	\draw (0.7,0)
	node[anchor=south]{$q^{-2}$}
	node[anchor=north]{$-2m$};
	\draw (2.3,0)
	node[anchor=south]{$q^{-2}$}
	node[anchor=north]{$-2m$};
	\draw (3.9,0)
	node[anchor=south]{$q^{-1}$}
	node[anchor=north]{$-m$};
	\end{tikzpicture}
\]
The only relevant reflection is $\r_4$, and the entire neighborhood is contained in $C(3)$. So we look up the reflection of the realization of $C(3)$ and leave the remaining realization unchanged:
\[
	\begin{tikzpicture}
	\draw (0,0)--(1.4,0) (1.6,0)--(3,0) (3.2,0)--(4.6,0);
	\draw (-0.1,0) circle[radius=0.1cm]
	node[anchor=south]{$q^2$}
	node[anchor=north]{$2m$};
	\draw (1.5,0) circle[radius=0.1cm]
	node[anchor=south]{$q^2$}
	node[anchor=north]{$2m$};
	\draw (3.1,0) circle[radius=0.1cm]
	node[anchor=south]{$-1$}
	node[anchor=north]{$1$};
	\draw (4.7,0) circle[radius=0.1cm]
	node[anchor=south]{$-1$}
	node[anchor=north]{$1$};
	\draw (0.7,0)
	node[anchor=south]{$q^{-2}$}
	node[anchor=north]{$-2m$};
	\draw (2.3,0)
	node[anchor=south]{$q^{-2}$}
	node[anchor=north]{$-2m$};
	\draw (3.9,0)
	node[anchor=south]{$q$}
	node[anchor=north]{$-2+m$};
	\end{tikzpicture}
\]
We now have to \emph{verify} that the subdiagram on $1$, $2$, $3$ with this decoration turns into a listed realization. Indeed this is the realization of $A(2|0)$ at $q^2$ with parameter $m_3=2m$.

The only relevant new reflection is $\r_3$ and the neighborhood is contained in $C(3)$. So we look up the reflection there:
\[
	\begin{tikzpicture}
	\draw (0,0)--(1.4,0) (1.6,0)--(3,0) (3.2,0)--(4.6,0);
	\draw (-0.1,0) circle[radius=0.1cm]
	node[anchor=south]{$q^2$}
	node[anchor=north]{$2m$};
	\draw (1.5,0) circle[radius=0.1cm]
	node[anchor=south]{$-1$}
	node[anchor=north]{$1$};
	\draw (3.1,0) circle[radius=0.1cm]
	node[anchor=south]{$-1$}
	node[anchor=north]{$1$};
	\draw (4.7,0) circle[radius=0.1cm]
	node[anchor=south]{$q$}
	node[anchor=north]{$m$};
	\draw (0.7,0)
	node[anchor=south]{$q^{-2}$}
	node[anchor=north]{$-2m$};
	\draw (2.3,0)
	node[anchor=south]{$q^2$}
	node[anchor=north]{$2m-2$};
	\draw (3.9,0)
	node[anchor=south]{$q^{-1}$}
	node[anchor=north]{$m$};
	\draw (1.5+0.1,-0.1) arc (-90-60:-90+60:1.74);
	\draw (3.1,-1)
	node[anchor=south]{$q^{-1}$}
	node[anchor=north]{$-m$};
	\end{tikzpicture}
\]
Since $3$ is also in the subdiagram on $1$, $2$, $3$, this is automatically still the realization of $A(2|0)$.

The only relevant new reflection is $\r_2$ at the branch point.We introduce a new subdiagram on~$1$, $2$, $4$. This is the $q$-diagram of $A(1|1)$ at $q$ with realizations parametrized by~$m_4$. But \emph{matching the decorations} on the right side of $2$ requires $m_4=m$ and matching it on the left side requires $2-m_4=2m$. This is only possible for $m=\frac{2}{3}$.

The only relevant new reflection is $\r_2$, a branch point, and part of the neighborhood is contained in $A(2|0)$ and parts in $A(1|1)$. So we look up the reflection there:
\[
	\begin{tikzpicture}
	\draw (0,0)--(1.4,0) (1.6,0)--(3,0); 
	\draw (-0.1,0) circle[radius=0.1cm]
	node[anchor=south]{$-1$}
	node[anchor=north]{$1$};
	\draw (1.5,0) circle[radius=0.1cm]
	node[anchor=south]{$-1$}
	node[anchor=north]{$1$};
	\draw (3.1,0) circle[radius=0.1cm]
	node[anchor=south]{$q^2$}
	node[anchor=north]{$2m$};
	\draw (4.7,0) circle[radius=0.1cm]
	node[anchor=south]{$-1$}
	node[anchor=north]{$1$};
	\draw (0.7,0)
	node[anchor=south]{$q^{2}$}
	node[anchor=north]{$-2+2m$};
	\draw (2.3,0)
	node[anchor=south]{$q^{-2}$}
	node[anchor=north]{$-2m$};
	\draw (1.5+0.1,-0.1) arc (-90-60:-90+60:1.74);
	\draw (3.1,-1)
	node[anchor=south]{$q$}
	node[anchor=north]{$-2+m$};
	\end{tikzpicture}
\]
Now there are two relevant new reflections:
$\r_1$ has its neighborhood in $A(2|0)$ and $A(1|1)$ and accordingly gives
\[
	\begin{tikzpicture}
	\draw (0,0)--(1.4,0) (1.6,0)--(3,0); 
	\draw (-0.1,0) circle[radius=0.1cm]
	node[anchor=south]{$-1$}
	node[anchor=north]{$1$};
	\draw (1.5,0) circle[radius=0.1cm]
	node[anchor=south]{$q^{2}$}
	node[anchor=north]{$2m$};
	\draw (3.1,0) circle[radius=0.1cm]
	node[anchor=south]{$q^2$}
	node[anchor=north]{$2m$};
	\draw (4.7,0) circle[radius=0.1cm]
	node[anchor=south]{$-1$}
	node[anchor=north]{$1$};
	\draw (0.7,0)
	node[anchor=south]{$q^{-2}$}
	node[anchor=north]{$-2m$};
	\draw (2.3,0)
	node[anchor=south]{$q^{-2}$}
	node[anchor=north]{$-2m$};
	\draw (1.5+0.1,-0.1) arc (-90-60:-90+60:1.74);
	\draw (3.1,-1)
	node[anchor=south]{$q$}
	node[anchor=north]{$-2+m$};
	\end{tikzpicture}
\]
We have to verify that the reflections in both subdiagrams agree (which is clear because they are reflections in a common~$A(1|0)$).

The second reflection $\r_2$ has its neighborhood in $A(1|1)$:
\[
	\begin{tikzpicture}
	\draw (0,0)--(1.4,0) (1.6,0)--(3,0); 
	\draw (-0.1,0) circle[radius=0.1cm]
	node[anchor=south]{$-1$}
	node[anchor=north]{$1$};
	\draw (1.5,0) circle[radius=0.1cm]
	node[anchor=south]{$q$}
	node[anchor=north]{$m$};
	\draw (3.1,0) circle[radius=0.1cm]
	node[anchor=south]{$q^2$}
	node[anchor=north]{$2m$};
	\draw (4.7,0) circle[radius=0.1cm]
	node[anchor=south]{$-1$}
	node[anchor=north]{$1$};
	\draw (0.7,0)
	node[anchor=south]{$q^{-1}$}
	node[anchor=north]{$-m$};
	\draw (2.3,0)
	node[anchor=south]{$q^{-2}$}
	node[anchor=north]{$-2m$};
	\draw (1.5+0.1,-0.1) arc (-90-60:-90+60:1.74);
	\draw (3.1,-1)
	node[anchor=south]{$q^{-1}$}
	node[anchor=north]{$-m$};
	\end{tikzpicture}
\]
We have to verify that the new diagram on $1$, $2$, $3$ appears in the list of realizations, namely super Lie type $C(3)$ at $q$ with $m_5=5$, which again requires on the edge $12$ the identity $-2+2m=-m$ for $m=\frac{2}{3}$.

Applying both reflections in either order gives the following (again this is not problematic because the neighborhood is in both diagrams and the reflection on $A(1|0)$ gives the same result:
\[
	\begin{tikzpicture}
	\draw (0,0)--(1.4,0) (1.6,0)--(3,0); 
	\draw (-0.1,0) circle[radius=0.1cm]
	node[anchor=south]{$-1$}
	node[anchor=north]{$1$};
	\draw (1.5,0) circle[radius=0.1cm]
	node[anchor=south]{$-1$}
	node[anchor=north]{$1$};
	\draw (3.1,0) circle[radius=0.1cm]
	node[anchor=south]{$q^2$}
	node[anchor=north]{$2m$};
	\draw (4.7,0) circle[radius=0.1cm]
	node[anchor=south]{$-1$}
	node[anchor=north]{$1$};
	\draw (0.7,0)
	node[anchor=south]{$q$}
	node[anchor=north]{$-2+m$};
	\draw (2.3,0)
	node[anchor=south]{$q^{-2}$}
	node[anchor=north]{$-2m$};
	\draw (1.5+0.1,-0.1) arc (-90-60:-90+60:1.74);
	\draw (3.1,-1)
	node[anchor=south]{$q^{-1}$}
	node[anchor=north]{$-m$};
	\end{tikzpicture}
\]
To summarize: The following is the unique realization
\[
	\begin{tikzpicture}
	\draw (0,0)--(1.4,0) (1.6,0)--(3,0) (3.2,0)--(4.6,0);
	\draw (-0.1,0) circle[radius=0.1cm]
	node[anchor=south]{$q^2$}
	node[anchor=north]{$\frac{4}{3}$};
	\draw (1.5,0) circle[radius=0.1cm]
	node[anchor=south]{$q^2$}
	node[anchor=north]{$\frac{4}{3}$};
	\draw (3.1,0) circle[radius=0.1cm]
	node[anchor=south]{$q$}
	node[anchor=north]{$\frac{2}{3}$};
	\draw (4.7,0) circle[radius=0.1cm]
	node[anchor=south]{$-1$}
	node[anchor=north]{$1$};
	\draw (0.7,0)
	node[anchor=south]{$q^{-2}$}
	node[anchor=north]{$-\frac{4}{3}$};
	\draw (2.3,0)
	node[anchor=south]{$q^{-2}$}
	node[anchor=north]{$-\frac{4}{3}$};
	\draw (3.9,0)
	node[anchor=south]{$q^{-1}$}
	node[anchor=north]{$-\frac{2}{3}$};
	\end{tikzpicture}
\]
\end{ese}

\section{Tables: realizing lattices of Nichols algebras in rank 2 and 3}

We now list from \cite{Hecklist} all finite-dimensional diagonal Nichols algebras in rank~$2$ and $3$ in terms of their $q$-diagrams, and below each of them we display the corresponding realizing lattice in terms of $m$-diagrams, such that $q_{ij}={\rm e}^{{\rm i} \pi m_{ij}}$ and the reflection compatibility \eqref{cond7} holds.

The numbers of the rows are Heckenberger's numbering, but sometimes we subdivide the cases, e.g., row $2$ into $2'$ for $q=-1$ and $2''$ for $q\neq \pm 1$, if they have different number of realizations.
Note that we display the Nichols algebras associated to quantum (super-)groups as in Heckenberger list (with $q_{ii}=1$ for a short root) in contrast to the notation used for quantum (super-)groups (with $q_{ii}=q^{(\alpha_i,\alpha_i)}=q^2$ for a short root), due to the usual normalization of the Killing form, which we used in Sections~\ref{Cartan} and~\ref{SuperLie}.

\begin{table}[h!]\centering\footnotesize \renewcommand{\arraystretch}{1.06}
 \caption{Realization of finite-dimensional diagonal Nichols algebras of rank~2.}\label{App2}
 \vspace{1mm}

\begin{tabular}{p{5mm}|p{63mm}|p{18mm}|p{57mm}}
\hline
row & braiding & conditions & comments \\
\hline
\raisebox{-2mm}[0pt][0pt]{$2'$} & \begin{tikzpicture}
	\draw (0,0)--(1.6,0);
	\draw (-0.1,0) circle[radius=0.1cm] node[anchor=south]{$ -1$}
			(1.7,0) circle[radius=0.1cm] node[anchor=south]{$ -1$};
	\draw (0.9,0) node[anchor=south]{$ -1$} node[anchor=north]{$ -m$};
	\draw (-0.1,-0.05) node[anchor=north]{$ m$};
	\draw (1.7,-0.05) node[anchor=north]{$ m$};
\end{tikzpicture} & & \raisebox{3mm}[0pt][0pt]{\text{one solution according to $A_2$ (see $2''$)}} \\
 \cline{2-2}
 & \begin{tikzpicture}
	\draw (0,0)--(1.6,0);
	\draw (-0.1,0) circle[radius=0.1cm] node[anchor=south]{$ -1$}
			(1.7,0) circle[radius=0.1cm] node[anchor=south]{$-1$};
	\draw (0.9,0) node[anchor=south]{$ {-1}$} node[anchor=north]{$ -m$};
	\draw (-0.1,-0.05) node[anchor=north]{$ m$};
	\draw (1.7,-0.05) node[anchor=north]{$ 1$};
	\draw (2.8,0)--(4.4,0);
	\draw (2.7,0) circle[radius=0.1cm] node[anchor=south]{$ -1$}
			(4.5,0) circle[radius=0.1cm] node[anchor=south]{$-1$};
	\draw (3.7,0) node[anchor=south]{$ -1$} node[anchor=north]{$ -2+m$};
		\draw (2.7,-0.05) node[anchor=north]{$ 1$};
	\draw (4.5,-0.05) node[anchor=north]{$ 1$};
\end{tikzpicture} & & \raisebox{3mm}[0pt][0pt]{\text{one solution according to $A(1,0)$ (see $3$)}}\\
\hline
\raisebox{3mm}[0pt][0pt]{$2''$} & \begin{tikzpicture}
	\draw (0,0)--(1.6,0);
	\draw (-0.1,0) circle[radius=0.1cm] node[anchor=south]{$ q$}
			(1.7,0) circle[radius=0.1cm] node[anchor=south]{$ q$};
	\draw (0.9,0) node[anchor=south]{$ q^{-1}$} node[anchor=north]{$ -m$};
	\draw (-0.1,-0.05) node[anchor=north]{$ m$};
	\draw (1.7,-0.05) node[anchor=north]{$ m$};
\end{tikzpicture} & \raisebox{3mm}[0pt][0pt]{$q\neq \pm 1$} & \raisebox{3mm}[0pt][0pt]{\text{Cartan, $A_2$}} \\
\hline
\end{tabular}
\end{table}

\begin{table}[t!]\centering\footnotesize \renewcommand{\arraystretch}{1.0}
\begin{tabular}{p{5mm}|p{66mm}|p{18mm}|p{54mm}}
\hline
\raisebox{5mm}[0pt][0pt]{8}& \begin{tikzpicture}
	\draw (0,0)--(1.2,0);
	\draw (-0.1,0) circle[radius=0.1cm] node[anchor=south]{$ -\zeta^{-2}$}
			(1.3,0) circle[radius=0.1cm] node[anchor=south]{$ -\zeta^{2}$};
	\draw (0.7,0) node[anchor=south]{$ -\zeta^{3}$};
	\draw (2.2,0)--(3.4,0);
	\draw (2.1,0) circle[radius=0.1cm] node[anchor=south]{$ -\zeta^{-2} $}
			(3.5,0) circle[radius=0.1cm] node[anchor=south]{$ -1$};
	\draw (2.9,0) node[anchor=south]{$ \zeta^{-1}$};
	\draw (4.4,0)--(5.6,0);
	\draw (4.3,0) circle[radius=0.1cm] node[anchor=south]{$ -\zeta^{2}$}
			(5.7,0) circle[radius=0.1cm] node[anchor=south]{$ -1$};
	\draw (5,0) node[anchor=south]{$ -\zeta$};
	\draw (0,-0.8)--(1.2,-0.8);
	\draw (-0.1,-0.8) circle[radius=0.1cm] node[anchor=south]{$ -\zeta^{3}$}
			(1.3,-0.8) circle[radius=0.1cm] node[anchor=south]{$ -1$};
	\draw (0.7,-0.8) node[anchor=south]{$ \zeta$};
	\draw (2.2,-0.8)--(3.4,-0.8);
	\draw (2.1,-0.8) circle[radius=0.1cm] node[anchor=south]{$ -\zeta^{3}$}
			(3.5,-0.8) circle[radius=0.1cm] node[anchor=south]{$ -1$};
	\draw (2.9,-0.8) node[anchor=south]{$ -\zeta^{-1}$};
\end{tikzpicture} & \raisebox{5mm}[0pt][0pt]{$\zeta \in \mathbb{G}_{12}$} & \raisebox{5mm}[0pt][0pt]{\text{no solution}} \\
\hline
\end{tabular}
\end{table}

\begin{table}[h!]\centering\footnotesize \renewcommand{\arraystretch}{1.0}
 \caption{Realization of finite-dimensional diagonal Nichols algebras of rank $3$.}\label{App3}
\vspace{1mm}

\begin{tabular}{p{5mm}|p{66mm}|p{15mm}|p{57mm}}
\hline
row & braiding & conditions & comments \\
\hline
$1'$ & \begin{tikzpicture}
	\draw (0,0)--(1.4,0) (1.6,0)--(3,0);
	\draw (-0.1,0) circle[radius=0.1cm] node[anchor=south]{$ -1$} node[anchor=north]{$ m$}
			(1.5,0) circle[radius=0.1cm] node[anchor=south]{$ -1 $}node[anchor=north]{$ m$}
			(3.1,0) circle[radius=0.1cm] node[anchor=south]{$ -1 $}node[anchor=north]{$ m$};
	\draw (0.7,0) node[anchor=south]{$ {-1}$}node[anchor=north]{$ -m$};
	\draw (2.3,0) node[anchor=south]{$ {-1}$} node[anchor=north]{$ -m$};
\end{tikzpicture} &   & \raisebox{5mm}[0pt][0pt]{\text{one solution according to $A_3$ (see $1''$)}} \\
 \cline{2-2}
 & \begin{tikzpicture}
	\draw (0,0)--(1.4,0) (1.6,0)--(3,0);
	\draw (-0.1,0) circle[radius=0.1cm] node[anchor=south]{$ -1$} node[anchor=north]{$ 1$}
			(1.5,0) circle[radius=0.1cm] node[anchor=south]{$ -1 $}node[anchor=north]{$ m$}
			(3.1,0) circle[radius=0.1cm] node[anchor=south]{$ -1 $}node[anchor=north]{$ m$};
	\draw (0.7,0) node[anchor=south]{$ {-1}$}node[anchor=north]{$ -m$};
	\draw (2.3,0) node[anchor=south]{$ {-1}$} node[anchor=north]{$ -m$};
		\draw (0.0,-1.0)--(1.4,-1.0) (1.6,-1.0)--(3.0,-1.0);
	\draw (-0.1,-1.0) circle[radius=0.1cm] node[anchor=south]{$ -1$}
	(1.5,-1.0) circle[radius=0.1cm] node[anchor=south]{$ -1$}
	(3.1,-1.0) circle[radius=0.1cm] node[anchor=south]{$ -1$};
	\draw (0.7,-1.0) node[anchor=south]{$ -1$};
	\draw (2.3,-1.0) node[anchor=south]{$ -1$};
			\draw (-0.1,-1.0) node[anchor=north]{$ 1$}
	(1.5,-1.0) node[anchor=north]{$ 1$}
	(3.1,-1.0) node[anchor=north]{$ m$};
	\draw (0.7,-1.1) node[anchor=north]{$ -2+m$};
	\draw (2.3,-1.0) node[anchor=north]{$ -m$};
\end{tikzpicture} & & \raisebox{10mm}[0pt][0pt]{\text{one solution according to $A(2,0)$ (see $4)$}}\\
 \cline{2-2}
 \end{tabular}
\end{table}

\begin{table}[h!]\centering\footnotesize \renewcommand{\arraystretch}{1.1}
\begin{tabular}{p{5mm}|p{66mm}|p{15mm}|p{57mm}}
  \cline{2-2}
   & \begin{tikzpicture}
	\draw (0,0)--(1.4,0) (1.6,0)--(3,0);
	\draw (-0.1,0) circle[radius=0.1cm] node[anchor=south]{$ -1$} node[anchor=north]{$ m$}
			(1.5,0) circle[radius=0.1cm] node[anchor=south]{$ -1 $}node[anchor=north]{$ 1$}
			(3.1,0) circle[radius=0.1cm] node[anchor=south]{$ -1 $};
	\draw (0.7,-1.0) node[anchor=south]{$ {-1}$}node[anchor=north]{$ -m$};
	\draw (2.3,-1.0) node[anchor=south]{$ {-1}$} node[anchor=north]{$ \!\!-2+m$};
	\draw (3.3,-1.0) node[anchor=north]{$\; 2-m$};
	\draw (0.0,-1.0)--(1.4,-1.0) (1.6,-1.0)--(3.0,-1.0);
	\draw (-0.1,-1.0) circle[radius=0.1cm] node[anchor=south]{$ -1$}
	(1.5,-1.0) circle[radius=0.1cm] node[anchor=south]{$ -1$}
	(3.1,-1.0) circle[radius=0.1cm] node[anchor=south]{$ -1$};
	\draw (0.7,-1.0) node[anchor=south]{$ -1$};
	\draw (2.3,-1.0) node[anchor=south]{$ {-1}$};
	\draw (0,-2)--(1.4,-2) (1.6,-2)--(3,-2);
	\draw (-0.1,-2) circle[radius=0.1cm] node[anchor=south]{$ -1$}
	(1.5,-2) circle[radius=0.1cm] node[anchor=south]{$ -1$}
	(3.1,-2) circle[radius=0.1cm] node[anchor=south]{$ -1$};
	\draw (0.7,-2) node[anchor=south]{$ {-1}$};
	\draw (2.3,-2) node[anchor=south]{$ {-1}$};
	\draw (0.0,-3.0)--(1.4,-3.0) (1.6,-3.0)--(3.0,-3.0);
	\draw (-0.1,-3.0) circle[radius=0.1cm] node[anchor=south]{$ -1$}
	(1.5,-3.0) circle[radius=0.1cm] node[anchor=south]{$ {-1}$}
	(3.1,-3.0) circle[radius=0.1cm] node[anchor=south]{$ -1$};
	\draw (0.7,-3.0) node[anchor=south]{$ -1$};
	\draw (2.3,-3.0) node[anchor=south]{$ -1$};
	\draw (-0.1,-1.0) node[anchor=north]{$ 1$}
	(1.5,-1.0) node[anchor=north]{$ 1$}
	(3.1,0.0) node[anchor=north]{$ 1$};
	\draw (0.7,-0.1) node[anchor=north]{$ m-2 $};
	\draw (2.3,0) node[anchor=north]{$ -m$};
	\draw (-0.1,-2.0) node[anchor=north]{$ 1$}
	(1.5,-2.0) node[anchor=north]{$ m$}
	(3.1,-2.0) node[anchor=north]{$ 1$};
	\draw (0.7,-2.0) node[anchor=north]{$ -m $};
	\draw (2.3,-2.0) node[anchor=north]{$ -m$};	
	\draw (-0.2,-3.0) node[anchor=north]{$ 1 $}
	(1.5,-3.0) node[anchor=north]{$ \;\;2-m $}
	(3.3,-3.0) node[anchor=north]{$ \;\;1 $};
	\draw (0.5,-3.1) node[anchor=north]{$ -2+m $};
	\draw (2.5,-3.1) node[anchor=north]{$\;\;\; -2+m $};
\end{tikzpicture} & & \raisebox{19mm}[0pt][0pt]{\text{one solution according to $A(1,1)$ (see $8$)}} \\
\hline
\end{tabular}
\end{table}

\begin{table}[h!]\centering\footnotesize \renewcommand{\arraystretch}{0.9}
\begin{tabular}{p{5mm}|p{60mm}|p{21mm}|p{57mm}}
\hline
\raisebox{34mm}[0pt][0pt]{7} & \begin{tikzpicture}
\draw (0,0)--(1.4,0) (1.6,0)--(3,0);
\draw (-0.1,0) circle[radius=0.1cm] node[anchor=south]{$ -1$}
(1.5,0) circle[radius=0.1cm] node[anchor=south]{$ q$}
(3.1,0) circle[radius=0.1cm] node[anchor=south]{$ q^3$};
\draw (0.7,0) node[anchor=south]{$ q^{-1}$};
\draw (2.3,0) node[anchor=south]{$ q^{-3}$};

\draw (1.5,-3)--(0,-5.5)--(3,-5.5)--(1.5,-3);
\draw (1.5,-2.9) circle[radius=0.1cm]
(-0.1,-5.5) circle[radius=0.1cm]
(3.1,-5.5) circle[radius=0.1cm];
\draw (1.5,-3.2) node[anchor=north]{$ -1$};
\draw (0.3,-5.1) node[anchor=north]{$ q$};
\draw (2.6,-5.1) node[anchor=north]{$ -1$};

\draw (0.5,-4.5) node[anchor=west]{$ \;q^{-1}$};
\draw (2.4,-4.5) node[anchor=east]{$ q^{3}$};
\draw (1.6,-5.5) node[anchor=south]{$ q^{-2}$};

\draw (1.5,-2.9) node[anchor=south]{$ 1$}
(-0.1,-5.5) node[anchor=east]{$ m$}
(3.1,-5.5) node[anchor=west]{$ 1$};
\draw (0.6,-4.3) node[anchor=east]{$ -m$};
\draw (2.3,-4.3) node[anchor=west]{$ 3m-2$};
\draw (1.5,-5.5) node[anchor=north]{$ -2m$};	

\draw (0,-1.0)--(1.4,-1.0) (1.6,-1.0)--(3,-1.0);
\draw (-0.1,-1.0) circle[radius=0.1cm] node[anchor=south]{$ -1$}
(1.5,-1.0) circle[radius=0.1cm] node[anchor=south]{$ -1$}
(3.1,-1.0) circle[radius=0.1cm] node[anchor=south]{$ q^{3}$};
\draw (0.7,-1.0) node[anchor=south]{$ q$};
\draw (2.3,-1.0) node[anchor=south]{$ q^{-3}$};

\draw (0,-2.0)--(1.4,-2.0) (1.6,-2.0)--(3,-2.0);
\draw (-0.1,-2.0) circle[radius=0.1cm] node[anchor=south]{$ q^{3}$}
(1.5,-2.0) circle[radius=0.1cm] node[anchor=south]{$ -1$}
(3.1,-2.0) circle[radius=0.1cm] node[anchor=south]{$ -q^{-1}$};
\draw (0.7,-2.0) node[anchor=south]{$ q^{-3}$};
\draw (2.3,-2.0) node[anchor=south]{$ q^{2}$};

\draw (-0.1,0) node[anchor=north]{$ 1$}
(1.5,0) node[anchor=north]{$ m$}
(3.1,0) node[anchor=north]{$ 3m$};
\draw (0.7,0) node[anchor=north]{$ -m $};
\draw (2.3,0) node[anchor=north]{$ -3m$};

\draw (-0.1,-1.0) node[anchor=north]{$ 1$}
(1.5,-1.0) node[anchor=north]{$ 1$}
(3.1,-1.0) node[anchor=north]{$ 3m$};
\draw (0.7,-1.0) node[anchor=north]{$ -2+m $};
\draw (2.3,-1.0) node[anchor=north]{$ -3m$};	

\draw (-0.1,-2.0) node[anchor=north]{$ 3m$}
(1.5,-2.0) node[anchor=north]{$ 1$}
(3.5,-2.0) node[anchor=north]{$ 1-m$};
\draw (0.7,-2.0) node[anchor=north]{$ -3m$};
\draw (2.3,-2.0) node[anchor=north]{$ -2+2m$};
\end{tikzpicture}	& \raisebox{34mm}[0pt][0pt]{$\begin{array}{@{}l@{}}q\neq \pm 1,\\ q \not \in \mathbb{G}_3\end{array}$} &
\raisebox{34mm}[0pt][0pt]{\text{super Lie, $G(3)$}} \\
\hline
\end{tabular}
\end{table}

\begin{table}[h!]\centering\footnotesize \renewcommand{\arraystretch}{1.0}
\begin{tabular}{p{5mm}|p{81mm}|p{15mm}|p{42mm}}
\hline
\raisebox{10mm}[0pt][0pt]{14} & \begin{tikzpicture}
\draw (0,0)--(1.4,0) (1.6,0)--(3,0);
\draw (-0.1,0) circle[radius=0.1cm] node[anchor=south]{$ -1$}
(1.5,0) circle[radius=0.1cm] node[anchor=south]{$ -\zeta^{-1}$}
(3.1,0) circle[radius=0.1cm] node[anchor=south]{$ \zeta $};
\draw (0.7,0) node[anchor=south]{$-\zeta $};
\draw (2.3,0) node[anchor=south]{$ -\zeta $};

\draw (4.2,0)--(5.6,0) (5.8,0)--(7.2,0);
\draw (4.1,0) circle[radius=0.1cm] node[anchor=south]{$ -1$} node[anchor=north]{$ $}
(5.7,0) circle[radius=0.1cm] node[anchor=south]{$ -1 $}node[anchor=north]{$ $}
(7.3,0) circle[radius=0.1cm] node[anchor=south]{$ \zeta $}node[anchor=north]{$ $};
\draw (4.9,0) node[anchor=south]{$ -\zeta^{-1}$}node[anchor=north]{$ $};
\draw (6.5,0) node[anchor=south]{$ -\zeta$} node[anchor=north]{$ $};	

\draw (0,-1.0)--(1.4,-1.0) (1.6,-1.0)--(3,-1.0);
\draw (-0.1,-1.0) circle[radius=0.1cm] node[anchor=south]{$ -\zeta^{-1}$}
(1.5,-1.0) circle[radius=0.1cm] node[anchor=south]{$ -1$}
(3.1,-1.0) circle[radius=0.1cm] node[anchor=south]{$ \zeta^{-1} $};
\draw (0.7,-1.0) node[anchor=south]{$ -\zeta $};
\draw (2.3,-1.0) node[anchor=south]{$ -\zeta^{-1} $};
\end{tikzpicture}	& \raisebox{10mm}[0pt][0pt]{$\zeta \in \mathbb{G}_3$} & \raisebox{10mm}[0pt][0pt]{\text{no solution}} \\
\hline
\end{tabular}
\end{table}

\subsection*{Acknowledgements}
IF and SL are partially supported by the RTG 1670 ``Mathematics inspired by String theory and Quantum Field Theory''. Many thanks to Christian Reiher for suggesting the analytic continuation by partial integration in Section~\ref{sec_recursion}, to Ivan Angiono for explaining to us Proposition~\ref{prop_IvanCartanRoot}, to Sven Ole Warnaar for answering questions on the $\g$-Selberg integral formula, and to the anonymous referees for many improving comments on the manuscript.

\LastPageEnding

\end{document}